\documentclass{article}

\usepackage{amsmath, amssymb, amsthm}
\usepackage{mathtools}
\usepackage{mathrsfs}
\usepackage{bm}
\usepackage{extarrows}
\usepackage{geometry}
\usepackage{float}
\usepackage{indentfirst}
\usepackage{anyfontsize}
\usepackage{booktabs,multirow}
\usepackage{multicol}
\usepackage{enumitem}
\usepackage{appendix}

\usepackage[dvipsnames]{xcolor}
\usepackage{graphicx}
\graphicspath{
    {./figure/}{./figures/}{./image/}{./images/}{./graphic/}
    {./graphics/}{./picture/}{./pictures/}
}

\usepackage{subcaption}
\usepackage{epstopdf}

\usepackage[ruled, lined, linesnumbered]{algorithm2e}

\numberwithin{equation}{section}
\numberwithin{figure}{section}
\numberwithin{table}{section}

\usepackage{thmtools}
\declaretheorem[style=plain, name=Theorem, numbered=yes, numberwithin=section]{theorem}

\declaretheorem[style=plain, name=Corollary, numbered=yes, sibling=theorem]{corollary}
\declaretheorem[style=plain, name=Lemma, numbered=yes, numberwithin=section]{lemma}
\declaretheorem[style=definition, name=Definition, numbered=yes, numberwithin=section]{definition}
\declaretheorem[style=remark, name=Remark, numbered=yes, numberwithin=section]{remark}

\usepackage{hyperref}
\hypersetup{colorlinks=true,linkcolor=,citecolor=blue}

\newenvironment{keywords}{\par\vskip 2ex\textbf{Keywords:}\enspace}{\par}

\DeclarePairedDelimiter{\abs}{\lvert}{\rvert}
\DeclarePairedDelimiter{\inner}{\langle}{\rangle}
\DeclarePairedDelimiter{\norm}{\lVert}{\rVert}

\DeclareMathOperator{\tr}{Tr}
\DeclareMathOperator{\qr}{qr}

\usepackage{lipsum}

\title{Temporal-Stability-Enhanced and Energy-Stable Dynamical Low-Rank Approximation for Multiscale Linear Kinetic Transport Equations}

\author{
    Shun Li\thanks{School of Mathematical Sciences, University of Science and Technology of China. Email: {\tt lishun@mail.ustc.edu.cn}.} \and
    Yan Jiang\thanks{School of Mathematical Sciences, University of Science and Technology of China. Email: {\tt jiangy@ustc.edu.cn}.} \and
    Mengping Zhang\thanks{School of Mathematical Sciences, University of Science and Technology of China. Email: {\tt mpzhang@ustc.edu.cn}.}
    \and
    Tao Xiong\thanks{School of Mathematical Sciences, University of Science and Technology of China. Email: {\tt taoxiong@ustc.edu.cn}.}
}

\date{}

\begin{document}

\maketitle

\begin{abstract}
    In this paper, we develop an asymptotic-preserving dynamical low-rank method for the multiscale linear kinetic transport equation.
    The proposed scheme is unconditionally stable in the diffusive regime while preserving the correct asymptotic behavior, and can achieve significant reductions in computational cost through a low-rank representation and large time step stability.
    A low-rank formulation consistent with the discrete energy is introduced under the discrete ordinates discretization, and energy stability of the resulting scheme is established.
    Numerical experiments confirm the energy stability and demonstrate that the method is efficient while maintaining accuracy across different regimes and capturing the correct asymptotic limits.
\end{abstract}

\begin{keywords}
    kinetic transport equation; dynamical low-rank approximation; macro-micro decomposition; energy stability; asymptotic preserving; implicit-explicit
\end{keywords}

\section{Introduction} 

Kinetic transport equation is a fundamental model in fields like nuclear engineering, yet its high-dimensional phase space poses significant computational challenges.
This complexity is compounded by the multi-scale nature of the system, which formally converges to a macroscopic equation as the Knudsen number $\varepsilon \to 0$.
Standard numerical discretizations of the kinetic transport equation often suffer from severe stability restrictions when $\varepsilon$ is small.
Asymptotic-preserving (AP) schemes were introduced to overcome this difficulty by ensuring that the discrete scheme automatically captures the correct macroscopic limit as $\varepsilon \to 0$, without resolving the small-scale dynamics explicitly \cite{jin1999efficient}.
Significant progress has been made in the development of AP methods for multiscale hyperbolic and kinetic equations.
For an overview of AP schemes, we refer the readers to \cite{jin2010asymptotic, hu2017asymptotic, jin2022asymptotic}.
While AP schemes effectively bridge the kinetic and fluid regimes without severe stability constraints, such as those based on a macro-micro decomposition, e.g. \cite{liu2004boltzmann, lemou2008new}, high-dimensionality remains a primary computational bottleneck.

In recent years, low-rank approximation has attracted increasing interest in numerical analysis as a means to mitigate the curse of dimensionality, enabling efficient approximations of solutions to high-dimensional partial differential equations \cite{hackbusch2012tensor, grasedyck2013literature, khoromskij2018tensor, bachmayr2023low,sands2025high}.
A prominent method for time-dependent problems is the dynamical low-rank approximation (DLRA) \cite{koch2007dynamical}, which evolves a low-rank solution by projecting the governing equation onto the tangent space of a fixed-rank manifold. This allows the solution to be approximated by a small number of time-dependent basis functions, substantially reducing computational and memory costs compared with full-rank discretizations.
DLRA has been successfully applied to various high-dimensional kinetic and quantum problems, and recent developments include adaptively updating the basis or allowing rank adaptivity in order to improve accuracy and stability \cite{ceruti2022unconventional, ceruti2022rank, ceruti2024robust, ceruti2024parallel, appelo2025robust,einkemmer2025construction}.
As an alternative approach, step-and-truncation (SAT) method builds a low-rank solution from a traditional full-rank discretization, which updates the solution and then truncates small singular modes at each time step \cite{guo2024conservative,guo2024local,sands2025high}.
In the context of high-dimensional kinetic equations, microscopic fluctuations around a macroscopic quantity often exhibit intrinsic low-rank structures, particularly near equilibrium \cite{einkemmer2025review}.
Low-rank methods, including DLRA and SAT approaches, allow the solution to be represented by a small number of basis functions, can significantly reduce computational and memory costs compared with full-rank discretizations, and have been successfully applied to a variety of kinetic systems \cite{einkemmer2018lowrank, peng2020low, peng2021highorder, einkemmer2021mass, ding2021dynamical, einkemmer2021efficient, coughlin2024robust}. For discussions of low-rank methods for kinetic equations and recent developments, we direct the reader to the review \cite{einkemmer2025review}.

For multiscale kinetic equations, several AP DLRA methods have been proposed specifically for linear kinetic transport equations, which generally combines the macro–micro decomposition with low-rank evolution of the microscopic component \cite{einkemmer2021asymptoticpreserving,einkemmer2024asymptoticpreserving,patwardhan2025parallel}.
In these approaches, an implicit–explicit (IMEX) time discretization is typically adopted, where the stiff collision operator is treated implicitly, while the coupling terms involving the macroscopic density are handled explicitly in the microscopic equation.
Such formulations are formally AP and can be implemented in an explicit way without solving any linear systems. However, in the diffusion limit the resulting scheme of the diffusion type equation for the macroscopic density is usually discretized explicitly in time, leading to a parabolic time step restriction of $\mathcal{O}(\Delta x^2)$ \cite{einkemmer2024asymptoticpreserving}, where $\Delta x$ is the mesh size.
Such a stability bottleneck originates from the explicit macro–micro coupling at the discrete level, rather than from the collision operator itself, motivating the development of AP DLRA schemes which can also be efficient with large time step stability.

In this work, we aim to develop an AP DLRA scheme for the linear kinetic transport equation, which integrates an IMEX scheme with the Schur complement strategy  \cite{peng2021asymptotic} and the basis-update and Galerkin (BUG) or augmented BUG (aBUG) integrator \cite{ceruti2022rank}.
By applying the Schur complement \cite{zhang2006schur} to the macro–micro system, the microscopic variable can be algebraically eliminated to yield a closed system for the macroscopic density  which is unconditionally stable in the diffusive limit \cite{peng2021asymptotic}.
This approach still allows the microscopic component to evolve on a low-rank manifold, but without the restrictive parabolic CFL condition or a need for artificial weighting functions. Another new ingredient is that, we combine the DLRA scheme with the discrete ordinates (also known as $S_N$) discretization, different from \cite{einkemmer2021asymptoticpreserving,einkemmer2024asymptoticpreserving}.
By embedding a quadrature-based energy structure directly into the low-rank ansatz, we rigorously prove that our scheme inherits the energy-dissipative properties of the underlying full-rank discretization. As a result, our framework effectively reduces the computational complexity of high-dimensional kinetic dynamics both in space and time discretizations, while ensuring structural stability and capturing correct diffusive limits.

For the remainder,
Section~\ref{sec:background} briefly reviews the linear kinetic transport equation and the macro–micro decomposition, as well as a dynamical low-rank approximation.
Section~\ref{sec:numerical-method-fullrank} presents an AP full-rank discretization that employs $S_N$ and an IMEX time-stepping, while Section~\ref{sec:numerical-method-lowrank} devotes to low-rank approximation and rigorous proof of energy stability.
Some numerical experiments are presented in Section~\ref{sec:numerical-results}, and concluding remarks are provided in Section~\ref{sec:conclusion}.

\section{Background} \label{sec:background}

\subsection{Kinetic transport equation}
We consider the multiscale linear kinetic transport equation:
\begin{equation}\label{eq:rte-2d}
    \partial_t f + \frac{1}{\varepsilon} \bm{\Omega} \cdot\nabla_{\bm{x}} f = \frac{\sigma^s}{\varepsilon^2}\Big(\inner*{f}_{\bm{\Omega}} - f\Big) - \sigma^a f + \Phi.
\end{equation}
Here, $f = f(t,\bm{x},\bm{\Omega})$ denotes the particle densit at time $t \in \mathbb{R}^+$, position $\bm{x} = (x,y) \in D_{\bm{x}} \subset \mathbb{R}^2$ and angular direction $\bm{\Omega} = (\Omega^x,\Omega^y,\Omega^z) \in D_{\bm{\Omega}} = \mathbb{S}^2$. The source term is denoted by $\Phi = \Phi(t,\bm{x})$.
The average over the angular domain is defined as
\begin{equation} \label{eq:omega-ave}
    \inner*{f}_{\bm{\Omega}} = \frac{1}{|D_{\bm{\Omega}}|} \int_{\mathbb{S}^2} f\,\mathrm{d}\,\bm{\Omega} = \frac{1}{4\pi} \int_{\mathbb{S}^2} f\,\mathrm{d}\,\bm{\Omega}.
\end{equation}
The scalar flux is given by $\rho = \inner*{f}_{\bm{\Omega}}$.
The scattering and absorption coefficients satisfy $\sigma^s(\bm{x}) \geq \sigma^s_0 > 0$ and $\sigma^a(\bm{x}) \geq 0$, respectively.
The Knudsen number $\varepsilon > 0$ denotes the ratio of the mean free path of particles to the characteristic length.

In \cite{lemou2008new}, by following \cite{liu2004boltzmann} a macro-micro decomposition which separates $f$ into an equilibrium and non-equilibrium components is proposed, that is
\[
    f(t,\bm{x},\bm{\Omega}) = \rho(t,\bm{x}) + \varepsilon\,g(t,\bm{x},\bm{\Omega}),
\]
where the macroscopic density $\rho = \inner*{f}_{\bm{\Omega}}$ and the microscopic component $g$ satisfies $\inner*{g}_{\bm{\Omega}} = 0$.
Applying $\inner*{\,\cdot\,}_{\bm{\Omega}}$ and its orthogonal complement to \eqref{eq:rte-2d} yields the following macro-micro system:
\begin{equation}\label{eq:macro-micro-2d}
    \begin{aligned}
        \partial_t \rho + \nabla_{\bm{x}} \cdot \inner*{\bm{\Omega} g}_{\bm{\Omega}}                                                                                                               & = - \sigma^a \rho + \Phi,
        \\
        \varepsilon \partial_t g + \left(\mathbf{I}-\inner*{\,\cdot\,}_{\bm{\Omega}}\right) (\nabla_{\bm{x}} \cdot(\bm{\Omega} g)) + \frac{1}{\varepsilon} \bm{\Omega} \cdot \,\nabla_{\bm{x}}\rho & = - \frac{\sigma^s}{\varepsilon} g - \varepsilon \sigma^a g.
    \end{aligned}
\end{equation}
In the limit $\varepsilon \to 0$, the second equation of the
system \eqref{eq:macro-micro-2d} formally reduces to $\sigma^s g = - \bm{\Omega} \cdot \,\nabla_{\bm{x}}\rho$. Substituting it into the first equation, it leads to the diffusion equation for the macroscopic density $\rho$
\begin{equation}
\label{eq:diffusion}
    \partial_t \rho - \nabla_{\bm{x}} \cdot \left( \inner*{\bm{\Omega} \otimes \bm{\Omega}}_{\bm{\Omega}} \frac{1}{\sigma^s} \nabla_{\bm{x}} \rho
    \right) = - \sigma^a \rho + \Phi, \quad \inner*{\bm{\Omega} \otimes \bm{\Omega}}_{\bm{\Omega}} = \frac{\mathbf{I}}3.
\end{equation}
We refer to \cite{lemou2008new} for more details.

\subsection{Dynamic Low-Rank Approximation}

Here we briefly review the DLRA framework
for time-dependent matrix-valued problems,
which forms the basis of the proposed low-rank discretization of the kinetic equation.
The key idea is to approximate the solution by evolving it on the manifold of low-rank matrices, where the dynamics is obtained by projecting the original equation onto the tangent space of the manifold \cite{koch2007dynamical}.

Let $A(t) \in \mathbb{R}^{m_1 \times m_2}$ satisfy the matrix differential equation
\begin{equation*}
    \partial_t A(t) = \mathcal{F}\bigl(t,A(t)\bigr).
\end{equation*}
Denote by $\mathcal{M}_r$ the manifold of matrices of rank $r$.
A low-rank approximation $Y(t) \in \mathcal{M}_r$ admits the factorization
\begin{equation*}
    A(t) \approx Y(t) = X(t)\,S(t)\,V(t)^\top,
\end{equation*}
where $X(t) \in \mathbb{R}^{m_1 \times r}$, $V(t) \in \mathbb{R}^{m_2 \times r}$, $S(t) \in \mathbb{R}^{r \times r}$.
Both $X(t)$ and $V(t)$ have orthonormal columns.
Let $\mathcal{T}_{Y}(\mathcal{M}_r)$ denote the tangent space of $\mathcal{M}_r$
at $Y \in \mathcal{M}_r$.
The DLRA evolution is defined by the projected equation
\begin{equation}\label{eq:DLRA}
    \partial_t Y(t)
    = \mathcal{P}_{Y(t)}\,
    \mathcal{F}\bigl(t,Y(t)\bigr),
\end{equation}
where $\mathcal{P}_{Y}$ denotes the orthogonal projection onto
$\mathcal{T}_{Y}(\mathcal{M}_r)$.
Let the singular value decomposition (SVD) of $Y \in \mathcal{M}_r$ be
$Y = X S V^\top$. The orthogonal projection $\mathcal{P}_{Y}$ is explicitly given by \cite[Lemma~4.1]{koch2007dynamical}
\begin{equation*}
    \mathcal{P}_{Y}(Z)
    =
    X X^\top Z
    + Z V V^\top
    - X X^\top Z V V^\top,
    \qquad \forall\, Z \in \mathbb{R}^{m_1 \times m_2}.
\end{equation*}

A practical time discretization of \eqref{eq:DLRA} is provided by the BUG integrator \cite{koch2007dynamical},
which advances the low-rank factors via a splitting strategy and avoids explicit projections.
The BUG method preserves a fixed rank $r$ throughout the simulation.
In contrast, the aBUG integrator \cite{ceruti2022rank} enables adaptive rank selection by performing SVD truncation at each step according to a specified relative tolerance $\tau > 0$. For other time integrators, we refer to the recent review \cite{einkemmer2025review}.

\section{Full-rank Method} \label{sec:numerical-method-fullrank}

\subsection{Temporal Discretization}

We begin with a first order IMEX temporal discretization applied to the macro-micro formulation of the linear kinetic transport equation~\eqref{eq:macro-micro-2d}, where the macroscopic density is treated explicitly in the microscopic equation \cite{lemou2008new}
\begin{subequations}
    \label{eq:macro-micro-imex}
    \begin{align}
        \frac{\rho^{n+1} - \rho^n}{\Delta t} + \nabla_{\bm{x}} \cdot \inner*{\bm{\Omega} g^{n+1}}_{\bm{\Omega}}                                                                                                                          & = - \sigma^a \rho^{n+1} + \Phi^{n+1},       \label{eq:macro-micro-imex-1}                              \\
        \varepsilon\,\frac{g^{n+1} - g^n}{\Delta t} + \left(\mathbf{I}-\inner*{\,\cdot\,}_{\bm{\Omega}}\right) \left(\nabla_{\bm{x}} \cdot(\bm{\Omega} g^n)\right) + \frac{1}{\varepsilon} \bm{\Omega} \cdot \,\nabla_{\bm{x}}\,\rho^{n} & = - \frac{\sigma^s}{\varepsilon} g^{n+1} - \varepsilon \sigma^a g^{n+1}. \label{eq:macro-micro-imex-2}
    \end{align}
\end{subequations}
With this choice, the microscopic variable $g$ can be updated explicitly once the macroscopic density $\rho$ is given at previous time level $t_n$, resulting in a simple and decoupled time-stepping procedure.
However, this explicit treatment of the macroscopic density in the diffusive regime as $\varepsilon \to 0$,
leads to an explicit discretization of the diffusion equation and still imposes a strict parabolic time step restriction of the form $\Delta t = \mathcal{O}(\Delta x^2)$ \cite{liu2010analysis}.
Instead, a modified strategy in which the macroscopic density is treated implicitly was proposed \cite{peng2021asymptotic}
\begin{align*}
    \frac{\rho^{n+1} - \rho^n}{\Delta t} + \nabla_{\bm{x}} \cdot \inner*{\bm{\Omega} g^{n+1}}_{\bm{\Omega}}                                                                                                                            & = - \sigma^a \rho^{n+1} + \Phi^{n+1},                                    \\
    \varepsilon\,\frac{g^{n+1} - g^n}{\Delta t} + \left(\mathbf{I}-\inner*{\,\cdot\,}_{\bm{\Omega}}\right) \left(\nabla_{\bm{x}} \cdot(\bm{\Omega} g^n)\right) + \frac{1}{\varepsilon} \bm{\Omega} \cdot \,\nabla_{\bm{x}}\,\rho^{n+1} & = - \frac{\sigma^s}{\varepsilon} g^{n+1} - \varepsilon \sigma^a g^{n+1}.
\end{align*}
This choice ensures that, in the diffusive limit, the scheme degenerates into an implicit discretization of the limiting diffusion equation,
thereby allowing for much larger time steps independent of the spatial resolution.
The implicit discretization of the macroscopic equation results in a coupled linear system:
\begin{align*}
    \begin{bmatrix}
        \left(\frac{1}{\Delta t} + \sigma^a \right) \mathbf{I}                 & \nabla_{\bm{x}} \cdot \inner*{\bm{\Omega} \,\cdot\,}_{\bm{\Omega}}
        \\
        \frac{1}{\varepsilon^2} \bm{\Omega} \cdot \,\nabla_{\bm{x}}(\,\cdot\,) & \left(\frac{1}{\Delta t} + \frac{\sigma^s}{\varepsilon^2} + \sigma^a\right) \mathbf{I}
    \end{bmatrix}
    \begin{bmatrix}
        \rho^{n+1} \\ g^{n+1}
    \end{bmatrix}
     & =
    \begin{bmatrix}
        \frac{1}{\Delta t} \rho^n + \Phi^{n+1} \\
        \frac{1}{\Delta t} g^n - \frac{1}{\varepsilon} \left(\mathbf{I}-\inner*{\,\cdot\,}_{\bm{\Omega}}\right) \left(\nabla_{\bm{x}} \cdot(\bm{\Omega} g^n)\right)
    \end{bmatrix},
\end{align*}
where $\mathbf{I}$ is the identity operator.
A Schur complement is then used to decouple the variables, as proposed in \cite{peng2021asymptotic}.
The system can be written in a block form
\begin{align*}
    \begin{bmatrix}
        A_{11} & A_{12}
        \\
        A_{21} & A_{22}
    \end{bmatrix}
    \begin{bmatrix}
        \rho^{n+1} \\ g^{n+1}
    \end{bmatrix}
     & =
    \begin{bmatrix}
        b_1 \\ b_2
    \end{bmatrix},
\end{align*}
where $A_{22}$ corresponds to the operator $\left(\frac{1}{\Delta t} + \frac{\sigma^s}{\varepsilon^2} + \sigma^a\right) \mathbf{I}$  and is invertible for all $\varepsilon>0$.
The block factorization
\[
    A =
    \begin{bmatrix}
        A_{11} & A_{12} \\
        A_{21} & A_{22}
    \end{bmatrix}
    =
    \begin{bmatrix}
        I & A_{12} A_{22}^{-1} \\
          & I
    \end{bmatrix}
    \begin{bmatrix}
        T &        \\
          & A_{22}
    \end{bmatrix}
    \begin{bmatrix}
        I                  &   \\
        A_{22}^{-1} A_{21} & I
    \end{bmatrix}
\]
identifies the Schur complement $T := A_{11} - A_{12} A_{22}^{-1} A_{21}$ \cite{zhang2006schur}, which defines a closed system for the macroscopic density
\begin{equation}\label{eq:schur-rho}
    T \rho^{n+1} = b_1 -A_{12} A_{22}^{-1} b_2.
\end{equation}
Once $\rho^{n+1}$ is obtained from \eqref{eq:schur-rho}, $g^{n+1}$ can be recovered from
\begin{equation}\label{eq:schur-g}
    g^{n+1} = A_{22}^{-1} (b_2 - A_{21} \rho^{n+1}).
\end{equation}
Substituting the explicit expressions of the operators into \eqref{eq:schur-rho} and \eqref{eq:schur-g} yields
\begin{subequations}
    \label{eq:macro-micro-imex-s-schur}
    \begin{align}
         & \left\{
        \left(\frac{1}{\Delta t} + \sigma^a \right) \mathbf{I} - \nabla_{\bm{x}} \cdot \inner*{\bm{\Omega} \,\cdot\,}_{\bm{\Omega}} R
        \left(\frac{1}{\varepsilon^2} \bm{\Omega} \cdot \,\nabla_{\bm{x}}(\,\cdot\,)\right)
        \right\} \rho^{n+1}
        \notag                                              \\
         & \quad ={} \frac{1}{\Delta t} \rho^n + \Phi^{n+1}
        - \nabla_{\bm{x}} \cdot \inner*{\bm{\Omega} \,\cdot\,}_{\bm{\Omega}}
        R
        \left(
        \frac{1}{\Delta t} g^n - \frac{1}{\varepsilon} \left(\mathbf{I}-\inner*{\,\cdot\,}_{\bm{\Omega}}\right) \left(\nabla_{\bm{x}} \cdot(\bm{\Omega} g^n)\right)
        \right),
        \label{eq:macro-micro-imex-s-schur-1}               \\
         & g^{n+1} ={} R
        \left(
        \frac{1}{\Delta t} g^n - \frac{1}{\varepsilon} \left(\mathbf{I}-\inner*{\,\cdot\,}_{\bm{\Omega}}\right) \left(\nabla_{\bm{x}} \cdot(\bm{\Omega} g^n)\right)
        - \frac{1}{\varepsilon^2} \bm{\Omega} \cdot \,\nabla_{\bm{x}}(\rho^{n+1})
        \right). \label{eq:macro-micro-imex-s-schur-2}
    \end{align}
\end{subequations}
where $R := \left(\frac{1}{\Delta t} + \frac{\sigma^s}{\varepsilon^2} + \sigma^a\right)^{-1}$.

By employing the Schur complement for the solution procedure, we refer to the temporal discretization \eqref{eq:macro-micro-imex-s-schur} described above as the IMEX-S scheme, while the original scheme \eqref{eq:macro-micro-imex} is denoted as the IMEX scheme, for the sake of distinction.
It is worth noting that the linear system involved in the IMEX-S scheme is of modest size and usually diagonally dominant, with coefficient matrices that are independent of time.
As a result, the per-step computational cost of IMEX-S scheme is only slightly higher than that of the IMEX scheme.
The main advantage of IMEX-S scheme lies in its temporal enhanced stability, particularly in the diffusive regime $\varepsilon \to 0$ \cite{peng2021asymptotic}, where the IMEX scheme suffers from strict time step restrictions, which is important for long time simulations.

\subsection{Angular and Spatial Discretization}

We employ the $S_N$ method for angular discretization using the Chebyshev-Legendre (CL) product quadrature rule \cite{lewis1983computational}.
For two-dimensional geometries, the angular variable is discretized into a set of $N_{\bm{\Omega}} = 2N^2$ directions and weights $\{\bm{\Omega}_m, w_m\}_{m=1}^{N_\Omega}$, where $\bm{\Omega}_m = (\Omega_m^x, \Omega_m^y)$ represents the projection of the 3D unit direction onto the $xy$-plane.
Accordingly, the normalization constant $|D_{\bm{\Omega}}|$ in \eqref{eq:omega-ave} is set to $2\pi$ to account for this reduced phase space.

For the spatial discretization, we adopt a first-order upwind scheme on staggered grids following  \cite{kupper2016asymptotic,einkemmer2021asymptoticpreserving}.
We restrict the presentation of the discretization to two spatial dimensions on a rectangular domain $D_{\bm{x}} = [x_L, x_R] \times [y_B, y_T]$.
The spatial domain is discretized by a uniform staggered grid with mesh sizes $\Delta x = \frac{x_R - x_L}{N_x}$ and $\Delta y = \frac{y_T - y_B}{N_y}$ where $N_x, N_y \in \mathbb{N}$.
The cell interfaces are located at $\{x_{i+\frac12}\}_{i=0}^{N_x}$ and $\{y_{j+\frac12}\}_{j=0}^{N_y}$, while the cell centers are given by $\{x_i\}_{i=1}^{N_x}$ and $\{y_j\}_{j=1}^{N_y}$.
The discretization is arranged such that $\rho$ is defined at $(x_i,y_j)$ and $(x_{i+\frac12},y_{j+\frac12})$, whereas $\bm{g}$ is defined at $(x_{i+\frac12},y_j)$ and $(x_{i},y_{j+\frac12})$ (see Figure~\ref{fig:mesh-2d}).

\begin{figure}[htbp]
    \centering
    \includegraphics[width=0.3\textwidth]{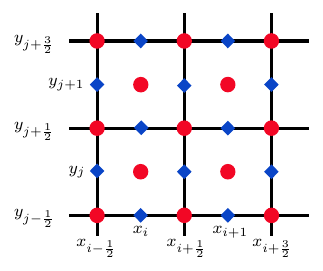}
    \caption{Two-dimensional staggered grid \cite{einkemmer2021asymptoticpreserving}. Red: $\rho$; Blue: $\bm{g}$.} \label{fig:mesh-2d}
\end{figure}

We introduce bijective index maps $B_{\rho}$, and $B_{\bm{g}}$ to reorder the density grids, and the microscopic component grids into linear indices, respectively.
The macroscopic density and microscopic correction can be represented in vector–matrix form, defined by
\[
    \bm{\rho} = \begin{bmatrix}
        \rho_1 \\ \vdots \\ \rho_{N_{\rho}}
    \end{bmatrix} \in \mathbb{R}^{N_{\rho}},
    \quad
    \bm{G}
    =
    \begin{bmatrix}
        g_{1,1}          & \cdots & g_{1,N_{\bm{\Omega}}}          \\
        \vdots           & \ddots & \vdots                         \\
        g_{N_{\bm{g}},1} & \cdots & g_{N_{\bm{g}},N_{\bm{\Omega}}} \\
    \end{bmatrix}
    \in \mathbb{R}^{N_{\bm{g}} \times N_{\bm{\Omega}}}.
\]
where $\rho_{j} = \rho(t,\bm{x}_{B_{\rho}^{-1}(j)})$, $g_{j,\ell} = g(t,\bm{x}_{B_{\bm{g}}^{-1}(j)},\bm{\Omega}_{\ell})$, and $N_{\rho} = N_{\bm{g}} = 2N_x N_y$ for periodic boundary conditions.
Let $\bm{Q}^s = \text{diag}(\Omega^s_j)$ and $|\bm{Q}^s| = \text{diag}(|\Omega^s_j|)$ for $j=1,\dots,N_{\bm{\Omega}}$ and $s \in \{x,y\}$.
We define the upwind matrices $\bm{Q}^{s,\pm} = \frac{1}{2}(\bm{Q}^s \pm |\bm{Q}^s|)$ and introduce the discrete advection operator $\mathcal{A}$, defined by
\begin{equation}
    \label{eq:operatorA}
    \mathcal{A}(\bm{G}^n) =
    \mathcal{D}^{x,-} \bm{G}^n \bm{Q}^{x,+} + \mathcal{D}^{x,+} \bm{G}^n \bm{Q}^{x,-} + \mathcal{D}^{y,-} \bm{G}^n \bm{Q}^{y,+} + \mathcal{D}^{y,+} \bm{G}^n \bm{Q}^{y,-},
\end{equation}
where $\mathcal{D}^{x,\pm}$ and $\mathcal{D}^{y,\pm}$ are the first-order upwind/downwind finite-difference operators along the $x$- and $y$- directions, respectively.
We further define the operators
\begin{subequations}
    \label{eq:operators}
    \begin{align}
        \mathcal{H}(\bm{G}^{n+1}) ={}    & \frac1{|D_{\bm{\Omega}}|} \Big(\mathcal{D}^{x,-} \bm{G}^{n+1} \bm{Q}^x \bm{w}
        + \mathcal{D}^{y,-} \bm{G}^{n+1} \bm{Q}^y \bm{w} \Big), \label{eq:operatorH}
        \\
        \mathcal{J}(\bm{\rho}^{n+1}) ={} &
        \mathcal{D}^{x,+} \bm{\rho}^{n+1} \bm{1}^\top \bm{Q}^x
        + \mathcal{D}^{y,+} \bm{\rho}^{n+1} \bm{1}^\top \bm{Q}^y, \label{eq:operatorJ}
    \end{align}
\end{subequations}
where $\bm{w} = \begin{bmatrix} w_1, \dots, w_{N_{\bm{\Omega}}} \end{bmatrix}^\top$.
With these definitions, a fully discrete IMEX scheme reads
{\small\begin{equation} \label{eq:IMEX-2d}
    \left\{
    \begin{aligned}
         & \frac{\bm{\rho}^{n+1} - \bm{\rho}^n}{\Delta t} + \mathcal{H}(\bm{G}^{n+1})
        = -  \bm{\sigma}^a \bm{\rho}^{n+1} + \bm{\Phi}^{n+1},
        \\
         & \frac{\bm{G}^{n+1} - \bm{G}^n}{\Delta t} + \frac{1}{\varepsilon}  \mathcal{A}(\bm{G}^n) \left(\mathbf{I} - \frac1{|D_{\bm{\Omega}}|} \bm{w} \bm{1}^\top \right) =
        -\frac{1}{\varepsilon^2} \mathcal{J}(\bm{\rho}^{n})
        - \frac{1}{\varepsilon^2} \bm{\sigma}^s \bm{G}^{n+1} - \bm{\sigma}^a \bm{G}^{n+1}.
    \end{aligned}
    \right.
\end{equation}}
Similarly, a fully discrete IMEX-S scheme is given by
{\small\begin{equation} \label{eq:IMEX-S-2d}
    \left\{
    \begin{aligned}
         & \frac{\bm{\rho}^{n+1} - \bm{\rho}^n}{\Delta t} + \mathcal{H}(\bm{G}^{n+1})
        = -  \bm{\sigma}^a \bm{\rho}^{n+1} + \bm{\Phi}^{n+1},
        \\
         & \frac{\bm{G}^{n+1} - \bm{G}^n}{\Delta t} + \frac{1}{\varepsilon}  \mathcal{A}(\bm{G}^n) \left(\mathbf{I} - \frac1{|D_{\bm{\Omega}}|} \bm{w} \bm{1}^\top \right) =
        -\frac{1}{\varepsilon^2} \mathcal{J}(\bm{\rho}^{n+1})
        - \frac{1}{\varepsilon^2} \bm{\sigma}^s \bm{G}^{n+1} - \bm{\sigma}^a \bm{G}^{n+1}.
    \end{aligned}
    \right.
\end{equation}}

To clearly show the dependence on the dimension $d$, the operators in the schemes \eqref{eq:IMEX-2d} and \eqref{eq:IMEX-S-2d} can be written as
\begin{align*}
    \mathcal{A}(\bm{G}) ={} & \sum_{j=1}^d \left( \mathcal{D}^{(j),-} \bm{G} \bm{Q}^{(j),+} + \mathcal{D}^{(j),+} \bm{G} \bm{Q}^{(j),-} \right), \\
    \mathcal{H}(\bm{G}) ={} & \frac{1}{|D_{\bm{\Omega}}|} \sum_{j=1}^d \mathcal{D}^{(j),-} \bm{G} \bm{Q}^{(j)} \bm{w}, \qquad
    \mathcal{J}(\bm{\rho}) ={}  \sum_{j=1}^d \mathcal{D}^{(j),+} \bm{\rho} \bm{1}^\top \bm{Q}^{(j)},
\end{align*}
where $s \in \{1, \dots, d\}$ denotes the spatial dimension index.
The operators $\mathcal{D}^{(j),\pm}$ represent the first-order upwind/downwind finite difference operators along the $j$-th dimension on the staggered grid.
The diagonal matrices $\bm{Q}^{(j)}$ and their upwind components $\bm{Q}^{(j),\pm}$ are constructed from the $j$-th coordinates of the quadrature nodes.

\subsection{Energy stability of the full-rank method}

In the following analysis, we introduce the discrete inner product and the weighted inner product (and the associated norms) as
\begin{align*}
    \inner*{\bm{f}_1, \bm{f}_2}    = \left(\prod_{j=1}^d\,\Delta x^{(j)}\right) (\bm{f}_1)^\top \bm{f}_2, \quad
    \inner*{\bm{F}_1, \bm{F}_2}_w  = \left(\prod_{j=1}^d\,\Delta x^{(j)}\right) \tr\left(\bm{F}_1 \bm{M}^2 (\bm{F}_2)^\top \right),
\end{align*}
for any $\bm{f}_1,\bm{f}_2 \in \mathbb{R}^{N_{\bm{\rho}}}$ and $\bm{F}_1,\bm{F}_2 \in \mathbb{R}^{N_{\bm{g}} \times N_{\bm{\Omega}}}$,
where
\begin{equation}
\label{eq:M}
    \bm{M} = \text{diag}(\sqrt{w_1}, \dots, \sqrt{w_{N_{\bm{\Omega}}}}) \in \mathbb{R}^{N_{\bm{\Omega}} \times N_{\bm{\Omega}}}
\end{equation}
is the diagonal matrix consisting of the square roots of the quadrature weights.
\begin{definition} \label{def:energy}
    Given $\theta \in [0,1]$, we define a discrete energy
    \begin{align*}
        E^n_{\theta} ={} & |D_{\bm{\Omega}}| \norm*{\bm{\rho}^n}^2 + \varepsilon^2 \norm*{\bm{G}^n}_w^2 + (1-\theta) \Delta t \sigma^s_0 \norm*{\bm{G}^n}_w^2.
    \end{align*}
    The scheme is said to be: (i) $\theta$-stable if $E^{n+1}_{\theta} \le E^n_{\theta}$ for $n \ge 0$; (ii) stable if there exists such a $\theta \in [0, 1]$; and (iii) unconditionally stable if the stability holds for any $\Delta t > 0$.
\end{definition}

Building on \cite{liu2010analysis, einkemmer2024asymptoticpreserving},
we adapt temporal stability arguments to the full-rank IMEX scheme as follows.
\begin{theorem}[Energy stability of the IMEX scheme]\label{thm:IMEX-stability}
    Assume that $\Phi = 0$.
    The IMEX scheme~\eqref{eq:IMEX-2d} is $1$-stable under the time step condition
    \begin{equation}\label{eq:IMEX-time-step}
        \Delta t \le{}
        \min_{j = 1,\dots,d} \left(
        \frac{\varepsilon (\Delta x^{(j)} / d)}{C_A  \norm*{\bm{\Omega}}_{\infty} + C_B^{(j)}}
        +
        \frac{\sigma_0^s (\Delta x^{(j)} / d)^2}{2 C_A \norm*{\bm{\Omega}}_{\infty}^2 + 2 C_B^{(j)} \norm*{\bm{\Omega}}_{\infty}}
        \right),
    \end{equation}
    where $C_A = 1$, $C_B^{(j)} = \frac{1}{|D_{\bm{\Omega}}|} \bm{w}^\top \abs*{\bm{\Omega}^{(j)}} \bm{1} = \frac12$ for both one- and two-dimensional cases.
\end{theorem}

The following energy analysis of IMEX-S scheme follows the work \cite{peng2020stabilityenhanced}.
While our spatial and angular discretizations differ from that in \cite{peng2020stabilityenhanced}, the stability is primarily governed by the temporal treatment.
The proof techniques remain applicable to our full-rank IMEX-S scheme and serve as the foundation for the subsequent low-rank analysis.

\begin{theorem}[Energy $\theta$-stability of the IMEX-S scheme]\label{thm:IMEX-S-theta-stability}
    Assume that $\Phi = 0$.
    Given $\theta \in [0,1]$, the IMEX-S scheme~\eqref{eq:IMEX-S-2d} is $\theta$-stable under the time step condition
    \begin{equation}\label{eq:IMEX-S-theta-time-step}
        \left( \frac{\varepsilon\,C_A d}{2\min_{j} \Delta x^{(j)}} - \frac{(1-\theta)}{4\norm*{\bm{\Omega}}_{\infty}} \sigma^s_0\right) \Delta t \le \frac{\varepsilon^2}{2\norm*{\bm{\Omega}}_{\infty}},
    \end{equation}
    where $C_A = \norm*{\bm{\Omega}}_{\infty} = 1$.
    The scheme is unconditionally $\theta$-stable if
    \begin{equation*}
        \frac{\varepsilon\,C_A d}{2\min_{j} \Delta x^{(j)}} \le{} \frac{(1-\theta)}{4\norm*{\bm{\Omega}}_{\infty}} \sigma^s_0.
    \end{equation*}
\end{theorem}

We can optimize the results in Theorem~\ref{thm:IMEX-S-theta-stability} in $\theta$ to maximize the unconditionally stable region and also the allowable time step size when the scheme is conditionally stable.

\begin{corollary}[Energy stability of the IMEX-S scheme] \label{thm:IMEX-S-stability}
    Assume that $\Phi = 0$.
    The IMEX-S scheme~\eqref{eq:IMEX-S-2d} is stable under the time step condition
    \begin{equation}\label{eq:IMEX-S-time-step}
        \left( \frac{\varepsilon\,C_A d}{2\min_{j} \Delta x^{(j)}} - \frac{\sigma^s_0}{4\norm*{\bm{\Omega}}_{\infty}} \right) \Delta t \le \frac{\varepsilon^2}{2\norm*{\bm{\Omega}}_{\infty}},
    \end{equation}
    where $C_A = \norm*{\bm{\Omega}}_{\infty} = 1$.
    The scheme is unconditionally $\theta$-stable if
    \begin{equation*}
        \frac{\varepsilon\,C_A d}{2\min_{j} \Delta x^{(j)}} \le{} \frac{\sigma^s_0}{4\norm*{\bm{\Omega}}_{\infty}}.
    \end{equation*}
\end{corollary}

We first present some auxiliary lemmas required in the analysis,
followed by the proofs of the corresponding theorems.

\begin{lemma}\label{lemma:lm-1}
    \begin{equation}\label{eq:lm-1-target}
        \begin{aligned}
            \inner*{\mathcal{A}(\bm{G}^n), \bm{G}^{n+1}}_w ={} & \sum_{j=1}^d \frac{\Delta x^{(j)}}2 \inner*{\mathcal{D}^{(j),+} \bm{G}^{n+1} \abs*{\bm{Q}^{(j)}}, \mathcal{D}^{(j),+} \bm{G}^{n+1}}_w
            \\ & - \inner*{\mathcal{A}^*(\bm{G}^{n+1}), \bm{G}^{n+1} - \bm{G}^n}_w.
        \end{aligned}
    \end{equation}
    where $\mathcal{A}^*$ denotes the adjoint operator of $\mathcal{A}$ with respect to the weighted inner product $\inner*{\,\cdot\,,\,\cdot\,}_w$, defined as
    \begin{align*}
        \mathcal{A}^*(\bm{G}^n) ={} &
        - \sum_{j=1}^d \left(\mathcal{D}^{(j),-} \bm{G}^n \bm{Q}^{(j),-} + \mathcal{D}^{(j),+} \bm{G}^n \bm{Q}^{(j),+}\right).
    \end{align*}
\end{lemma}
\begin{proof}
    We begin by reformulating the discrete advection operator $\mathcal{A}$ as
    \begin{equation}
        \begin{aligned}
            \mathcal{A}(\bm{G}) ={} & \sum_{j=1}^d \left(
            \mathcal{D}^{(j),-} \bm{G} \bm{Q}^{(j),+} + \mathcal{D}^{(j),+} \bm{G} \bm{Q}^{(j),-}
            \right)
            \\
            ={}                     &
            \sum_{j=1}^d \left(
            \frac{\mathcal{D}^{(j),+} - \mathcal{D}^{(j),-}}{2 \Delta x^{(j)}}
            \bm{G} \bm{Q}^{(j)} - \frac{\Delta x^{(j)}}2 \mathcal{D}^{(j),-} \mathcal{D}^{(j),+} \bm{G} \abs*{\bm{Q}^{(j)}}
            \right).
        \end{aligned}
    \end{equation}
    Thus, we derive
    \begin{equation}\label{eq:lm-AG-G}
        \begin{aligned}
            \inner*{\mathcal{A}(\bm{G}^{n+1}), \bm{G}^{n+1}}_w
            ={} & \sum_{j=1}^d
            \inner*{\frac{\mathcal{D}^{(j),+} - \mathcal{D}^{(j),-}}{2 \Delta x^{(j)}} \bm{G}^{n+1} \bm{Q}^{(j)}, \bm{G}^{n+1}}_w
            \\ & - \sum_{j=1}^d \frac{\Delta x^{(j)}}2 \inner*{\mathcal{D}^{(j),-} \mathcal{D}^{(j),+} \bm{G}^{n+1} \abs*{\bm{Q}^{(j)}}, \bm{G}^{n+1}}_w                                                                                                                                     \\
            ={} & \sum_{j=1}^d \frac{\Delta x^{(j)}}2 \inner*{\mathcal{D}^{(j),+} \bm{G}^{n+1} \abs*{\bm{Q}^{(j)}}, \mathcal{D}^{(j),+} \bm{G}^{n+1}}_w.
        \end{aligned}
    \end{equation}
    The first term vanishes owing to a shift of indices, whereas the second term can be reformulated using a discrete summation-by-parts argument.
    To prove \eqref{eq:lm-1-target}, we proceed as follows:
    \begin{align*}
        \inner*{\mathcal{A}(\bm{G}^n), \bm{G}^{n+1}}_w ={} & \inner*{\mathcal{A}(\bm{G}^{n+1}), \bm{G}^{n+1}}_w
        - \inner*{\mathcal{A}(\bm{G}^{n+1}-\bm{G}^n), \bm{G}^{n+1}}_w.
    \end{align*}
    The first term corresponds to \eqref{eq:lm-AG-G}. Applying summation by parts to the second term yields
        {\small\begin{align*}
                & \inner*{\mathcal{A}(\bm{G}^{n+1}-\bm{G}^n), \bm{G}^{n+1}}
                \\
                ={} & - \sum_{j=1}^d \inner*{
                    \mathcal{D}^{(j),-}\left(\bm{G}^{n+1} - \bm{G}^n\right) \bm{Q}^{(j),+} + \mathcal{D}^{(j),+}\left(\bm{G}^{n+1} - \bm{G}^n\right) \bm{Q}^{(j),-},
                    \bm{G}^{n+1}}_w
                \\
                ={}                                                           & \sum_{j=1}^d \inner*{
                    \bm{G}^{n+1} - \bm{G}^n,
                    \mathcal{D}^{(j),+} \bm{G}^{n+1} \bm{Q}^{(j),+} + \mathcal{D}^{(j),-} \bm{G}^{n+1} \bm{Q}^{(j),-}}_w
                \\
                ={}                                                           & - \inner*{\mathcal{A}^*(\bm{G}^{n+1}), \bm{G}^{n+1} - \bm{G}^n}_w.
            \end{align*}}
    This concludes the proof.
\end{proof}

\begin{lemma}\label{lemma:lm-2}
    The adjoint advection operator $\mathcal{A}^*$ fulfills the following inequality:
    \begin{align*}
        \norm*{\mathcal{A}^*(\bm{G}^{n+1})}_w^2 \le{} C_A d \sum_{j=1}^d \norm*{\mathcal{D}^{(j),+} \bm{G}^{n+1} \abs*{\bm{Q}^{(j)}}}_w^2,
    \end{align*}
    where the constant $C_A$ satisfies $C_A = 1$ for the $S_N$ method.
\end{lemma}
\begin{proof}
    From the definition of the adjoint operator $\mathcal{A}^*$, it follows that
    \begin{align*}
        \norm*{\mathcal{A}^*(\bm{G}^{n+1})}_w^2 \le{} & d \sum_{j=1}^d \norm*{
            \mathcal{D}^{(j),+} \bm{G}^{n+1} \bm{Q}^{(j),+} + \mathcal{D}^{(j),-} \bm{G}^{n+1} \bm{Q}^{(j),-}
        }_w^2.
    \end{align*}
    Using the property $\bm{Q}^{(j),+} \bm{M}^2 \bm{Q}^{(j),-} = \bm{0}$ for $j=1,\dots,d$, the cross terms vanish, yielding
    \begin{equation}\label{eq:lm-2-1}
        \norm*{\mathcal{A}^*(\bm{G}^{n+1})}_w^2 \le{}  d \sum_{j=1}^d \left(\norm*{\mathcal{D}^{(j),+} \bm{G}^{n+1} \bm{Q}^{(j),+}}_w^2
        + \norm*{\mathcal{D}^{(j),-} \bm{G}^{n+1} \bm{Q}^{(j),-}}_w^2\right).
    \end{equation}
    Applying a shift of indices to the second term in \eqref{eq:lm-2-1}, we obtain
    \begin{align*}
        \norm*{\mathcal{A}^*(\bm{G}^{n+1})}_w^2 \le{} & d \sum_{j=1}^d \left(\norm*{\mathcal{D}^{(j),+} \bm{G}^{n+1} \bm{Q}^{(j),+}}_w^2
        + \norm*{\mathcal{D}^{(j),+} \bm{G}^{n+1} \bm{Q}^{(j),-}}_w^2\right)
        \\
        ={}                                           & d \sum_{j=1}^d \norm*{\mathcal{D}^{(j),+} \bm{G}^{n+1} \bm{Q}^{(j),+} - \mathcal{D}^{(j),+} \bm{G}^{n+1} \bm{Q}^{(j),-}}_w^2
        \\
        ={}& d \sum_{j=1}^d \norm*{\mathcal{D}^{(j),+} \bm{G}^{n+1} \abs*{\bm{Q}^{(j)}}}_w^2.
    \end{align*}
\end{proof}

\begin{lemma}\label{lemma:lm-3}
    For each spatial direction $j=1,\dots,d$, the matrix inequality below holds:
    \begin{align*}
        \bm{Q}^{(j)} \bm{w} \bm{w}^\top (\bm{Q}^{(j)})^\top \le{}
        C_B^{(j)} |D_{\bm{\Omega}}| \abs*{\bm{Q}^{(j)}} \bm{M}^2, \quad j=1,\dots,d,
    \end{align*}
    where the constant $C_B^{(j)}$ is given by $C_B^{(j)} := \frac{1}{|D_{\bm{\Omega}}|} \bm{w}^\top \abs*{\bm{\Omega}^{(j)}} \bm{1}$.
\end{lemma}
\begin{proof}
    For an arbitrary vector $\bm{h}$, applying the Cauchy--Schwarz inequality, we obtain
    \begin{align*}
        & \bm{h}^\top \bm{Q}^{(j)} \bm{w} \bm{w}^\top (\bm{Q}^{(j)})^\top \bm{h}
        ={}  \left(\sum_k w_k \Omega^{(j)}_k h_k\right)^2
        \\
       ={}  &  \left(\sum_k w_k^{\frac12} \operatorname{sgn}\left(\Omega^{(j)}_k\right)
        w_k^{\frac12}  \abs*{\Omega^{(j)}_k}^{\frac12}  \abs*{\Omega^{(j)}_k}^{\frac12} h_k\right)^2
        \\
        \le{} & \left(\sum_k w_k \abs*{\Omega^{(j)}_k} \right) \left(\sum_{\ell} w_{\ell} \abs*{\Omega^{(j)}_{\ell}} h_{\ell}^2 \right)
        ={}  C_B^{(j)} |D_{\bm{\Omega}}| \bm{h}^\top \abs*{\bm{Q}^{(j)}} \bm{M}^2 \bm{h}.
    \end{align*}
\end{proof}

\begin{proof}[\textbf{Proof of Theorem~\ref{thm:IMEX-stability}}]
    Multiplying the macroscopic equation from the left by $\Delta t (\prod_{j=1}^d \Delta x^{(j)}) (\bm{\rho}^{n+1})^\top$, we obtain
    \begin{equation}\label{eq:rte2-full-1}
        \frac12 \left(\norm*{\bm{\rho}^{n+1}}^2 -  \norm*{\bm{\rho}^n}^2 + \norm*{\bm{\rho}^{n+1} - \bm{\rho}^n}^2\right) + \Delta t \inner*{\bm{\rho}^{n+1}, \mathcal{H}(\bm{G}^{n+1})} = - \Delta t \norm*{\sqrt{\bm{\sigma}^a} \bm{\rho}^{n+1}}^2.
    \end{equation}
    Multiplying the microscopic equation from the right by $\Delta t (\prod_{j=1}^d \Delta x^{(j)}) \bm{M}^2 (\bm{G}^{n+1})^\top$ and then taking the trace,
    using $\bm{G}^{n+1} \bm{M}^2 \bm{1} = \bm{G}^{n+1} \bm{w} = \bm{0}$,
    we obtain
    \begin{align}
         & \frac12 \left(\norm*{\bm{G}^{n+1}}_{w}^2 - \norm*{\bm{G}^{n}}_w^2 +  \norm*{\bm{G}^{n+1} - \bm{G}^{n}}_w^2\right) + \frac{\Delta t}{\varepsilon}
        \inner*{\mathcal{A}(\bm{G}^n), \bm{G}^{n+1}}_w   \notag                                                                                             \\
         & \qquad = - \frac{\Delta t }{\varepsilon^2} \inner*{\mathcal{J}(\bm{\rho}^{n}), \bm{G}^{n+1}}_w
        - \frac{\Delta t }{\varepsilon^2} \norm*{\sqrt{\bm{\sigma}^s}\bm{G}^{n+1}}_w^2 - \Delta t \norm*{\sqrt{\bm{\sigma}^a} \bm{G}^{n+1}}_w^2.  \label{eq:rte2-full-2}
    \end{align}
    Substituting the summation-by-parts identity, $|D_{\bm{\Omega}}| \inner*{\bm{\rho}, \mathcal{H}(\bm{G})} = - \inner*{\mathcal{J}(\bm{\rho}),\bm{G}}_w$,
    into \eqref{eq:rte2-full-2} yields
    \begin{align}\label{eq:rte2-full-2a}
         & \frac12 \left(\norm*{\bm{G}^{n+1}}_{w}^2 - \norm*{\bm{G}^{n}}_w^2 +  \norm*{\bm{G}^{n+1} - \bm{G}^{n}}_w^2\right) + \frac{\Delta t}{\varepsilon}
        \inner*{\mathcal{A}(\bm{G}^n), \bm{G}^{n+1}}_w   \notag                                                                                             \\
         & \qquad = \frac{|D_{\bm{\Omega}}|\Delta t}{\varepsilon^2} \inner*{\bm{\rho}^n, \mathcal{H}(\bm{G}^{n+1})}
        - \frac{\Delta t }{\varepsilon^2} \norm*{\sqrt{\bm{\sigma}^s}\bm{G}^{n+1}}_w^2 - \Delta t \norm*{\sqrt{\bm{\sigma}^a} \bm{G}^{n+1}}_w^2.
    \end{align}
    Adding $|D_{\bm{\Omega}}| \times$ \eqref{eq:rte2-full-1} and $\varepsilon^2 \times$ \eqref{eq:rte2-full-2a} yields
    \begin{align*}
         & \frac{1}{2} \left(E^{n+1}_{1} - E^n_{1}\right)
        +  \frac12 \left(|D_{\bm{\Omega}}| \norm*{\bm{\rho}^{n+1} - \bm{\rho}^{n}}^2
        + \varepsilon^2 \norm*{\bm{G}^{n+1} - \bm{G}^{n}}_w^2\right)
        \notag                                                                                            \\
         & + \varepsilon \Delta t
        \inner*{\mathcal{A}(\bm{G}^n), \bm{G}^{n+1}}_w
        + |D_{\bm{\Omega}}| \Delta t \norm*{\sqrt{\bm{\sigma}^a} \bm{\rho}^{n+1}}^2
        + \Delta t \norm*{\sqrt{\bm{\sigma}^s}\bm{G}^{n+1}}_w^2
        \notag                                                                                            \\
         & + \varepsilon^2 \Delta t \norm*{\sqrt{\bm{\sigma}^a} \bm{G}^{n+1}}_w^2
        + |D_{\bm{\Omega}}| \Delta t \inner*{\bm{\rho}^{n+1} - \bm{\rho}^n, \mathcal{H}(\bm{G}^{n+1})}
        = 0.
    \end{align*}
    According to Lemma~\ref{lemma:lm-1}, we have
    \begin{align*}
        \inner*{\mathcal{A}(\bm{G}^n), \bm{G}^{n+1}}_w
        ={} &
        \sum_{j=1}^d \frac{\Delta x^{(j)}}2 \inner*{\mathcal{D}^{(j),+} \bm{G}^{n+1} \abs*{\bm{Q}^{(j)}}, \mathcal{D}^{(j),+} \bm{G}^{n+1}}_w
        \\
        & - \inner*{\mathcal{A}^*(\bm{G}^{n+1}), \bm{G}^{n+1} - \bm{G}^n}_w.
    \end{align*}
    By the Young's inequality and Lemma~\ref{lemma:lm-2},
    \begin{align*}
        & \abs*{\inner*{\mathcal{A}^*(\bm{G}^{n+1}), \bm{G}^{n+1} - \bm{G}^n}_w}
        \\
        \le{} &
        \eta \norm*{\bm{G}^{n+1} - \bm{G}^n}_w^2 + \frac{C_A d}{4\eta} \sum_{j=1}^d \norm*{\mathcal{D}^{(j),+} \bm{G}^{n+1} \abs*{\bm{Q}^{(j)}}}_w^2
        \\
        ={}   &
        \frac{\varepsilon}{2\Delta t} \norm*{\bm{G}^{n+1} - \bm{G}^n}_w^2
        + \frac{\Delta t}{2\varepsilon} C_A d \sum_{j=1}^d \norm*{\mathcal{D}^{(j),+} \bm{G}^{n+1} \abs*{\bm{Q}^{(j)}}}_w^2,
    \end{align*}
    where $\eta = \frac{\varepsilon}{2\Delta t}$.
    By the macro equation, it follows that
    \begin{align*}
            & \frac12 \norm*{\bm{\rho}^{n+1} - \bm{\rho}^{n}}^2
        + \Delta t \inner*{\bm{\rho}^{n+1} - \bm{\rho}^n, \mathcal{H}(\bm{G}^{n+1})}
        \\
        ={} & \frac{\Delta t^2}2 \left(
        \norm*{\mathcal{H}(\bm{G}^{n+1})}^2
        + \norm*{\bm{\sigma}^a \bm{\rho}^{n+1}}^2
        + 2 \inner*{\bm{\sigma}^a \bm{\rho}^{n+1}, \mathcal{H}(\bm{G}^{n+1})}
        \right)
        \\
            & - \Delta t^2 \inner*{\bm{\sigma}^a \bm{\rho}^{n+1} + \mathcal{H}(\bm{G}^{n+1}), \mathcal{H}(\bm{G}^{n+1})}
        \\
        ={} & \frac{\Delta t^2}2 \norm*{\bm{\sigma}^a \bm{\rho}^{n+1}}^2
        - \frac{\Delta t^2}2 \norm*{\mathcal{H}(\bm{G}^{n+1})}^2.
    \end{align*}
    Hence, we obtain
    \begin{align*}
        \frac{1}{2} \left(E^{n+1}_{1} - E^n_{1}\right)
        \le{} & -\frac{\varepsilon^2}2 \norm*{\bm{G}^{n+1} - \bm{G}^{n}}_w^2
        \\
        & - \varepsilon \Delta t \sum_{j=1}^d \frac{\Delta x^{(j)}}2 \inner*{\mathcal{D}^{(j),+} \bm{G}^{n+1} \abs*{\bm{Q}^{(j)}}, \mathcal{D}^{(j),+} \bm{G}^{n+1}}_w                     \\
              & + \frac{\varepsilon^2}{2} \norm*{\bm{G}^{n+1} - \bm{G}^n}_w^2 + \frac{\Delta t^2}{2} C_A d \sum_{j=1}^d \norm*{\mathcal{D}^{(j),+} \bm{G}^{n+1} \abs*{\bm{Q}^{(j)}}}_w^2
        \notag
        \\
              & - |D_{\bm{\Omega}}| \Delta t \norm*{\sqrt{\bm{\sigma}^a} \bm{\rho}^{n+1}}^2
        - \Delta t \norm*{\sqrt{\bm{\sigma}^s}\bm{G}^{n+1}}_w^2 - \varepsilon^2 \Delta t \norm*{\sqrt{\bm{\sigma}^a} \bm{G}^{n+1}}_w^2
        \\
              & - \frac{\Delta t^2}2 |D_{\bm{\Omega}}| \left(
        \norm*{\bm{\sigma}^a \bm{\rho}^{n+1}}^2
        - \norm*{\mathcal{H}(\bm{G}^{n+1})}^2
        \right).
    \end{align*}
    Then, by dropping some negative terms and utilizing the lower bound of the scattering coefficient $\sigma^s \ge \sigma^s_0$, we have
    \begin{align*}
        \frac{1}{2} \left(E^{n+1}_{1} - E^n_{1}\right)
        \le{} &
        - \varepsilon \Delta t \sum_{j=1}^d \frac{\Delta x^{(j)}}2 \inner*{\mathcal{D}^{(j),+} \bm{G}^{n+1} \abs*{\bm{Q}^{(j)}}, \mathcal{D}^{(j),+} \bm{G}^{n+1}}_w
        \\
              & + \frac{\Delta t^2}2 C_A d \sum_{j=1}^d \norm*{\mathcal{D}^{(j),+} \bm{G}^{n+1} \abs*{\bm{Q}^{(j)}}}_w^2
        - \Delta t \sigma_0^s \norm*{\bm{G}^{n+1}}_w^2
        \\
        & + \frac{\Delta t^2}2 |D_{\bm{\Omega}}| \norm*{\mathcal{H}(\bm{G}^{n+1})}^2.
    \end{align*}
    It is easy to verify that
    \begin{align*}
        \norm*{\bm{G}^{n+1}}_w^2 & \ge \frac{1}{d} \sum_{j=1}^d \frac{(\Delta x^{(j)})^2}{4} \norm*{\mathcal{D}^{(j),-} \bm{G}^{n+1}}_w^2,
    \end{align*}
    and
    \[
        |D_{\bm{\Omega}}|^2 \norm*{\mathcal{H}(\bm{G}^{n+1})}^2 \le{} d \sum_{j=1}^d
        \norm*{\mathcal{D}^{(j),-} \bm{G}^{n+1} \bm{Q}^{(j)} \bm{w}}^2.
    \]
    To guarantee that $E^{n+1}_{\theta} \le{} E^n_{\theta}$, it is sufficient to show that
    \begin{align*}
         & \varepsilon \Delta t \sum_{j=1}^d \frac{\Delta x^{(j)}}2 \inner*{\mathcal{D}^{(j),+} \bm{G}^{n+1} \abs*{\bm{Q}^{(j)}}, \mathcal{D}^{(j),+} \bm{G}^{n+1}}_w
        \\
        & - \frac{\Delta t^2}2 C_A d \sum_{j=1}^d \norm*{\mathcal{D}^{(j),+} \bm{G}^{n+1} \abs*{\bm{Q}^{(j)}}}_w^2
         + \Delta t \sigma_0^s \sum_{j=1}^d \frac{(\Delta x^{(j)})^2}{4d} \norm*{\mathcal{D}^{(j),-} \bm{G}^{n+1}}_w^2
        \\
        & - \frac{\Delta t^2}2 \frac{d}{|D_{\bm{\Omega}}|} \sum_{j=1}^d \norm*{\mathcal{D}^{(j),-} \bm{G}^{n+1} \bm{Q}^{(j)} \bm{w}}^2
        \ge 0.
    \end{align*}
    Our objective is to prove that for $j=1,\dots,d$,
    \begin{align*}
        \bm{S}^{(j)} :={} &
        \frac{\varepsilon \Delta t \Delta x^{(j)}}2 \bm{M}^2 |\bm{Q}^{(j)}|
        - \frac{\Delta t^2}2 C_A d \bm{M}^2 \abs*{\bm{Q}^{(j)}}^2
        + \frac{\sigma_0^s \Delta t (\Delta x^{(j)})^2}{4 d} \bm{M}^2
        \\
        & - \frac{\Delta t^2 \,d}{2 |D_{\bm{\Omega}}|} \bm{Q}^{(j)} \bm{w} \bm{w}^\top (\bm{Q}^{(j)})^\top,
    \end{align*}
    is positive definite if $\Delta t$ is sufficiently small.
    Using Lemma~\ref{lemma:lm-3}, we have
    \begin{align*}
        \bm{S}^{(j)} \ge{} &
        \frac{\varepsilon \Delta t \Delta x^{(j)}}2 \bm{M}^2 |\bm{Q}^{(j)}|
        - \frac{\Delta t^2}2 C_A d \bm{M}^2 \abs*{\bm{Q}^{(j)}}^2
        + \frac{\sigma_0^s \Delta t (\Delta x^{(j)})^2}{4 d} \bm{M}^2
        \\
        & - \frac{\Delta t^2 \,d}{2} C_B^{(j)} \abs*{\bm{Q}^{(j)}} \bm{M}^2.
    \end{align*}
    Hence,
    the matrices $\{\bm{S}^{(j)},j=1,\dots,d\}$ are positive definite if $\Delta t$ is sufficiently small such that for all $j,k$
    \begin{align*}
        \Delta t w_k \left(
        \frac{\varepsilon \Delta x^{(j)}}2 \abs*{\Omega^{(j)}_k} - \frac{\Delta t}2 C_A d \abs*{\Omega^{(j)}_k}^2
        + \frac{\sigma_0^s (\Delta x^{(j)})^2}{4 d}
        - \frac{\Delta t d}{2} C_B^{(j)} \abs*{\Omega^{(j)}_k}
        \right)
        \ge{} & 0.
    \end{align*}
    Thus, we derive the following bound
    \begin{align*}
        \Delta t \le{} &
        \min_{j = 1,\dots,d} \left(
        \frac{\varepsilon (\Delta x^{(j)} / d)}{C_A  \norm*{\bm{\Omega}}_{\infty} + C_B^{(j)}}
        +
        \frac{\sigma_0^s (\Delta x^{(j)} / d)^2}{2 C_A \norm*{\bm{\Omega}}_{\infty}^2 + 2 C_B^{(j)} \norm*{\bm{\Omega}}_{\infty}}
        \right),
    \end{align*}
    where $C_A = 1$, $C_B^{(j)} = \frac{1}{|D_{\bm{\Omega}}|} \bm{w}^\top \abs*{\bm{\Omega}^{(j)}} \bm{1} = \frac12$ for both one- and two-dimensional cases.
    Here we use $\norm*{\bm{\Omega}}_{\infty}$ to replace $\max_{k} \abs*{\Omega^{(j)}_k}$ for simplicity.
\end{proof}

\begin{proof}[\textbf{Proof of Theorem~\ref{thm:IMEX-S-theta-stability}}]
    Following the same procedure as in the proof of Theorem~\ref{thm:IMEX-stability}, we derive
    \begin{equation}\label{eq:proof-IMEX-S-1}
        \begin{aligned}
             & \frac{1}{2} \left(E^{n+1}_{1} - E^n_{1}\right)
            +  \frac12 \left(|D_{\bm{\Omega}}| \norm*{\bm{\rho}^{n+1} - \bm{\rho}^{n}}^2
            + \varepsilon^2 \norm*{\bm{G}^{n+1} - \bm{G}^{n}}_w^2\right) \\
             & + \varepsilon \Delta t
            \inner*{\mathcal{A}(\bm{G}^n), \bm{G}^{n+1}}_w
            + |D_{\bm{\Omega}}| \Delta t \norm*{\sqrt{\bm{\sigma}^a} \bm{\rho}^{n+1}}^2
            \\
            & + \Delta t \norm*{\sqrt{\bm{\sigma}^s}\bm{G}^{n+1}}_w^2 + \varepsilon^2 \Delta t \norm*{\sqrt{\bm{\sigma}^a} \bm{G}^{n+1}}_w^2
            = 0.
        \end{aligned}
    \end{equation}
    We first bound the scattering term by $\sigma^s_0 \norm*{\bm{G}^{n+1}}_w^2 \le \norm*{\sqrt{\bm{\sigma}^s}\bm{G}^{n+1}}_w^2$.
    By further decomposing $\norm*{\bm{G}^{n+1}}_w^2$ as
    \begin{align*}
        \norm*{\bm{G}^{n+1}}_w^2
        ={} & \theta \norm*{\bm{G}^{n+1}}_w^2
        + (1-\theta) \left(\frac12\norm*{\bm{G}^{n+1}}_w^2 - \frac12\norm*{\bm{G}^{n}}_w^2\right)
        \\
            & + (1-\theta) \left(
        \frac14 \norm*{(\bm{G}^{n+1}-\bm{G}^{n})}_w^2 + \frac14 \norm*{(\bm{G}^{n+1}+\bm{G}^{n})}_w^2
        \right).
    \end{align*}
    and substituting this into \eqref{eq:proof-IMEX-S-1}, we arrive at the energy inequality:
    \begin{align*}
         & \frac{1}{2} \left(E^{n+1}_{\theta} - E^n_{\theta}\right)
        +  \frac12 \left(|D_{\bm{\Omega}}| \norm*{\bm{\rho}^{n+1} - \bm{\rho}^{n}}^2
        + \varepsilon^2 \norm*{\bm{G}^{n+1} - \bm{G}^{n}}_w^2\right)
        \notag                                                      \\
         & + \varepsilon \Delta t
        \inner*{\mathcal{A}(\bm{G}^n), \bm{G}^{n+1}}_w
        + |D_{\bm{\Omega}}| \Delta t \norm*{\sqrt{\bm{\sigma}^a} \bm{\rho}^{n+1}}^2
        + \Delta t \theta \sigma^s_0 \norm*{\bm{G}^{n+1}}_w^2
        \notag
        \\ & + \varepsilon^2 \Delta t \norm*{\sqrt{\bm{\sigma}^a} \bm{G}^{n+1}}_w^2
        \\ & + \Delta t (1-\theta) \left(
        \frac14 \sigma^s_0 \norm*{(\bm{G}^{n+1}-\bm{G}^{n})}_w^2 + \frac14 \sigma^s_0 \norm*{(\bm{G}^{n+1}+\bm{G}^{n})}_w^2
        \right)
        \le{} 0.
    \end{align*}
    Following Lemma~\ref{lemma:lm-1}, we have
    \begin{align*}
        \inner*{\mathcal{A}(\bm{G}^n), \bm{G}^{n+1}}_w
        ={} &
        \sum_{j=1}^d \frac{\Delta x^{(j)}}2 \inner*{\mathcal{D}^{(j),+} \bm{G}^{n+1} \abs*{\bm{Q}^{(j)}}, \mathcal{D}^{(j),+} \bm{G}^{n+1}}_w
        \\
        & - \inner*{\mathcal{A}^*(\bm{G}^{n+1}), \bm{G}^{n+1} - \bm{G}^n}_w.
    \end{align*}
    Applying Young's inequality and Lemma~\ref{lemma:lm-2}, we obtain
    \begin{align*}
        \abs*{\inner*{\mathcal{A}^*(\bm{G}^{n+1}), \bm{G}^{n+1} - \bm{G}^n}_w}
        \le{} &
        \eta \norm*{\bm{G}^{n+1} - \bm{G}^n}_w^2 + \frac{C_A d}{4\eta} \sum_{j=1}^d \norm*{\mathcal{D}^{(j),+} \bm{G}^{n+1} \abs*{\bm{Q}^{(j)}}}_w^2,
    \end{align*}
    where $\eta > 0$ will be determined later.
    Hence,
    \begin{align*}
        \frac{1}{2} \left(E^{n+1}_{\theta} - E^n_{\theta}\right)
        \le{} &
        -\left(\frac{\varepsilon^2}2 + \frac{\Delta t (1-\theta)}4 \sigma^s_0 - \varepsilon \Delta t \eta\right) \norm*{\bm{G}^{n+1} - \bm{G}^{n}}_w^2
        \\
              & -
        \varepsilon \Delta t \sum_{j=1}^d
        \frac{\Delta x^{(j)}}2 \inner*{\mathcal{D}^{(j),+} \bm{G}^{n+1} \abs*{\bm{Q}^{(j)}}, \mathcal{D}^{(j),+} \bm{G}^{n+1}}_w
        \\
              & +
        \varepsilon \Delta t \sum_{j=1}^d
        \frac{C_A d}{4\eta} \norm*{\mathcal{D}^{(j),+} \bm{G}^{n+1} \abs*{\bm{Q}^{(j)}}}_w^2.
    \end{align*}
    To ensure that $E^{n+1}_{\theta} \le{} E^n_{\theta}$, it is sufficient to impose the conditions
    \begin{equation}\label{eq:proof-IMEX-S-2}
        \frac{\varepsilon^2}2 + \frac{\Delta t(1-\theta)}4 \sigma^s_0  - \varepsilon \Delta t \eta \ge 0,
    \end{equation}
    and
    \begin{align*}
        \frac{\Delta x^{(j)}}2 \inner*{\mathcal{D}^{(j),+} \bm{G}^{n+1} \abs*{\bm{Q}^{(j)}}, \mathcal{D}^{(j),+} \bm{G}^{n+1}}_w
        - \frac{C_A d}{4\eta} \norm*{\mathcal{D}^{(j),+} \bm{G}^{n+1} \abs*{\bm{Q}^{(j)}}}_w^2 \ge 0
    \end{align*}
    for $j=1,\ldots,d$.
    The parameter $\eta$ can be selected to satisfy
    \begin{align*}
        \frac{\Delta x^{(j)}}2 \abs*{\bm{Q}^{(j)}}
        \ge
        \frac{C_A d}{4\eta} \abs*{\bm{Q}^{(j)}}^2, \quad j=1,\dots,d.
    \end{align*}
    This inequality simplifies to
    \begin{equation}\label{eq:proof-IMEX-S-3}
        \eta \ge \frac{C_A d}2 \norm*{\bm{\Omega}}_{\infty}
        \left(\max_j \frac{1}{\Delta x^{(j)}}\right)
        =  \frac{C_A d \norm*{\bm{\Omega}}_{\infty}}{2\min_{j} \Delta x^{(j)}}.
    \end{equation}
    By combining \eqref{eq:proof-IMEX-S-2} and \eqref{eq:proof-IMEX-S-3}, we obtain
    \begin{align*}
                   & \frac{\varepsilon^2}2 + \frac{\Delta t(1-\theta)}4 \sigma^s_0 \ge \varepsilon \Delta t \eta \ge \varepsilon \Delta t \frac{C_A d \norm*{\bm{\Omega}}_{\infty}}{2 \min_{j} \Delta x^{(j)}}
        \\
        \iff \quad & \left( \frac{\varepsilon\,C_A d}{2\min_{j} \Delta x^{(j)}} - \frac{(1-\theta)}{4\norm*{\bm{\Omega}}_{\infty}} \sigma^s_0\right) \Delta t \le \frac{\varepsilon^2}{2\norm*{\bm{\Omega}}_{\infty}},
    \end{align*}
    where $C_A = 1$.
    This completes the proof.
\end{proof}

\section{Low-Rank method} \label{sec:numerical-method-lowrank}

We now combine IMEX/IMEX-S time discretizations with DLRA techniques to reduce both computational cost and memory consumption.
Specifically, the coupled macro–micro system is advanced in time by either the IMEX or the IMEX-S scheme, while the microscopic correction is evolved using BUG-type low-rank integrators.
The combination of IMEX schemes with BUG integrators has been investigated previously in the context of $P_N$ discretizations \cite{einkemmer2024asymptoticpreserving,patwardhan2025asymptoticpreserving}.
Here, we instead consider the $S_N$ framework, in which angular integrals are approximated by quadrature rules and are also popularly used. This would result in a weighted discrete angular space and a corresponding energy structure. Such a weighted setting necessitates a compatible modification of the low-rank formulation to ensure stability.
Our primary focus is on the combination of the IMEX-S scheme with an energy-consistent dynamical low-rank approximation.
This approach yields an efficient solver that remains uniformly stable in the diffusion limit, achieving efficiency in phase space through the low-rank evolution and $S_N$ discretization, and in time through the IMEX-S scheme.
For conciseness, we describe the numerical method in detail only for the representative IMEX-S–BUG combination.
Other combinations considered in the numerical experiments are treated analogously.

\subsection{Energy-Consistent Low-Rank Approximation}

We approximate the microscopic correction using a low-rank representation that is compatible with the discrete energy structure of the $S_N$ discretization.
Rather than a standard decomposition, we introduce a low-rank ansatz specifically constructed to be consistent with the energy functional induced by the quadrature weights.
Specifically, the microscopic correction is approximated in the following energy-consistent low-rank form:
\begin{equation*}
    \bm{G}(t) \bm{M} \approx \sum_{i,j=1}^r \bm{X}_i(t) S_{ij}(t) \bm{V}_j(t)^\top = \bm{X}(t) \bm{S}(t) \bm{V}(t)^\top \in \mathbb{R}^{N_{\bm{g}} \times N_{\bm{\Omega}}},
\end{equation*}
where $\bm{X}(t) \in \mathbb{R}^{N_{\bm{g}} \times r}$, $\bm{S}(t) \in \mathbb{R}^{r\times r}$ and $\bm{V}(t) \in \mathbb{R}^{N_{\bm{\Omega}} \times r}$,
and $\bm{M}$ is defined previously in \eqref{eq:M}.
The columns of $\bm{X}(t)$ and $\bm{V}(t)$ are orthonormal.
Multiplying the microscopic equation from the right by $\bm{M}$ yields an evolution equation for
$\bm{X}(t)\bm{S}(t)\bm{V}(t)^\top$,
\begin{align*}
    \partial_t \left(\bm{X}(t) \bm{S}(t) \bm{V}(t)^\top\right) ={}
     & - \frac{1}{\varepsilon} \mathcal{B}\left(\bm{X}(t) \bm{S}(t) \bm{V}(t)^\top\right)  - \frac{1}{\varepsilon^2}  \mathcal{J}(\rho) \bm{M}
    - \frac{1}{\varepsilon^2} \bm{\sigma}^s \bm{X}(t) \bm{S}(t) \bm{V}(t)^\top
    \\
     & - \bm{\sigma}^a \bm{X}(t) \bm{S}(t) \bm{V}(t)^\top,
\end{align*}
where the operator $\mathcal{B}$ is defined by
\begin{align*}
    \mathcal{B}\left(\bm{X}(t) \bm{S}(t) \bm{V}(t)^\top\right)
    ={} & \mathcal{A}\left(\bm{X}(t) \bm{S}(t) \bm{V}(t)^\top \bm{M}^{-1} \right) \left(\mathbf{I} - \frac1{|D_{\bm{\Omega}}|} \bm{w} \bm{1}^\top \right) \bm{M}.
\end{align*}

We now describe a representative implementation based on the IMEX-S scheme combined with the BUG integrator \cite{koch2007dynamical}:
at each time step, the macroscopic variable $\bm{\rho}^{n+1}$ is first obtained by solving the linear system corresponding to \eqref{eq:macro-micro-imex-s-schur-1},
\begin{align*}
     & \left(\frac{1}{\Delta t} + \bm{\sigma}^a
    - \frac{1}{|D_{\bm{\Omega}}| \varepsilon^2}
    \mathcal{D}^- \left\{
    R
    \mathcal{D}^+ (\cdot) \bm{1}^\top \bm{Q}
    \right\} \bm{Q} \bm{w}
    \right) \bm{\rho}^{n+1}
    \notag                                                                                                \\
     & \quad = \frac{1}{\Delta t} \bm{\rho}^n + \bm{\Phi}^{n+1} - \frac1{|D_{\bm{\Omega}}|} \mathcal{D}^-
    \left\{
    R
    \left(\frac{1}{\Delta t} \bm{G}^n - \frac{1}{\varepsilon} \mathcal{A}(\bm{G}^n) \left(\mathbf{I} - \frac1{|D_{\bm{\Omega}}|} \bm{w} \bm{1}^\top \right) \right)
    \right\} \bm{Q} \bm{w}.
\end{align*}
where $R := \left(\frac{1}{\Delta t} + \frac{\sigma^s}{\varepsilon^2} + \sigma^a\right)^{-1}$.
Given $\bm{\rho}^{n+1}$, the low-rank factors are advanced from $t_n$ to $t_{n+1}$ by applying the BUG integrator to the low-rank evolution equation.
The corresponding K-, L-, and S-steps follow the BUG integrator and are detailed below:
\begin{enumerate}
    \item \textbf{K-step:} For an initial condition $\bm{K}^n(x) = \bm{X}^n(x) \bm{S}^n \in\mathbb{R}^{N_x \times r}$,
          we solve the following equation to obtain $\bm{K}^{n+1}$
          \begin{align*}
              \frac{\bm{K}^{n+1} - \bm{K}^n}{\Delta t}
              ={} &
              -\frac{1}{\varepsilon} \mathcal{B}\left(\bm{K} (\bm{V}^n)^\top \right) \bm{V}^n
              - \frac{1}{\varepsilon^2} \mathcal{J}(\bm{\rho}^{n+1}) \bm{M} \bm{V}^n
              \\ & - \frac{1}{\varepsilon^2} \bm{\sigma}^s \bm{K}^{n+1} - \bm{\sigma}^a \bm{K}^{n+1}.
          \end{align*}

    \item \textbf{L-step:} For an initial condition $\bm{L}^n = \bm{V}^n (\bm{S}^n)^\top \in \mathbb{R}^{N \times r}$,
          we solve the following equation to obtain $\bm{L}^{n+1}$
          \begin{align*}
              \frac{\bm{L}^{n+1} - \bm{L}^n}{\Delta t}
              ={} &
              -\frac{1}{\varepsilon} \mathcal{B}\left(\bm{X}^n \bm{L}^\top \right)^\top \bm{X}^n - \frac{1}{\varepsilon^2} \bm{M} (\mathcal{J}(\bm{\rho}^{n+1}))^\top \bm{X}^n
              \notag                                                                                                                                      \\
                  & - \frac{1}{\varepsilon^2} \bm{L}^{n+1} (\bm{X}^n)^\top \bm{\sigma}^s \bm{X}^n -  \bm{L}^{n+1} (\bm{X}^n)^\top \bm{\sigma}^a \bm{X}^n.
          \end{align*}

    \item Update basis: $[\bm{X}^{n+1}, \sim] = \qr(\bm{K}^{n+1})$, $[\bm{V}^{n+1}, \sim] = \qr(\bm{L}^{n+1})$.

    \item \textbf{S-step:} For an initial condition $\widetilde{\bm{S}}^n = (\bm{X}^{n+1})^\top \bm{X}^n \bm{S}^n (\bm{V}^n)^\top \bm{V}^{n+1} \in \mathbb{R}^{r' \times r'}$,
          we solve the following equation to obtain $\bm{S}^{n+1}$.
          \begin{align*}
              \frac{\bm{S}^{n+1} - \widetilde{\bm{S}}^n}{\Delta t}  ={}
               & -\frac{1}{\varepsilon} (\bm{X}^{n+1})^\top \mathcal{B}\left(\bm{X}^{n+1} \bm{S} (\bm{V}^{n+1})^\top \right) \bm{V}^{n+1}
              \\
               & - \frac{1}{\varepsilon^2} (\bm{X}^{n+1})^\top \mathcal{J}(\bm{\rho}^{n+1}) \bm{M} \bm{V}^{n+1}
              \notag                                                                                                                                                  \\
               & - \frac{1}{\varepsilon^2} (\bm{X}^{n+1})^\top \bm{\sigma}^s \bm{X}^{n+1} \bm{S}^{n+1} - (\bm{X}^{n+1})^\top \bm{\sigma}^a \bm{X}^{n+1} \bm{S}^{n+1}.
          \end{align*}

    \item Update $\bm{G}^{n+1} \bm{M} = \bm{X}^{n+1} \bm{S}^{n+1} (\bm{V}^{n+1})^\top$.
\end{enumerate}
Here $\qr(\cdot)$ denotes a QR factorization.

For the IMEX scheme, the macroscopic variable is updated after advancing the
microscopic correction.
In this case, the BUG integrator is applied with the macroscopic density frozen at time level $t_n$, i.e., $\bm{\rho}^{n}$ is used in $\mathcal{J}(\bm{\rho})$ throughout the K-, L-, and S-steps.

From the computational perspective, the low-rank formulation only evolves three small matrices $\bm{X}$, $\bm{S}$, and $\bm{V}$ at each time step,
with the most expensive operation being at the SVD of the small $r\times r$ matrix $\bm{S}$ and the QR factorizations of the matrices $\bm{K}^{n+1}$ and $\bm{L}^{n+1}$.
Therefore, if the spatial and angular dimensions are discretized with $N$ points per dimension, the full-rank system involves $\mathcal{O}(N^{d_x + d_v})$ degrees of freedom.
When the rank $r \ll N$, the dynamical low-rank approximation reduces both memory consumption and per-step computational cost to $\mathcal{O}(N^{\max(d_x,d_v)})$, offering a substantial efficiency advantage, which will be confirmed by the results in Section~\ref{sec:numerical-results}.
Moreover, since the macroscopic variable $\bm{\rho}$ is independent of the angular directions, its update involves $\mathcal{O}(N^{d_x})$ operations, directly proportional to the spatial degrees of freedom, which is smaller than the cost of updating the low-rank microscopic component.

\subsection{Asymptotic-Preserving Property}

The AP property of the proposed low-rank schemes builds on existing IMEX-based dynamical low-rank analyses for kinetic equations \cite{einkemmer2021asymptoticpreserving,einkemmer2024asymptoticpreserving,patwardhan2025asymptoticpreserving}.
For the sake of self contained, we perform an asymptotic analysis in one spatial dimension using the IMEX-S–BUG scheme.

In the limit $\varepsilon \to 0$, the K-step, L-step, and S-step of the IMEX-S-BUG scheme reduce to
\begin{subequations}
    \begin{align}
        \bm{K}^{n+1} & = - (\bm{\sigma}^s)^{-1} \mathcal{J}(\bm{\rho}^{n+1}) \bm{M} \bm{V}^n, \label{eq:ap-1}
        \\
        \bm{L}^{n+1} (\bm{X}^n)^\top \bm{\sigma}^s \bm{X}^n
                     & = - \bm{M} (\mathcal{J}(\bm{\rho}^{n+1}))^\top \bm{X}^n, \label{eq:ap-2}
        \\
        (\bm{X}^{n+1})^\top \bm{\sigma}^s \bm{X}^{n+1} \bm{S}^{n+1}
       & ={}          - (\bm{X}^{n+1})^\top \mathcal{J}(\bm{\rho}^{n+1}) \bm{M} \bm{V}^{n+1}. \label{eq:ap-3}
    \end{align}
\end{subequations}
Using the definition of the discrete operator $\mathcal{H}$ together with the
low-rank representation
$\bm{G}^{n+1} = \bm{X}^{n+1} \bm{S}^{n+1} (\bm{V}^{n+1})^\top \bm{M}^{-1}$,
we obtain
\begin{align*}
    |D_{\bm{\Omega}}| \mathcal{H}(\bm{G}^{n+1})
    ={} & \mathcal{D}^{x,-} \bm{G}^{n+1} \bm{Q}^{x} \bm{w}
    ={} \mathcal{D}^{x,-} \bm{X}^{n+1} \bm{S}^{n+1} (\bm{V}^{n+1})^\top \bm{M} \bm{Q}^{x} \bm{1}.
\end{align*}
Since $(\bm{X}^{n+1})^\top \bm{\sigma}^s \bm{X}^{n+1}$ is invertible, the S-step \eqref{eq:ap-3} yields
\begin{equation}\label{eq:ap-H}
    \begin{aligned}
            & |D_{\bm{\Omega}}| \mathcal{H}(\bm{G}^{n+1})
        \\
        ={} & - \mathcal{D}^{x,-} \bm{X}^{n+1} \left((\bm{X}^{n+1})^\top \bm{\sigma}^s \bm{X}^{n+1}\right)^{-1} (\bm{X}^{n+1})^\top \mathcal{J}(\bm{\rho}^{n+1}) \bm{M} \bm{V}^{n+1} (\bm{V}^{n+1})^\top \bm{M} \bm{Q}^{x} \bm{1}.
    \end{aligned}
\end{equation}
To identify the diffusion-limit form of $\mathcal{H}(\bm{G}^{n+1})$, we need to verify that
\begin{equation}\label{eq:ap-claim}
    \mathcal{H}(\bm{G}^{n+1})
    ={} - \mathcal{D}^{x,-} \left(
    \inner*{\bm{\Omega} \otimes \bm{\Omega}}_{\bm{\Omega}} (\bm{\sigma}^s)^{-1} \mathcal{D}^{x,+} \bm{\rho^{n+1}}
    \right).
\end{equation}
From \eqref{eq:ap-1}, it follows that
\begin{align*}
    - (\bm{\sigma}^s)^{-1} \mathcal{D}^{x,+} \bm{\rho}^{n+1} & \in \textnormal{Range}(\bm{K}^{n+1})  = \textnormal{Range}(\bm{X}^{n+1}).
\end{align*}
Hence,
\begin{equation}\label{eq:ap-H-1}
    \begin{aligned}
            & - \mathcal{D}^{x,-} \bm{X}^{n+1} \left((\bm{X}^{n+1})^\top \bm{\sigma}^s \bm{X}^{n+1}\right)^{-1} (\bm{X}^{n+1})^\top \mathcal{J}(\bm{\rho}^{n+1})
        \\
        ={} & - \mathcal{D}^{x,-} \bm{X}^{n+1} \left((\bm{X}^{n+1})^\top \bm{\sigma}^s \bm{X}^{n+1}\right)^{-1} (\bm{X}^{n+1})^\top \mathcal{D}^{x,+} \bm{\rho}^{n+1} \bm{1}^\top \bm{Q}^x
        \\
        ={} & - \mathcal{D}^{x,-} (\bm{\sigma}^s)^{-1} \mathcal{D}^{x,+} \bm{\rho}^{n+1} \bm{1}^\top \bm{Q}^x,
    \end{aligned}
\end{equation}
Similarly, \eqref{eq:ap-2} implies that
\begin{align*}
    \bm{M} \bm{Q}^x \bm{1} & \in \textnormal{Range}(\bm{L}^{n+1})  = \textnormal{Range}(\bm{V}^{n+1}),
\end{align*}
which yields
\begin{equation}\label{eq:ap-H-2}
    \bm{M} \bm{V}^{n+1} (\bm{V}^{n+1})^\top \bm{M} \bm{Q}^{x} \bm{1} =  \bm{M}^2 \bm{Q}^x \bm{1}.
\end{equation}
The substitution of \eqref{eq:ap-H-1} and \eqref{eq:ap-H-2} into \eqref{eq:ap-H} leads to \eqref{eq:ap-claim}. Inserting \eqref{eq:ap-claim} into the macroscopic equation and taking the limit $\varepsilon \to 0$, we obtain
\begin{align*}
    0 ={} & \frac{\bm{\rho}^{n+1} - \bm{\rho}^n}{\Delta t}
    -  \mathcal{D}^{x,-} \left(
    \inner*{\bm{\Omega} \otimes \bm{\Omega}}_{\bm{\Omega}} (\bm{\sigma}^s)^{-1} \mathcal{D}^{x,+} \bm{\rho^{n+1}}
    \right)
    + \bm{\sigma}^a \bm{\rho}^{n+1} - \bm{\Phi}^{n+1},
\end{align*}
which is precisely a consistent discretization of the diffusion equation \eqref{eq:diffusion}.
Consequently, the IMEX-S–BUG scheme is AP; furthermore, it inherits the unconditional stability in the diffusive limit, consistent with the properties established in Theorem~\ref{thm:IMEX-S-theta-stability} and Corollary~\ref{thm:IMEX-S-stability}.

As shown in \cite{patwardhan2025asymptoticpreserving}, in contrast to the fixed-rank BUG integrator, a naive application of the augmented BUG integrator does not necessarily preserve the AP property, since the truncation step may remove basis vectors that are essential in the diffusion limit.
To address this issue, we adopt the strategy proposed in \cite{patwardhan2025asymptoticpreserving} and modify the aBUG integrator by enriching the approximation spaces as
\begin{align*}
    [\widehat{X}^{n+1}, \sim] & = \qr([- (\bm{\sigma}^s)^{-1} \mathcal{D}^{x,+} \bm{\rho}^{n+1}, \bm{K}^{n+1},\bm{X}^n]), \\
     [\widehat{V}^{n+1}, \sim] & = \qr([\bm{M} \bm{Q}^x \bm{1}, \bm{L}^{n+1},\bm{V}^n]).
\end{align*}
In addition, the standard SVD truncation is replaced by a conservative SVD truncation \cite{einkemmer2023conservation} to ensure the preservation of the AP property.

\subsection{Zero Density Constraint for the Microscopic Component}

The microscopic component is required to satisfy the zero-density constraint $\inner*{\bm{g}}_{\bm{\Omega}} = 0$.
After discretization with the $S_N$ method, it corresponds to $\bm{G} \bm{w} = \bm{0}$. This property is automatically preserved by the full-rank scheme.
If $\bm{G}^n \bm{w} = \bm{0}$, then using
\begin{align*}
    \left(\mathbf{I} - \frac1{|D_v|} \bm{w} \bm{1}^\top \right) \bm{w} ={} \bm{0},
    \quad
    \bm{1}^\top \bm{Q}^{(j)} \bm{w} ={} 0, \quad j=1,\dots,d,
\end{align*}
it follows directly that $\bm{G}^{n+1}\bm{w}=\bm{0}$.
In the low-rank setting, preserving this constraint is equivalent to requiring
\[
    \bm{0} = \bm{G} \bm{w} = \bm{G} \bm{M}^2 \bm{1} = \bm{X}(t,x) \bm{S}(t) \bm{V}(t)^\top \bm{M} \bm{1}.
\]
Since the columns of $\bm{V}(t)$ constitute the basis functions in the angular space, this condition is satisfied by imposing the constraint $\bm{1}^\top \bm{M} \bm{V}(t) = \bm{0}^\top$.

For the BUG integrator, assume that at time level $t_n$ the low-rank approximation satisfies
\[
    \bm{1}^\top \bm{M}\bm{V}^n(\bm{S}^n)^\top(\bm{X}^n)^\top = \bm{0}^\top
    \,\,\Rightarrow\,\,
    \bm{1}^\top \bm{M}\bm{L}^n
    = \bm{1}^\top \bm{M}\bm{V}^n(\bm{S}^n)^\top
    = \bm{0}^\top.
\]
During the L-step, $\bm{L}^{n+1}$ satisfies
\begin{align*}
     & \bm{L}^{n+1}\left(
    1 + \frac{\Delta t}{\varepsilon^2} (\bm{X}^n)^\top \bm{\sigma}^s \bm{X}^n + (\bm{X}^n)^\top \bm{\sigma}^a \bm{X}^n\right)
    \\
     & ={} \bm{L}^n - \frac{\Delta t}{\varepsilon^2} \left(
    \varepsilon \mathcal{B}\left(\bm{X}^n \bm{L}^\top \right)^\top \bm{X}^n
    + \bm{M} (\mathcal{J}(\bm{\rho}^{n+1}))^\top \bm{X}^n
    \right).
\end{align*}
Multiplying from the left by $\bm{1}^\top \bm{M}^2$ and using
\[
    \bm{1}^\top \bm{M}^2
    \Bigl(\mathbf{I}-\frac{1}{|D_{\bm{\Omega}}|}\bm{1}\bm{w}^\top\Bigr)
    = \bm{0}^\top,
    \qquad
    \bm{1}^\top \bm{M}^2 (\bm{Q}^{(j)})^\top \bm{1} = 0, \quad j=1,\dots,d,
\]
we obtain $\bm{1}^\top \bm{M}\bm{L}^{n+1} = \bm{0}^\top$.
Consequently, if any basis $\bm{V}^{n+1}\in \mathrm{Range}(\bm{L}^{n+1})$ satisfies $\bm{1}^\top \bm{M}\bm{V}^{n+1} = \bm{0}^\top$,
the zero-density constraint would be preserved.

Similarly, for the aBUG and AP-aBUG integrators, $\bm{V}^{n+1}$ remains within the feasible subspace.
For the latter in one spatial dimension, the inclusion $\bm{V}^{n+1} \in \mathrm{Range}\bigl(\bm{M} \bm{Q}^x \bm{1}, \bm{L}^{n+1}, \bm{V}^n\bigr)$ preserves the constraint, provided that $\bm{1}^\top \bm{M}^2 \bm{Q}^x \bm{1} = 0$.

A noteworthy technical detail arises regarding the numerical stability of the QR factorization.
For example, in the BUG integrator, the updated matrix $\bm{L}^{n+1}$ may exhibit numerical rank deficiency.
In such cases, a standard QR algorithm, lacking awareness of the underlying physical structure, may introduce basis vectors that drift away from the zero-density subspace. To robustly enforce the constraint within machine precision, we employ a modified QR procedure:
\begin{enumerate}
    \item Construct an orthonormal basis $\bm{Z} \in \mathbb{R}^{N_{\bm{\Omega}} \times (N_{\bm{\Omega}}-1)}$ for the null space of the constraint, such that $\bm{1}^\top \bm{M} \bm{Z} = \bm{0}^\top$. Then $\bm{L}^{n+1}=\bm{Z}\bm{Z}^\top \bm{L}^{n+1}$.
    \item Compute the QR factorization $[\widetilde{\bm{Q}},\widetilde{\bm{R}}] = \qr(\bm{Z}^\top \bm{L}^{n+1})$.
    \item Basis Update: Set $\bm{V}^{n+1} = \bm{Z} \widetilde{\bm{Q}}$. This construction ensures that $\bm{V}^{n+1}$ is orthonormal and strictly satisfies $\bm{1}^\top \bm{M} \bm{V}^{n+1} = \bm{0}^\top$.
\end{enumerate}

\subsection{Energy Stability of the Energy-Consistent Low-Rank method}

The proposed low-rank schemes inherit the same energy stability properties as their full-rank counterparts.
The proof proceeds along similar lines to the analysis in \cite{einkemmer2024asymptoticpreserving} for $P_N$ discretizations.
However, a key difference in the $S_N$ framework is the necessary treatment of quadrature weights, which we discuss in detail in Remark~\ref{rem:weighted-consistency}
Since the stability analysis for different low-rank formulations and time integrators follows the same line of argument, we focus on the low-rank IMEX-S scheme coupled with the BUG integrator as a representative case.
Without loss of generality, we state and prove the corresponding energy stability result for the IMEX-S-BUG scheme below.

\begin{theorem}[Energy $\theta$-stability of the IMEX-S-BUG scheme]
    Assume that $\Phi = 0$ and that the time step size $\Delta t$ satisfies the CFL condition \eqref{eq:IMEX-S-theta-time-step} associated with the IMEX-S scheme.
    Then, the low-rank IMEX-S scheme is $\theta$-stable.
\end{theorem}
\begin{proof}
    First, let us multiply the S-step with $\bm{X}^{n+1}$ from the left and $(\bm{V}^{n+1})^\top$ from the right, and define $\mathcal{P}^X = \bm{X}^{n+1} (\bm{X}^{n+1})^\top$, $\mathcal{P}^V = \bm{V}^{n+1} (\bm{V}^{n+1})^\top$.
    We have
    \begin{equation}\label{eq:proof-lowrank-0}
        \begin{aligned}
            \frac{\bm{G}^{n+1} - \widetilde{\bm{G}}^n}{\Delta t} \bm{M}
            ={} & - \frac{1}{\varepsilon} \mathcal{P}^X \mathcal{A}\left(\widetilde{\bm{G}}^n\right) \left(\mathbf{I} - \frac1{|D_{\bm{\Omega}}|} \bm{w} \bm{1}^\top \right)
            \bm{M} \mathcal{P}^V - \frac{1}{\varepsilon^2} \mathcal{P}^X \mathcal{J}(\bm{\rho}^{n+1}) \bm{M} \mathcal{P}^V
            \\
                & - \frac{1}{\varepsilon^2} \mathcal{P}^X \bm{\sigma}^s \bm{G}^{n+1}\bm{M} \mathcal{P}^V
            - \mathcal{P}^X \bm{\sigma}^a \bm{G}^{n+1} \bm{M}\mathcal{P}^V,
        \end{aligned}
    \end{equation}
    where
    \[
        \widetilde{\bm{G}}^n \bm{M} = \bm{X}^{n+1} \widetilde{\bm{S}}^n (\bm{V}^{n+1})^\top
        = \mathcal{P}^X \bm{G}^n \bm{M} \mathcal{P}^V.
    \]
    Multiplying \eqref{eq:proof-lowrank-0} from the right by $(\prod_{j=1}^d \Delta x^{(j)}) \bm{M}(\bm{G}^{n+1})^\top$ and then taking the trace, we obtain
    {\small\begin{equation}\label{eq:proof-lowrank-1}
        \begin{aligned}
            \frac{1}{\Delta t} \inner*{\bm{G}^{n+1} - \widetilde{\bm{G}}^n, \bm{G}^{n+1}}_w
            ={} & - \frac{1}{\varepsilon} \inner*{\mathcal{P}^X \mathcal{A}\left(\widetilde{\bm{G}}^n\right) \left(\mathbf{I} - \frac1{|D_{\bm{\Omega}}|} \bm{w} \bm{1}^\top \right), \bm{G}^{n+1} \bm{M} (\mathcal{P}^V)^\top \bm{M}^{-1}}_w
            \\
                & - \frac{1}{\varepsilon^2} \inner*{\mathcal{P}^X \mathcal{J}(\bm{\rho}^{n+1}), \bm{G}^{n+1} \bm{M} (\mathcal{P}^V)^\top \bm{M}^{-1}}_w
            \\
                & - \frac{1}{\varepsilon^2} \inner*{\mathcal{P}^X \bm{\sigma}^s \bm{G}^{n+1}, \bm{G}^{n+1} \bm{M} (\mathcal{P}^V)^\top\bm{M}^{-1} }_w
            \\
                & - \inner*{\mathcal{P}^X \bm{\sigma}^a \bm{G}^{n+1}, \bm{G}^{n+1} \bm{M} (\mathcal{P}^V)^\top \bm{M}^{-1}}_w.
        \end{aligned}
    \end{equation}}
    Using $\tr(AB)=\tr(BA)$ and
    \begin{equation*}
        \bm{G}^{n+1} = \mathcal{P}^X  \bm{G}^{n+1} = (\mathcal{P}^X )^\top \bm{G}^{n+1},
        \qquad
        \bm{G}^{n+1} \bm{M} = \bm{G}^{n+1} \bm{M} (\mathcal{P}^V )^\top  = \bm{G}^{n+1} \bm{M}\mathcal{P}^V,
    \end{equation*}
    we have
    \begin{align*}
        \inner*{\mathcal{P}^X \bm{Z}, \bm{G}^{n+1} \bm{M} (\mathcal{P}^V)^\top \bm{M}^{-1}}_w ={} & \tr\left(
        \mathcal{P}^X \bm{Z} \bm{M}^2 (\bm{G}^{n+1} \bm{M} (\mathcal{P}^V)^\top \bm{M}^{-1})^\top
        \right)                                                                                                                     \\
        ={}                                                                                       & \inner*{\bm{Z}, \bm{G}^{n+1}}_w
    \end{align*}
    for any matrix $\bm{Z}$.
    Thus, we reduce \eqref{eq:proof-lowrank-1} to
    \begin{align*}
        \frac{1}{\Delta t} \inner*{\bm{G}^{n+1} - \widetilde{\bm{G}}^n, \bm{G}^{n+1}}_w
        ={} & - \frac{1}{\varepsilon} \inner*{\mathcal{A}\left(\widetilde{\bm{G}}^n\right) \left(\mathbf{I} - \frac1{|D_{\bm{\Omega}}|} \bm{w} \bm{1}^\top \right), \bm{G}^{n+1}}_w
        \\
            & - \frac{1}{\varepsilon^2} \inner*{\mathcal{J}(\bm{\rho}^{n+1}), \bm{G}^{n+1}}_w
        \\
            & - \frac{1}{\varepsilon^2} \norm*{\sqrt{\bm{\sigma}^s} \bm{G}^{n+1}}_w
        - \norm*{\sqrt{\bm{\sigma}^a} \bm{G}^{n+1}}_w.
    \end{align*}
    The remainder of the proof parallels that of Theorem~\ref{thm:IMEX-S-theta-stability}.
    Under the CFL condition \eqref{eq:IMEX-S-theta-time-step}, we have the energy inequality $E_{\theta}^{n+1} \le \widetilde{E}_{\theta}^n$, where
    \[
        \widetilde{E}_{\theta}^n = |D_{\bm{\Omega}}| \norm*{\bm{\rho}^n}^2 + \varepsilon^2 \norm*{\widetilde{\bm{G}}^n}_w^2 + (1-\theta) \Delta t \sigma^s_0 \norm*{\widetilde{\bm{G}}^n}_w^2.
    \]
    By noting that
    \[
        \norm*{\widetilde{\bm{G}}^n}_w^2 = \norm*{\mathcal{P}^X \bm{G}^n \bm{M} \mathcal{P}^V \bm{M}^{-1}}_w^2 \le \norm*{\bm{G}^n}_w^2,
    \]
    it immediately follows that $\widetilde{E}_{\theta}^n \le E_{\theta}^n$. Combining these results yields $E_{\theta}^{n+1} \le E_{\theta}^n$, which completes the proof.
\end{proof}

\begin{remark}\label{rem:weighted-consistency}
    Unlike the $P_N$ method, the $S_N$ discretization defines a discrete energy functional through quadrature weights $\bm{M}^2$. A low-rank approximation $\bm{G} = \bm{X}\bm{S}\bm{V}^\top$ with orthonormal $\bm{V}$ typically fails to preserve this inherent stability.
    To ensure energy consistency, the low-rank ansatz must be aligned with the energy norm, either by (i) employing an energy-consistent factorization $\bm{G}\bm{M} = \bm{X}\bm{S}\bm{V}^\top$ with $\bm{V}^\top\bm{V} = \mathbf{I}_r$, or (ii) enforcing a weighted orthonormality condition $\bm{V}^\top\bm{M}^2\bm{V} = \mathbf{I}_r$ for the original representation.
    While some studies lean toward the latter \cite{einkemmer2021asymptoticpreserving}, stability analysis within the $S_N$ framework remains absent.
    In this work, we adopt the former approach—representing the factors of $\bm{G}\bm{M}$, which enables the seamless integration of robust low-rank integrators while rigorously maintaining the energy stability of the underlying full-rank scheme.
\end{remark}

\section{Numerical results} \label{sec:numerical-results}

In this section, we provide some numerical examples to validate our proposed method.
In the one-dimensional (1D) setting, the spatial domain is $[x_L, x_R]$ and the angular variable is $\mu \in [-1,1]$. In two dimensional (2D) case, the physical domain is a rectangular region $[x_L, x_R] \times [y_B, y_T]$, and the angular variable is discretized by directions $\bm{\Omega} \in \mathbb{S}^2$ projected onto the plane. Periodic boundary conditions are imposed in all numerical experiments.
Schemes such as  IMEX-S-aBUG  combine the IMEX-S time discretization with a BUG-type low-rank integrator. The AP modification in aBUG is activated when required in diffusive regimes with $\sigma^s_0>0$.
For both the one- and two-dimensional tests, the time step sizes are chosen as the maximal values permitted by the stability conditions \eqref{eq:IMEX-time-step} and \eqref{eq:IMEX-S-time-step}.
For the 1D tests, this yields
\begin{align*}
    \Delta t_{(E)} ={}  \frac{2}{3}\,\varepsilon\,\Delta x
    + \frac{1}{3}\,\sigma_0^s\,\Delta x^2, \quad \Delta t_{(I)} ={}
    \left(\frac{\varepsilon}{2\Delta x} - \frac{\sigma_0^s}{4}\right)^{-1}
    \frac{\varepsilon^2}{2},
    \quad
    \left(\frac{\varepsilon}{2\Delta x} > \frac{\sigma_0^s}{4}\right).
\end{align*}
For the 2D tests, the same stability conditions lead to
\begin{align*}
    \Delta t_{(E)} ={}  \frac{1}{3}\,\varepsilon\,\Delta s
    + \frac{1}{12}\,\sigma_0^s\,\Delta s^2, \quad \Delta t_{(I)} ={}
    \left(\frac{\varepsilon}{\Delta s} - \frac{\sigma_0^s}{4}\right)^{-1}
    \frac{\varepsilon^2}{2},
    \quad
    \left(\frac{\varepsilon}{\Delta s} > \frac{\sigma_0^s}{4}\right),
\end{align*}
where $\Delta s = \min\{\Delta x, \Delta y\}$.
Unless otherwise stated, the IMEX, IMEX-BUG, and IMEX-aBUG schemes employ $\Delta t = \Delta t_{(E)}$, while the IMEX-S, IMEX-S-BUG, and IMEX-S-aBUG schemes use $\Delta t = \Delta t_{(I)}$.
In cases where the IMEX-S scheme is unconditionally stable, we use an enlarged time step $\Delta t = 10\,\Delta t_{(E)}$ where $10$ is chosen arbitrarily to illustrate the efficiency of our IMEX-S scheme.
We set the tolerance $\tau = 10^{-5}$ in IMEX-aBUG and IMEX-S-aBUG methods.
All timing results reported in this section are obtained on an AMD Ryzen~9~7940H CPU running at 5.25~GHz, and each measurement is averaged over five independent runs.

\subsection{Gaussian initial data in one dimension}

We investigate the radiation transport equation in slab geometry in the spatial domain $D_x = [-1.5, 1.5]$
using the Gaussian initial condition
\[
    \left.f(t,x,v)\right|_{t=0} = \frac{1}{\sqrt{2\pi} \Sigma} \exp\left(-\frac{x^2}{2 \Sigma^2}\right),
    \quad \Sigma^2 = 9 \times 10^{-4}.
\]
We set $\sigma^s = 1$ and $\sigma^a = \Phi = 0$.
We fix $(N_{\bm{\Omega}}, N_x) = (200, 500)$ for all regimes.
Figure~\ref{fig:Gaussian_combined} displays the results for the kinetic ($\varepsilon=1, T=1, r=50$), intermediate ($\varepsilon=10^{-2}, T=0.2, r=10$), and diffusive ($\varepsilon=10^{-6}, T=0.2, r=3$) regimes, where $r$ denotes the fixed rank in BUG-type methods.
For $\varepsilon=1$, the numerical profiles match the trend of the analytical solution \cite{ganapol2008analytical}, with discrepancies attributed to the low-order nature of the schemes.
As $\varepsilon$ decreases, the methods consistently capture the transitional dynamics with a reduced rank $r$.
In the diffusive limit ($\varepsilon=10^{-6}$), the accurate reproduction of the limiting diffusion solution—even with $\Delta t = 10 \Delta t_{(E)}$ for implicit methods—confirms the AP property. Throughout all cases,
energy stability is guaranteed. $\langle g\rangle_{\bm \Omega}=0$ is preserved at machine precision and we omit the results to save space.

\begin{figure}[htbp]
    \centering
    \begin{subfigure}[c]{0.3\textwidth}
        \centering
        \includegraphics[width=\textwidth]{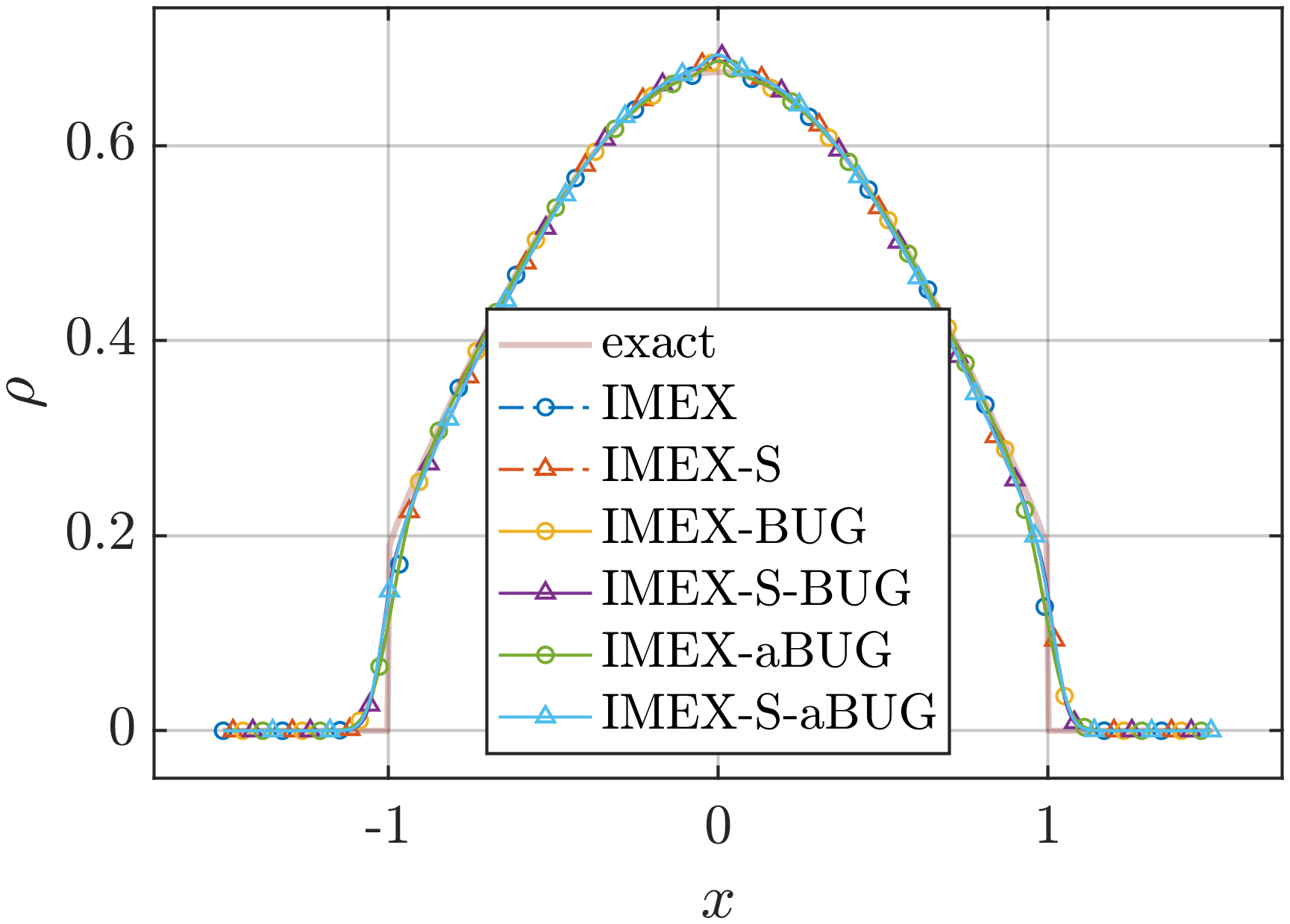}
    \end{subfigure}
    \begin{subfigure}[c]{0.3\textwidth}
        \centering
        \includegraphics[width=\textwidth]{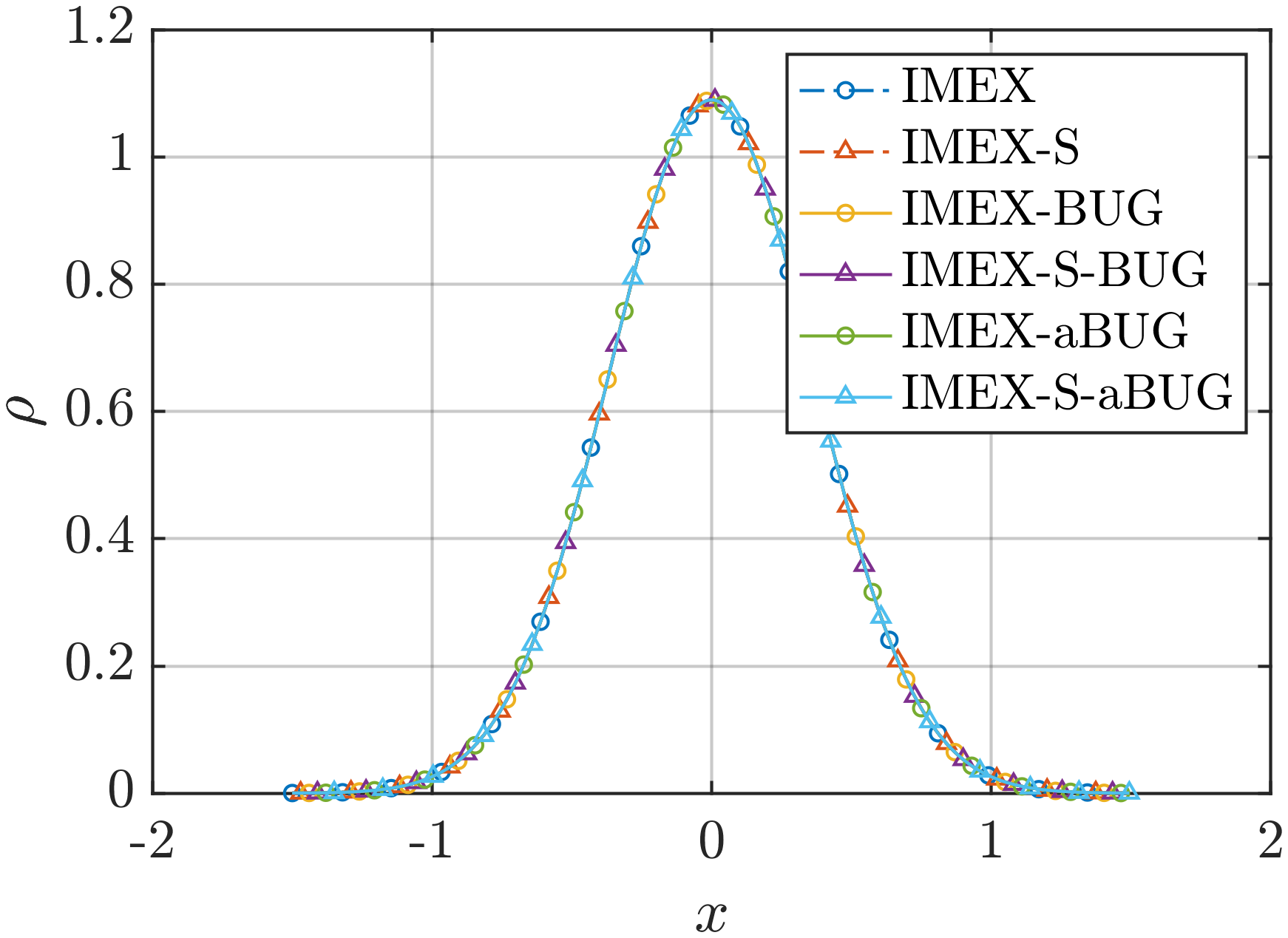}
    \end{subfigure}
    \begin{subfigure}[c]{0.3\textwidth}
        \centering
        \includegraphics[width=\textwidth]{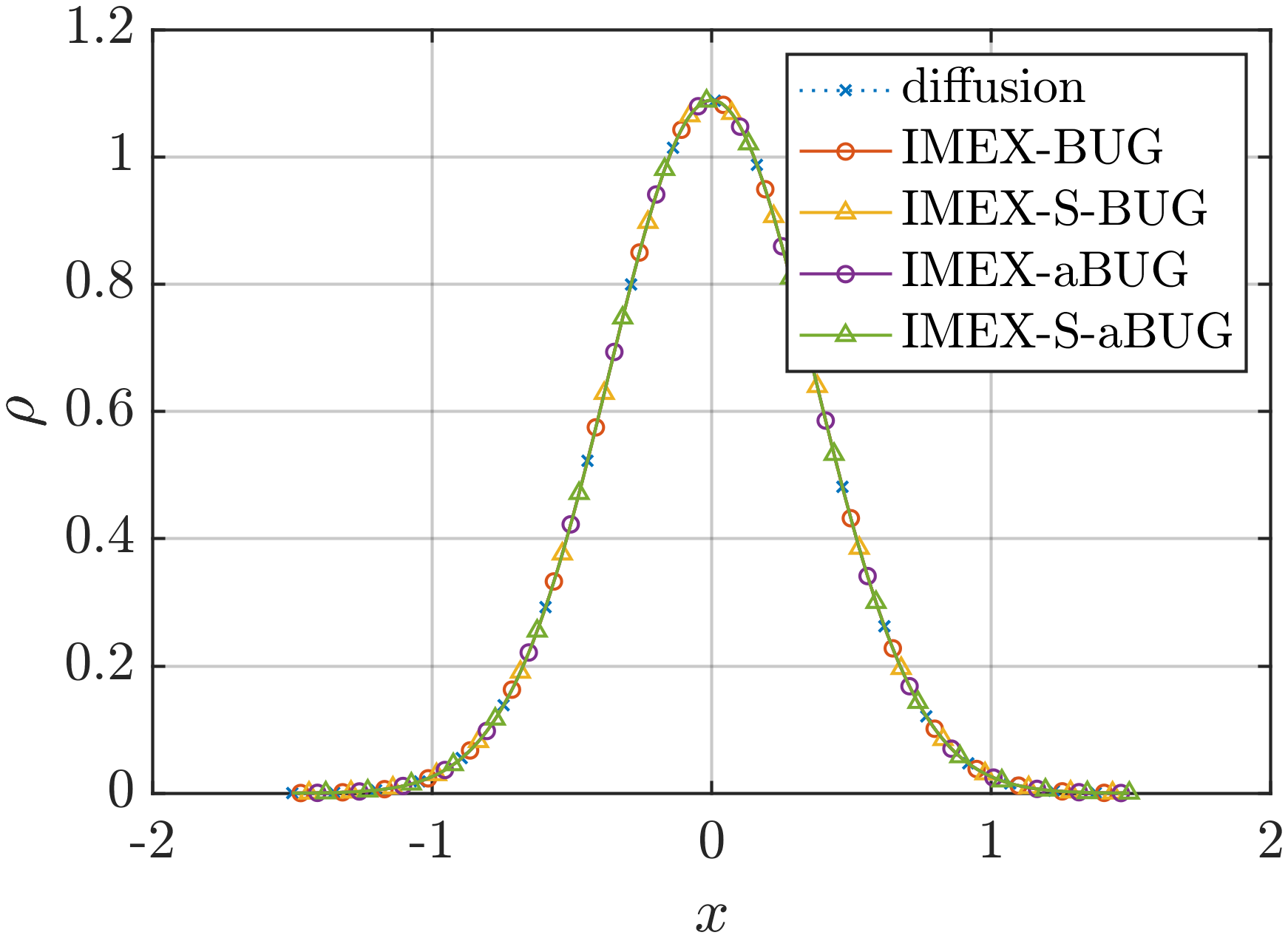}
    \end{subfigure}

    \begin{subfigure}[c]{0.3\textwidth}
        \centering
        \includegraphics[width=\textwidth]{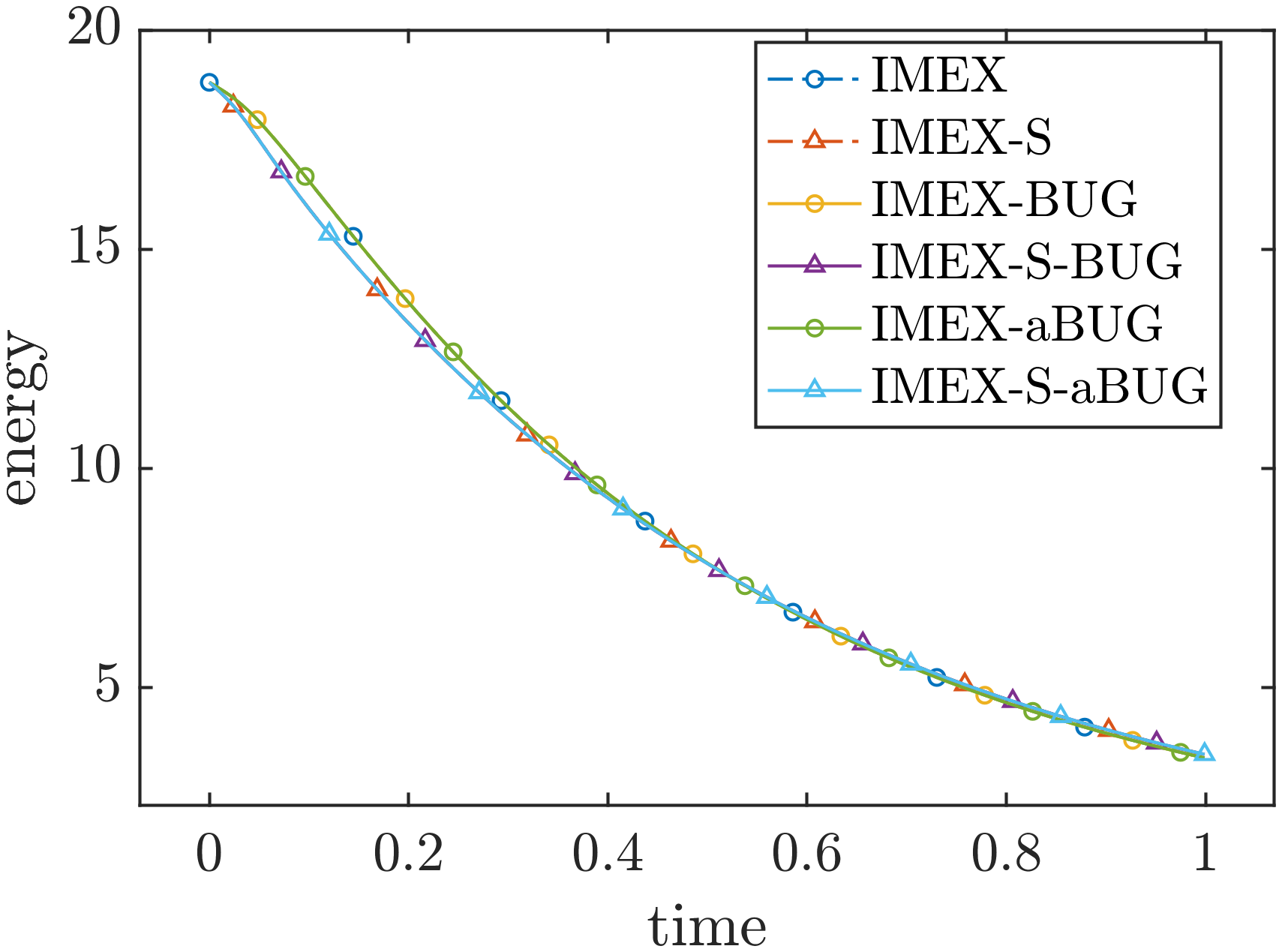}
    \end{subfigure}
    \begin{subfigure}[c]{0.3\textwidth}
        \centering
        \includegraphics[width=\textwidth]{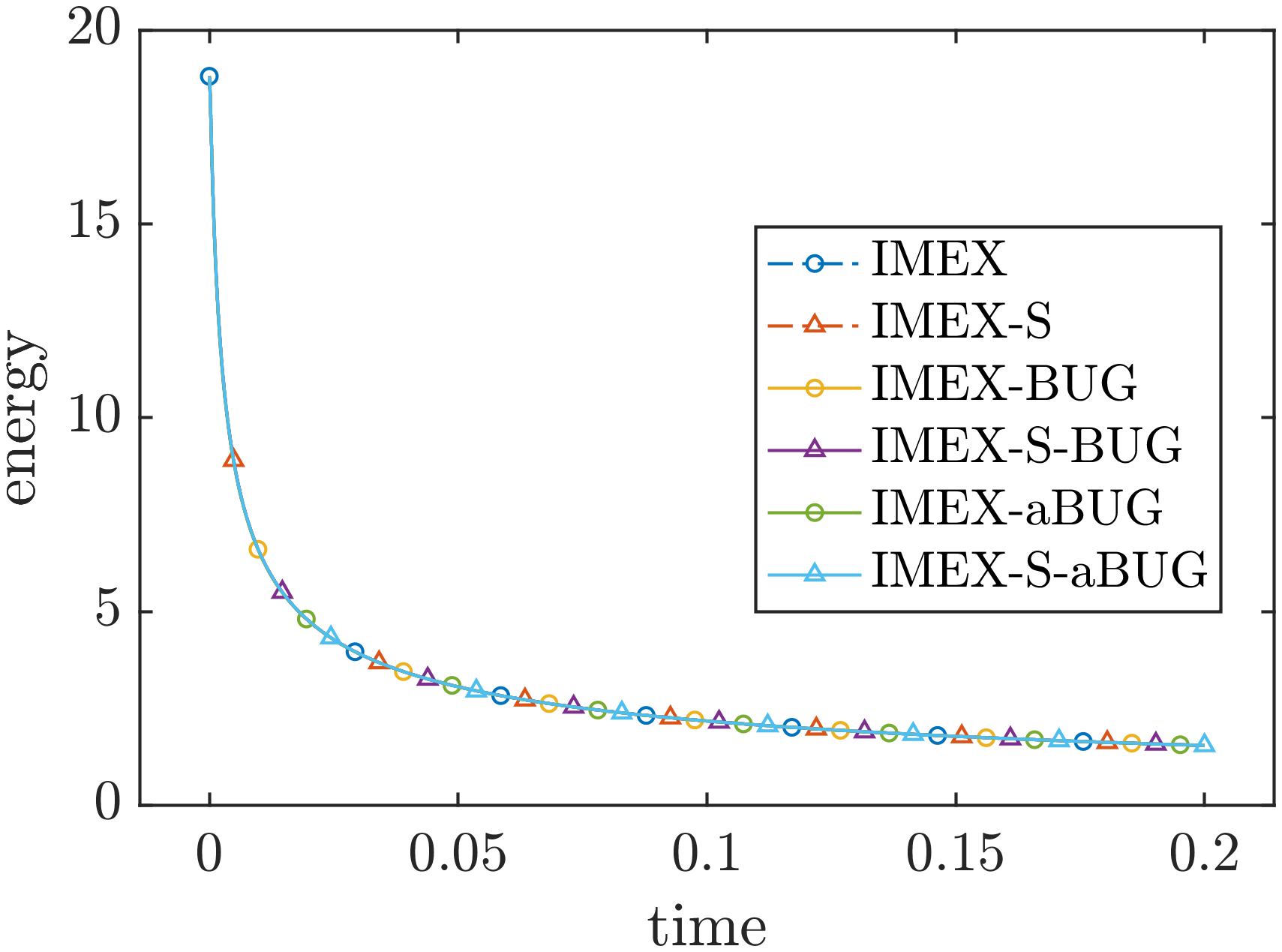}
    \end{subfigure}
    \begin{subfigure}[c]{0.3\textwidth}
        \centering
        \includegraphics[width=\textwidth]{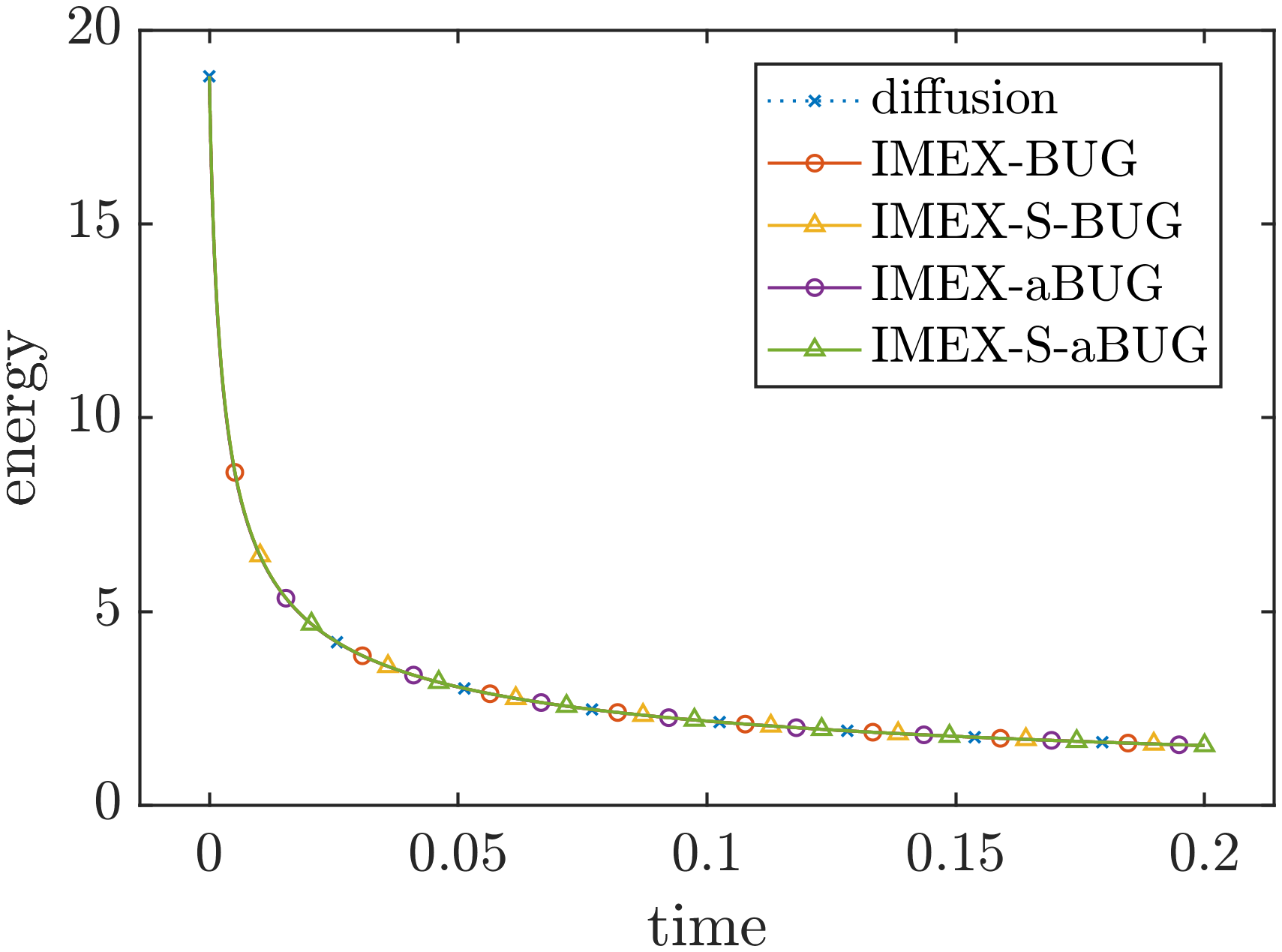}
    \end{subfigure}

    \begin{subfigure}[c]{0.3\textwidth}
        \centering
        \includegraphics[width=\textwidth]{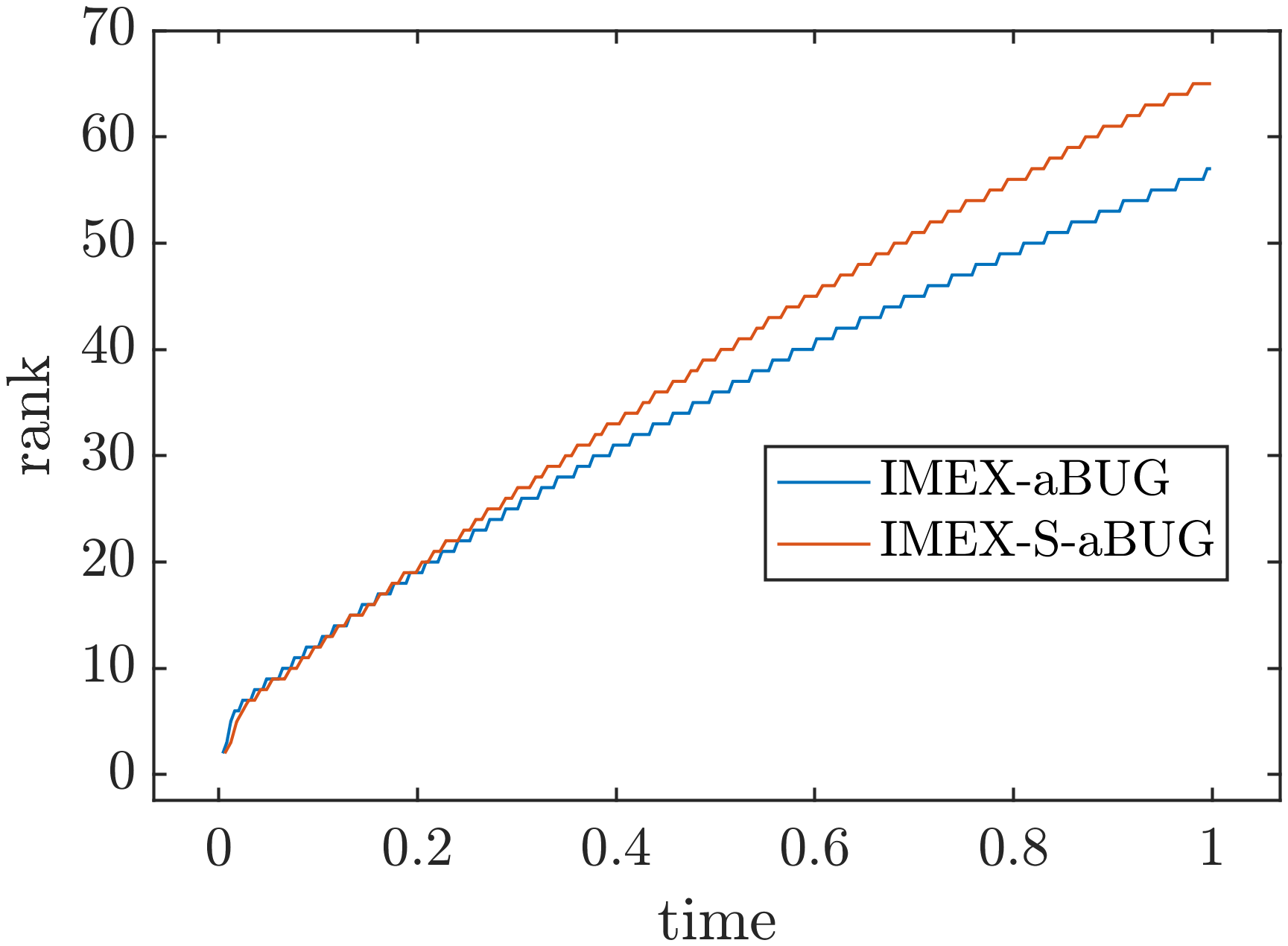}
    \end{subfigure}
    \begin{subfigure}[c]{0.3\textwidth}
        \centering
        \includegraphics[width=\textwidth]{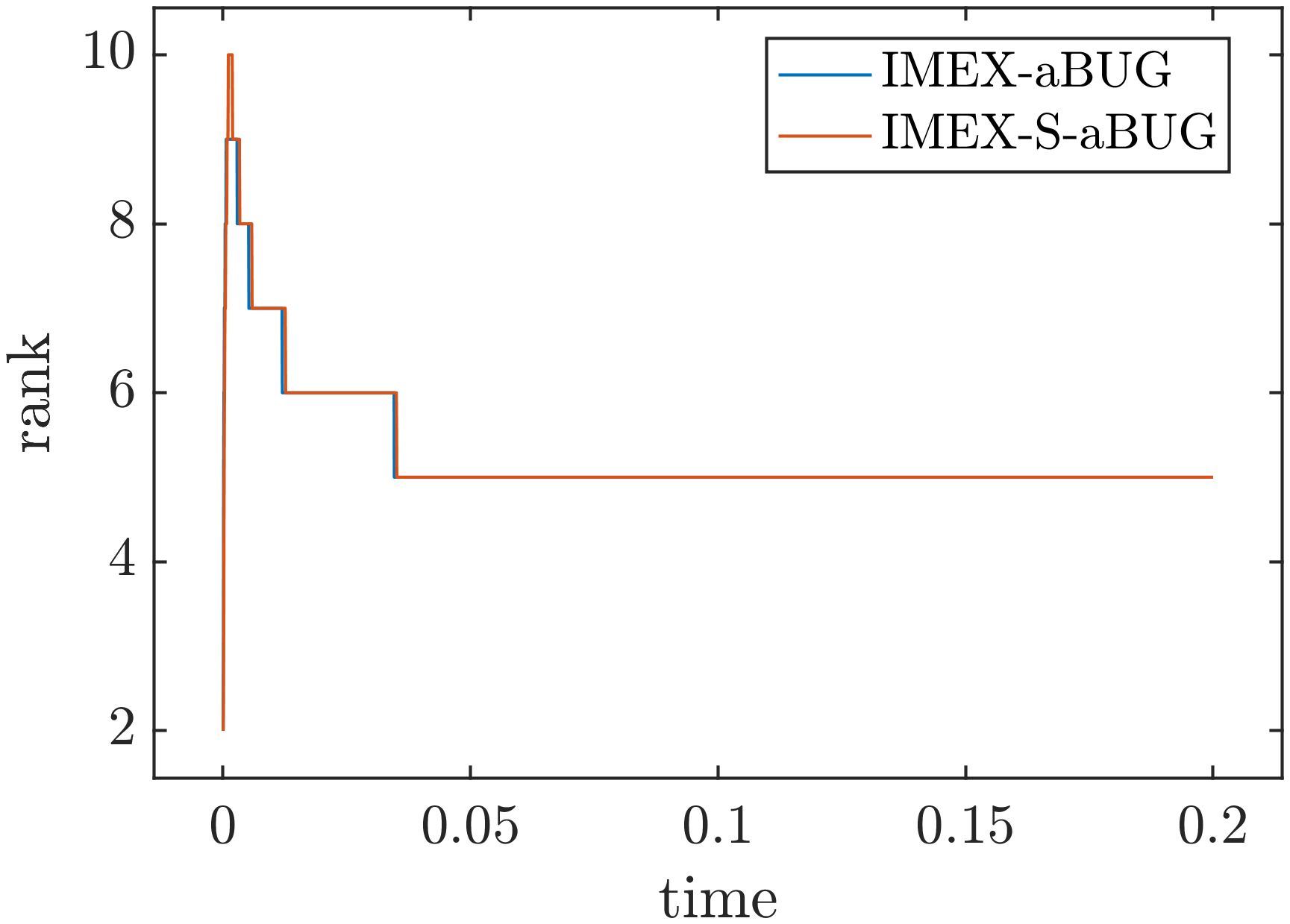}
    \end{subfigure}
    \begin{subfigure}[c]{0.3\textwidth}
        \centering
        \includegraphics[width=\textwidth]{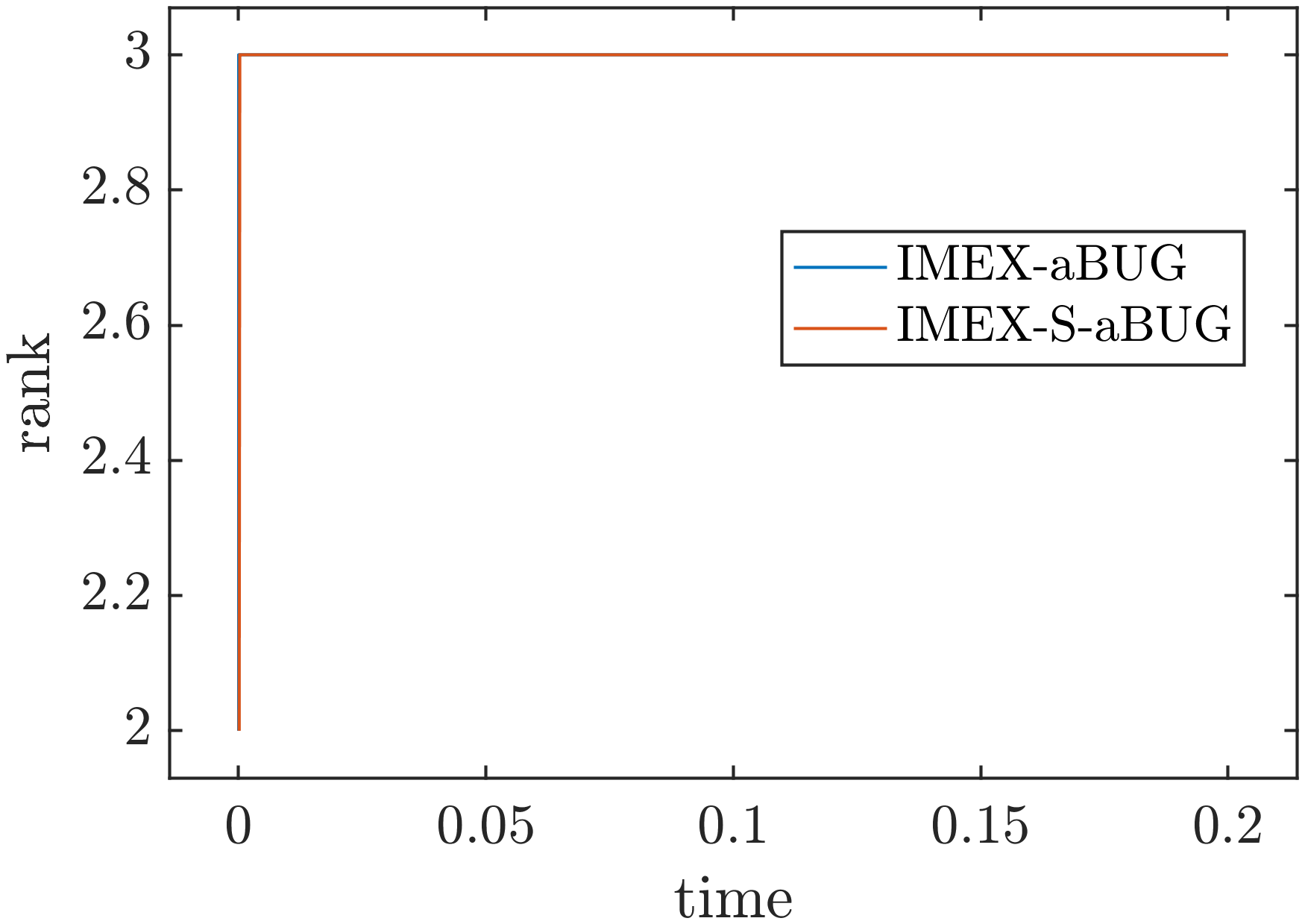}
    \end{subfigure}

    \caption{Gaussian initial data in one dimension.
        Columns from left to right: $(\varepsilon, T) = (1, 1.0)$, $(10^{-2}, 0.2)$, and $(10^{-6}, 0.2)$.
        Rows from top to bottom: density profile, energy evolution, and rank of the microscopic component $\bm{g}$.
    } \label{fig:Gaussian_combined}
\end{figure}

To verify the importance of energy-consistent low-rank factorization presented in Remark~\ref{rem:weighted-consistency}, we initialize the system with a non-equilibrium state:
\[
    \left.f(t,x,v)\right|_{t=0} = \frac{1}{2\pi \Sigma^2} \exp\left(-\frac{x^2}{2 \Sigma^2}\right)
    \left[
        \exp\left(-\frac{(v-1)^2}{2\Sigma^2}\right) + \exp\left(-\frac{(v+1)^2}{2\Sigma^2}\right)
    \right],
\]
where $\Sigma^2 = 10^{-4}$. We set $(N_{\bm{\Omega}}, N_x) = (50, 50)$, $\varepsilon=1$, and $r=2$.
As shown in Figure~\ref{fig:energy_comparison}, the unweighted IMEX-BUG scheme, which only enforces $\mathbf{V}^\top\mathbf{V} = \mathbf{I}_r$, exhibits non-physical growth during the initial transient due to its misalignment with the quadrature weights. In contrast, the weight-consistent IMEX-BUG scheme maintains monotonicity, as the implicit discretization introduces sufficient numerical dissipation to suppress the instability.

\begin{figure}[htbp]
    \centering
    \begin{subfigure}[c]{0.4\textwidth}
        \centering
        \includegraphics[width=\textwidth]{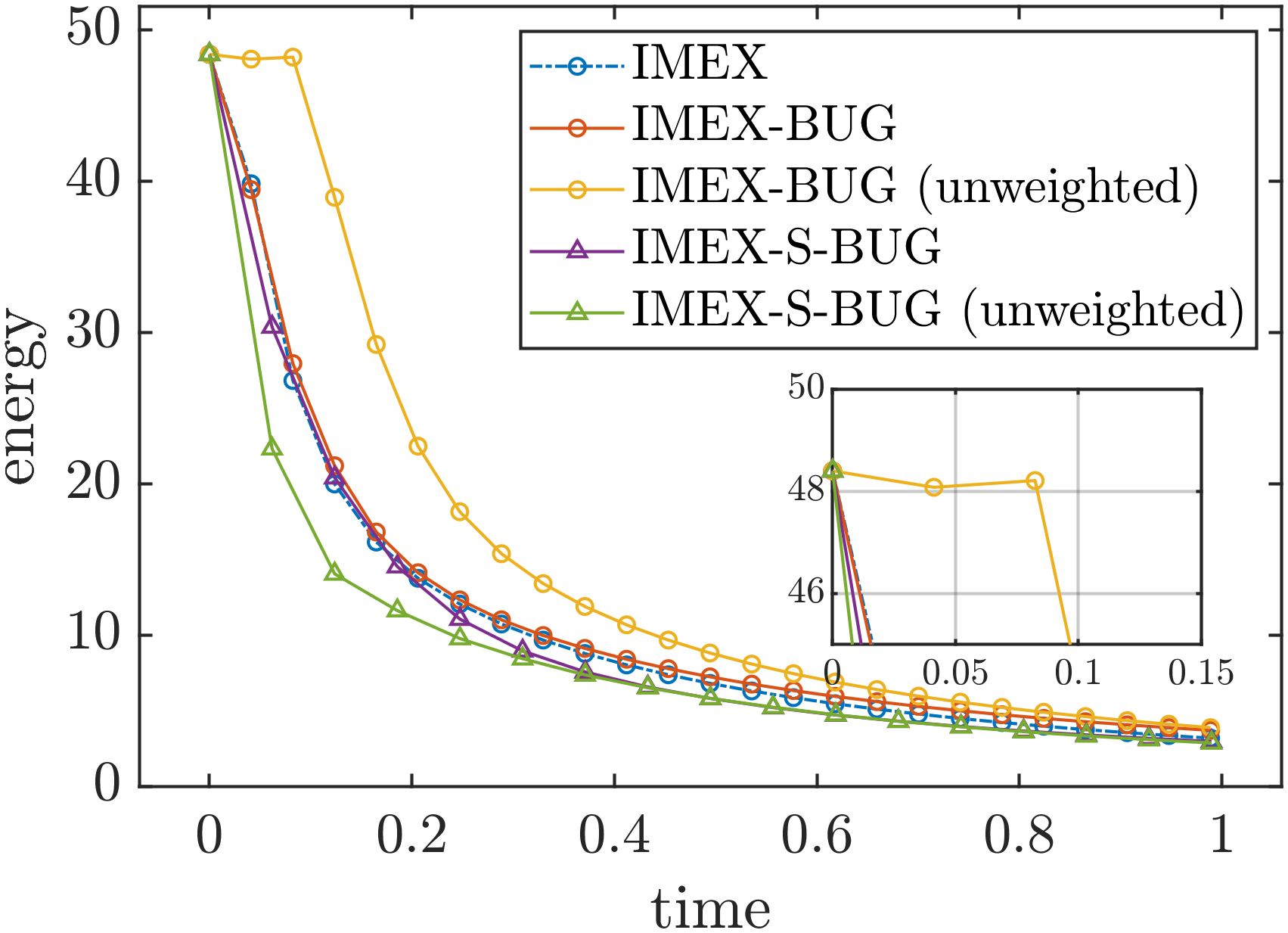}
    \end{subfigure}
    \caption{
    Energy evolution for various low-rank methods.
    The inset reveals that the unweighted IMEX-BUG scheme violates the energy dissipation principle via a non-physical growth.}
    \label{fig:energy_comparison}
\end{figure}

\subsection{Accuracy and computational cost in two dimensions}

We verify the order and the computational cost of the low-rank methods using a low-rank manufactured solution of the form
\[
    f(t,x,y,\theta,\mu) = 2 + e^{-t} \sin(2\pi x) \sin(2\pi y) + \varepsilon\, e^{-t} \sin(2\pi x) \sin(2\pi y) \sin(\theta) \sqrt{1-\mu^2},
\]
where $(x,y) \in [0,1]^2$.
The solution defines a low-rank source $\Phi(t, x, y, \theta, \mu)$ for the linear kinetic transport equation, which can be obtained from a direct calculation.
For the material parameters, we used $\sigma^s(\bm{x}) = 1$ and $\sigma^a(\bm{x}) = 0$.
Simulations are performed to $T=0.1$ with $\Delta t = \Delta t_{(E)}$ for all methods.
We set $r=4$ for BUG and $\tau=10^{-8}$ for aBUG methods. The mesh is refined by doubling $N$ from $16$ to $256$, paired with an $S_{N/8}$ angular discretization.
Figure~\ref{fig:ordertest_2d_combined} presents the $L^2$ error and wall time per step.
As $\varepsilon$ decreases from 1 to $10^{-6}$, the convergence order of $\rho$ improves from first to second order.
Furthermore, the per-step cost of the low-rank schemes scales as $\mathcal{O}(N^2)$, consistent with the theoretical complexity $\mathcal{O}(N^{\max(d_x,d_v)})$ for $r \ll N$, providing a significant reduction over the full-rank cost $\mathcal{O}(N^{d_x+d_v})$.

\begin{figure}[htbp]
    \centering
    \begin{subfigure}[c]{0.3\textwidth}
        \centering
        \includegraphics[width=\textwidth]{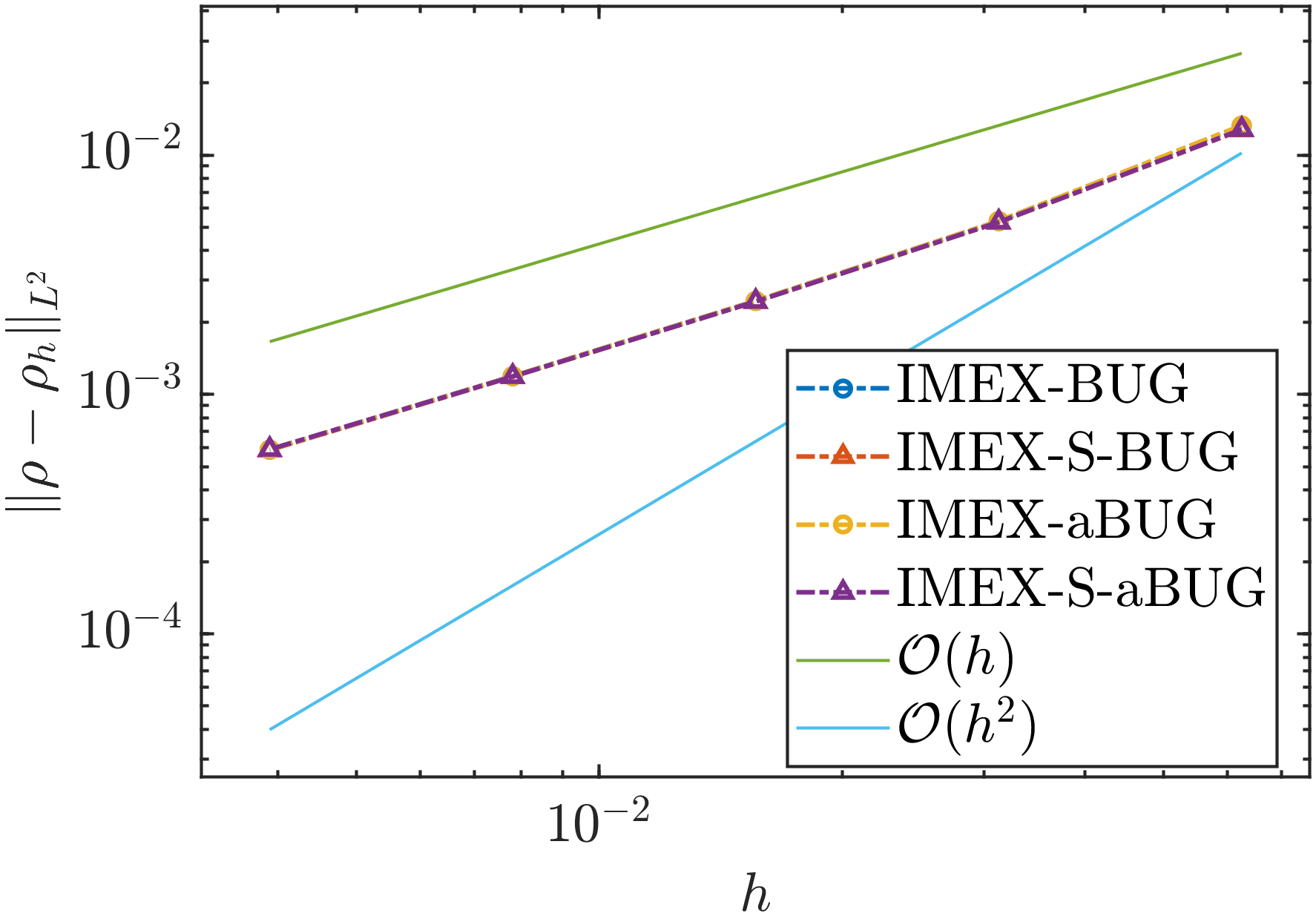}
    \end{subfigure}
    \begin{subfigure}[c]{0.3\textwidth}
        \centering
        \includegraphics[width=\textwidth]{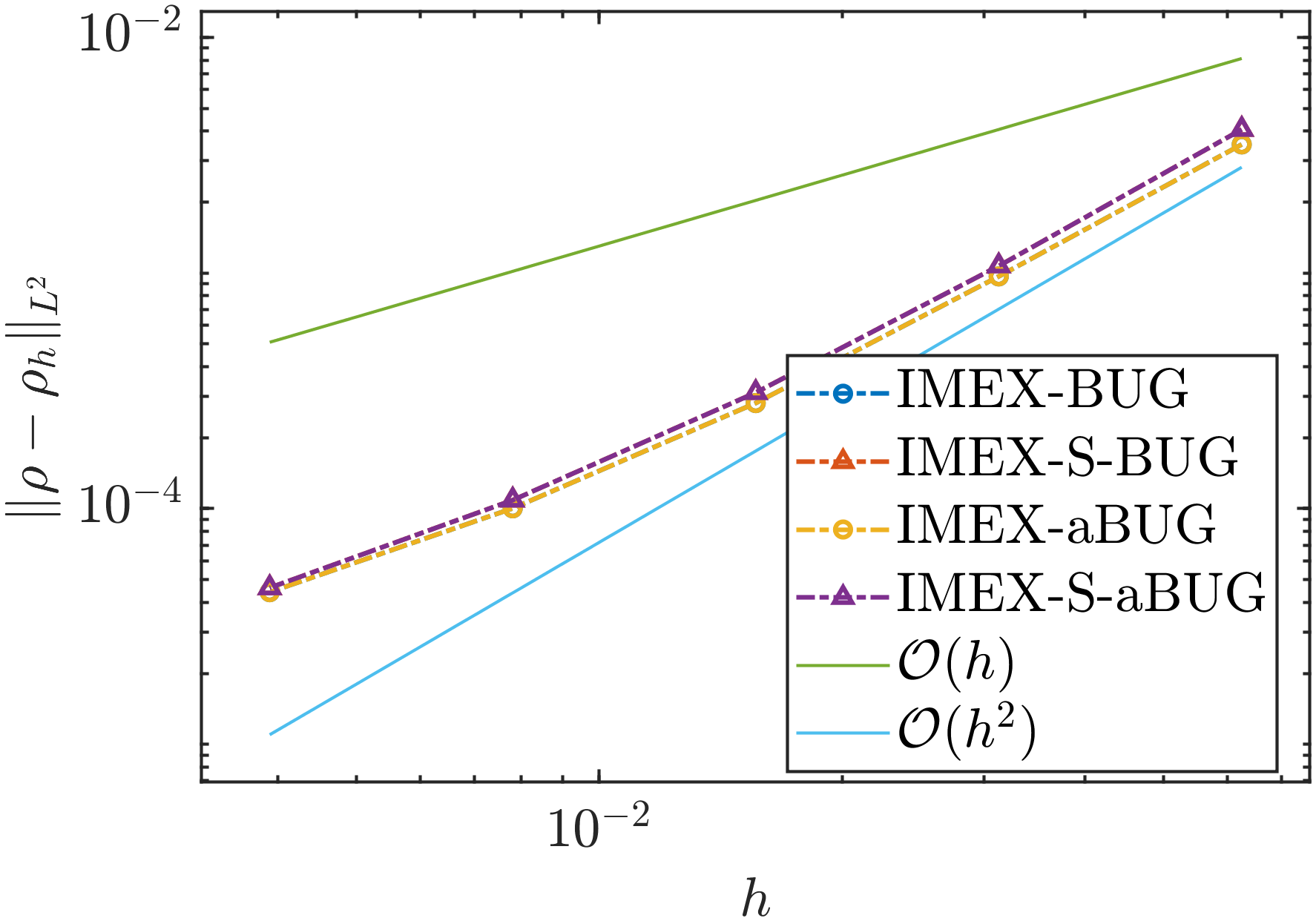}
    \end{subfigure}
    \begin{subfigure}[c]{0.3\textwidth}
        \centering
        \includegraphics[width=\textwidth]{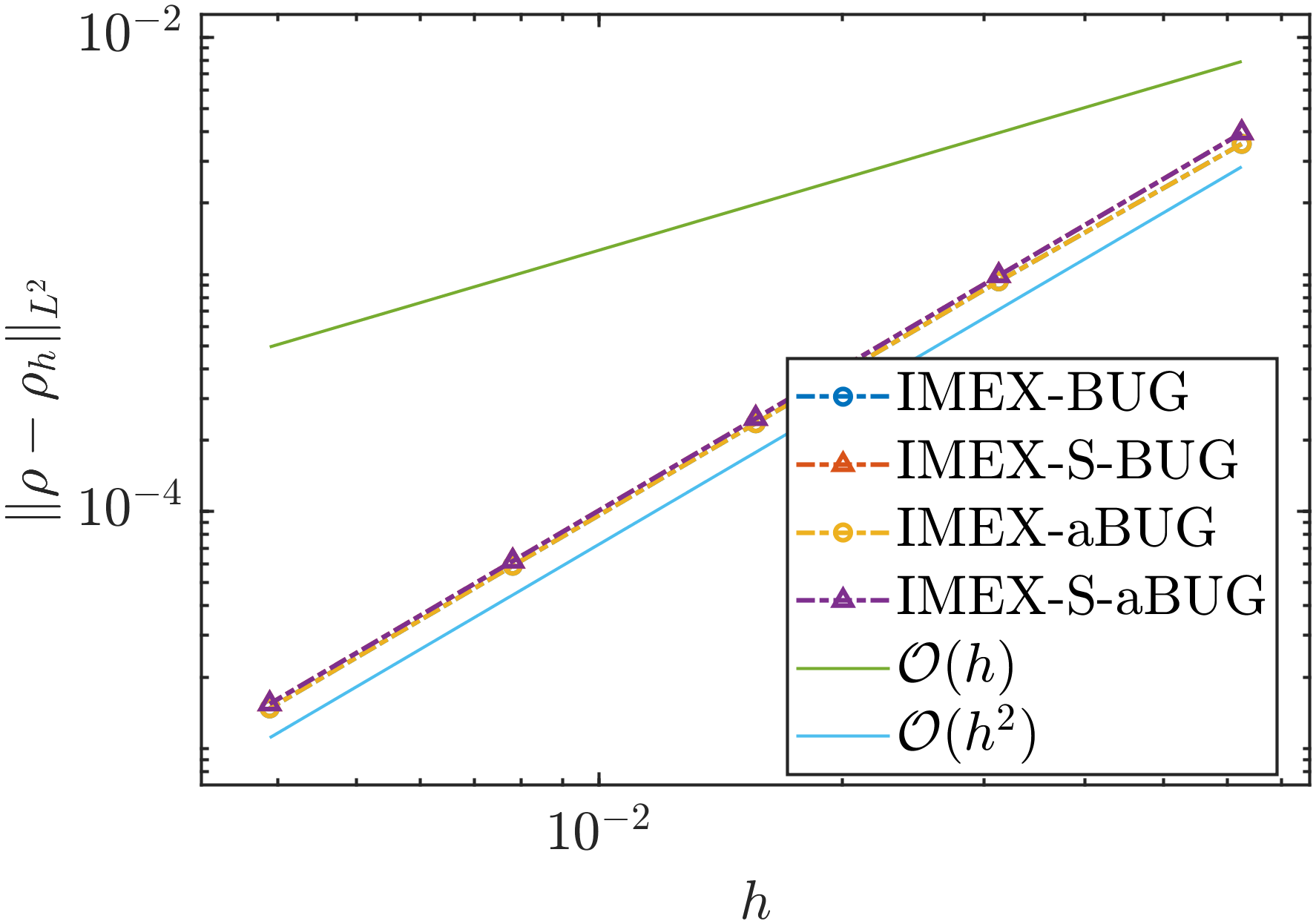}
    \end{subfigure}

    \begin{subfigure}[c]{0.3\textwidth}
        \centering
        \includegraphics[width=\textwidth]{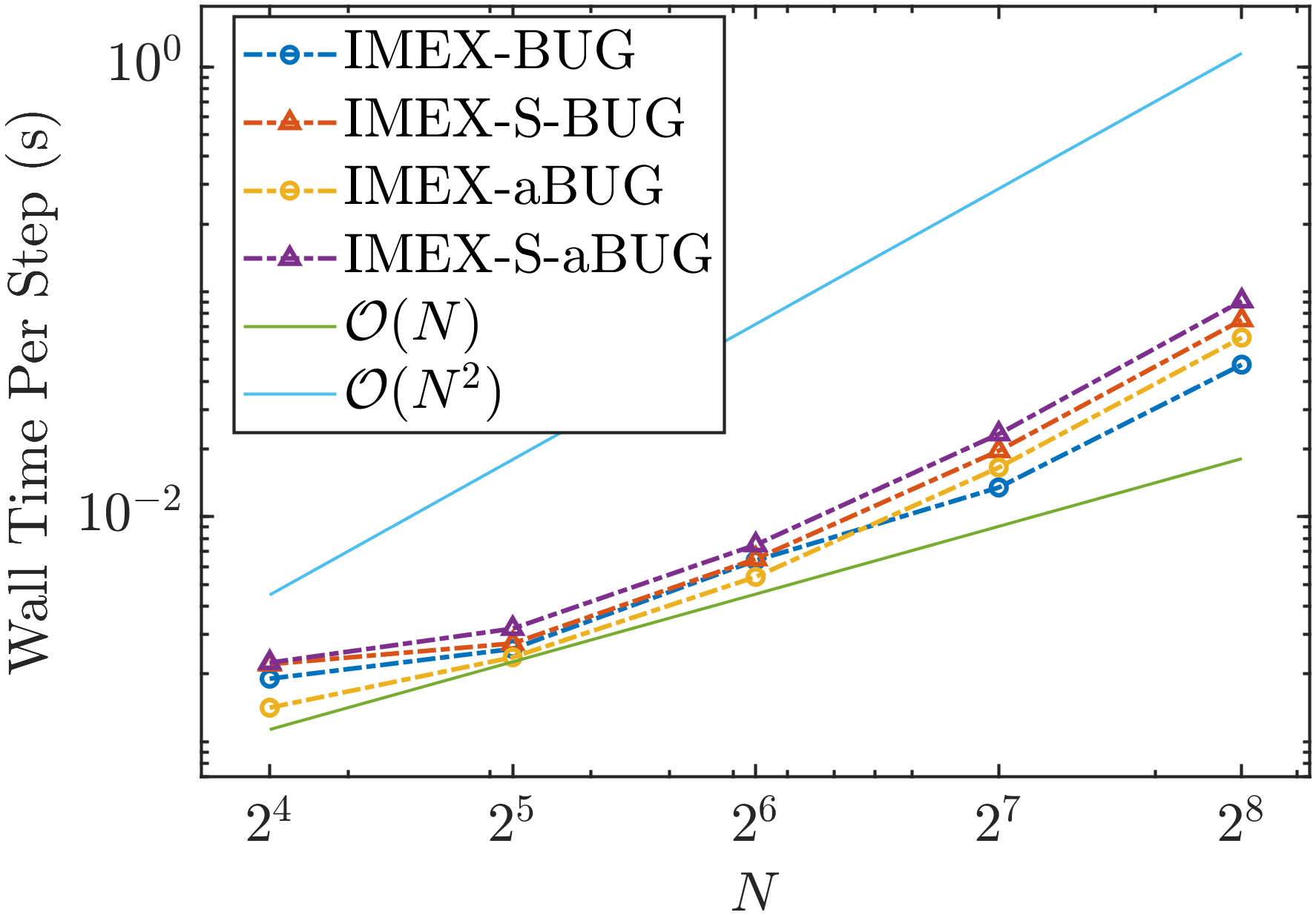}
    \end{subfigure}
    \begin{subfigure}[c]{0.3\textwidth}
        \centering
        \includegraphics[width=\textwidth]{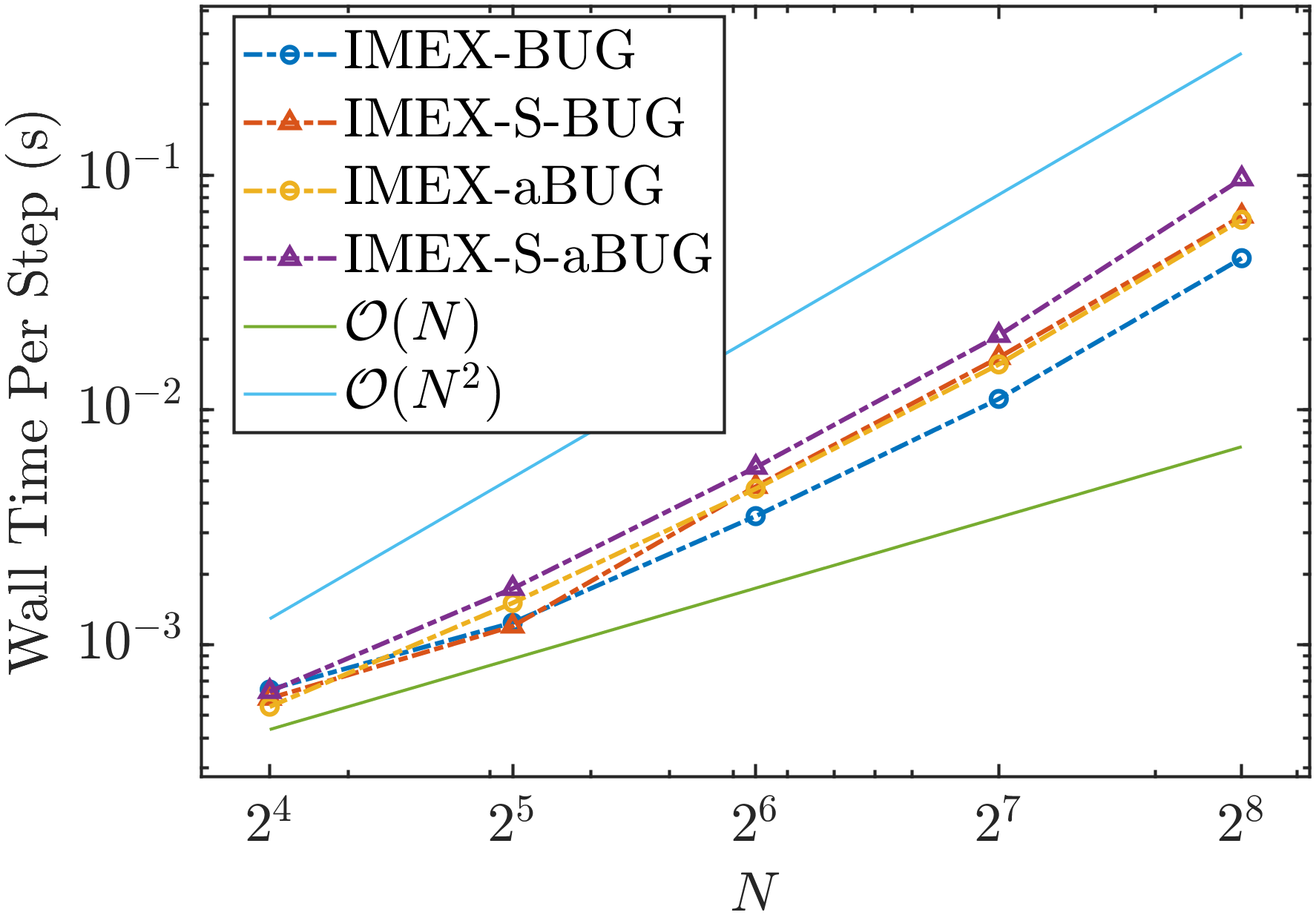}
    \end{subfigure}
    \begin{subfigure}[c]{0.3\textwidth}
        \centering
        \includegraphics[width=\textwidth]{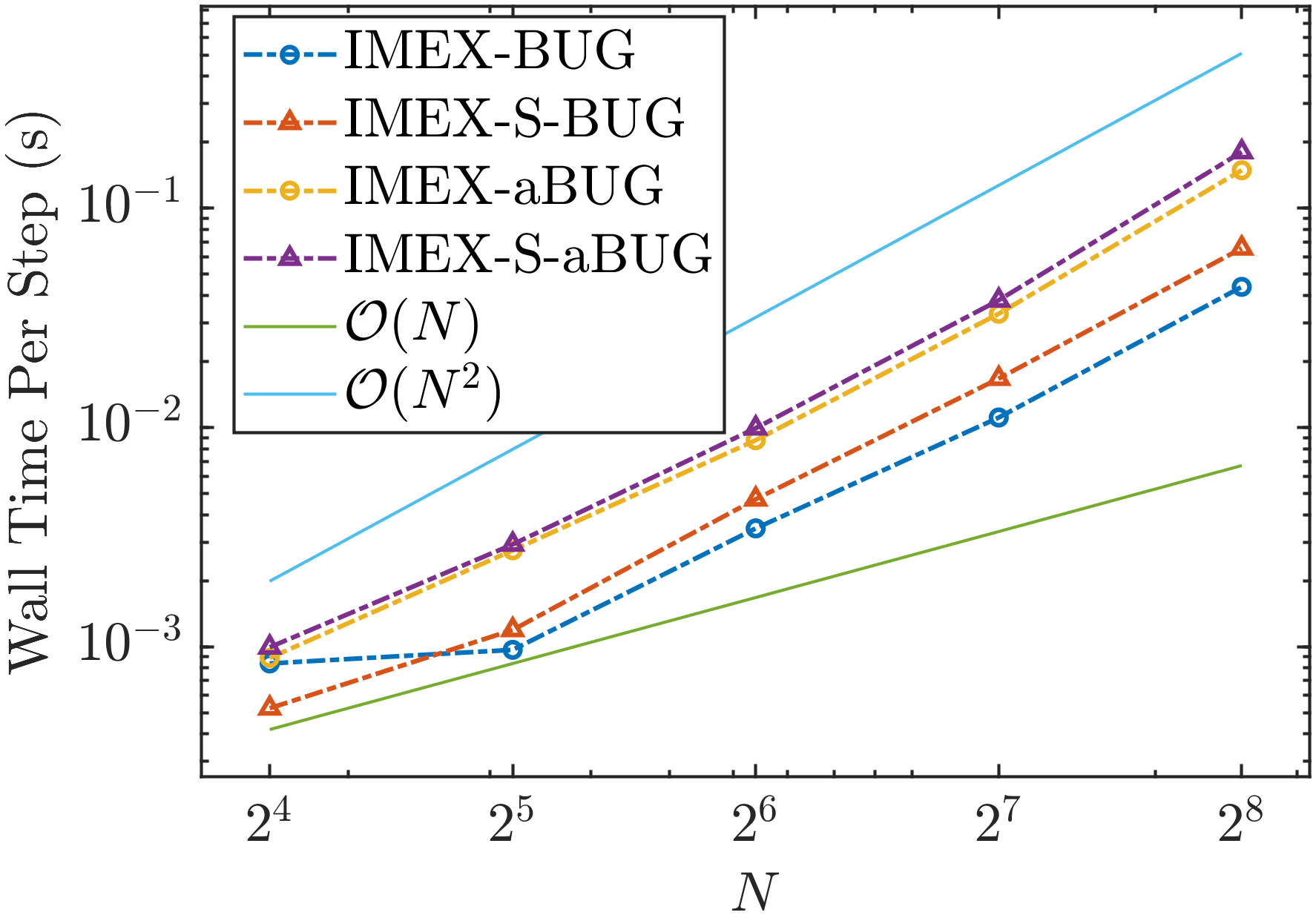}
    \end{subfigure}

    \caption{$L^2$ error (top) and average wall time per step (bottom) for $\varepsilon = 1, 10^{-2}, 10^{-6}$ (left to right).}
    \label{fig:ordertest_2d_combined}
\end{figure}

\subsection{Gaussian initial data in two dimensions}

In this test case, we consider a smooth Gaussian initial condition in $[-1,1]^2$
\[
    f(t,x,y,\theta,\mu) = \frac{1}{4\pi \Sigma^2} \exp\left(- \frac{x^2+y^2}{4\Sigma^2}\right), \quad \Sigma^2 = 10^{-2},
\]
with zero absorption coefficient and source term $\sigma^a = \Phi = 0$.
We use a $128 \times 128$ mesh, $S_{32}$ quadrature, and $r=10$ to simulate the diffusive regime ($\varepsilon = 10^{-6}$) up to $T=0.1$.
As shown by the density slice in Figure~\ref{fig:Gaussian_2d_more}, all schemes match the limiting diffusion solution, confirming the AP property.
We use the time-step size $\Delta t = 0.75 \Delta x^2$ for the diffusion limit and set $\Delta t_{(I)} = 10 \Delta t_{(E)}$.
The computational performance in Table~\ref{tab:Gaussian_2d_timecost} highlights the efficiency of the IMEX-S-(a)BUG schemes: their unconditional stability allows for large time steps, significantly reducing the total wall time compared to IMEX-(a)BUG schemes.

\begin{table}[htbp]
    \centering
    \caption{Gaussian initial data in two dimensions ($\varepsilon = 10^{-6}$). Wall time cost.} \label{tab:Gaussian_2d_timecost}
    \begin{tabular}{ccccc}
        \toprule
        Method      & Total (s) & Per step (s) & $\Delta t$ \\
        \midrule
        IMEX-BUG    & 135.20    & 2.75e-02     & 2.04e-05   \\
        IMEX-S-BUG  & 16.83     & 3.42e-02     & 2.04e-04   \\
        IMEX-aBUG   & 252.79    & 5.14e-02     & 2.04e-05   \\
        IMEX-S-aBUG & 27.64     & 5.62e-02     & 2.04e-04   \\
        \bottomrule
    \end{tabular}
\end{table}

\begin{figure}[htbp]
    \centering
    \begin{subfigure}[c]{0.3\textwidth}
        \centering
        \includegraphics[width=\textwidth]{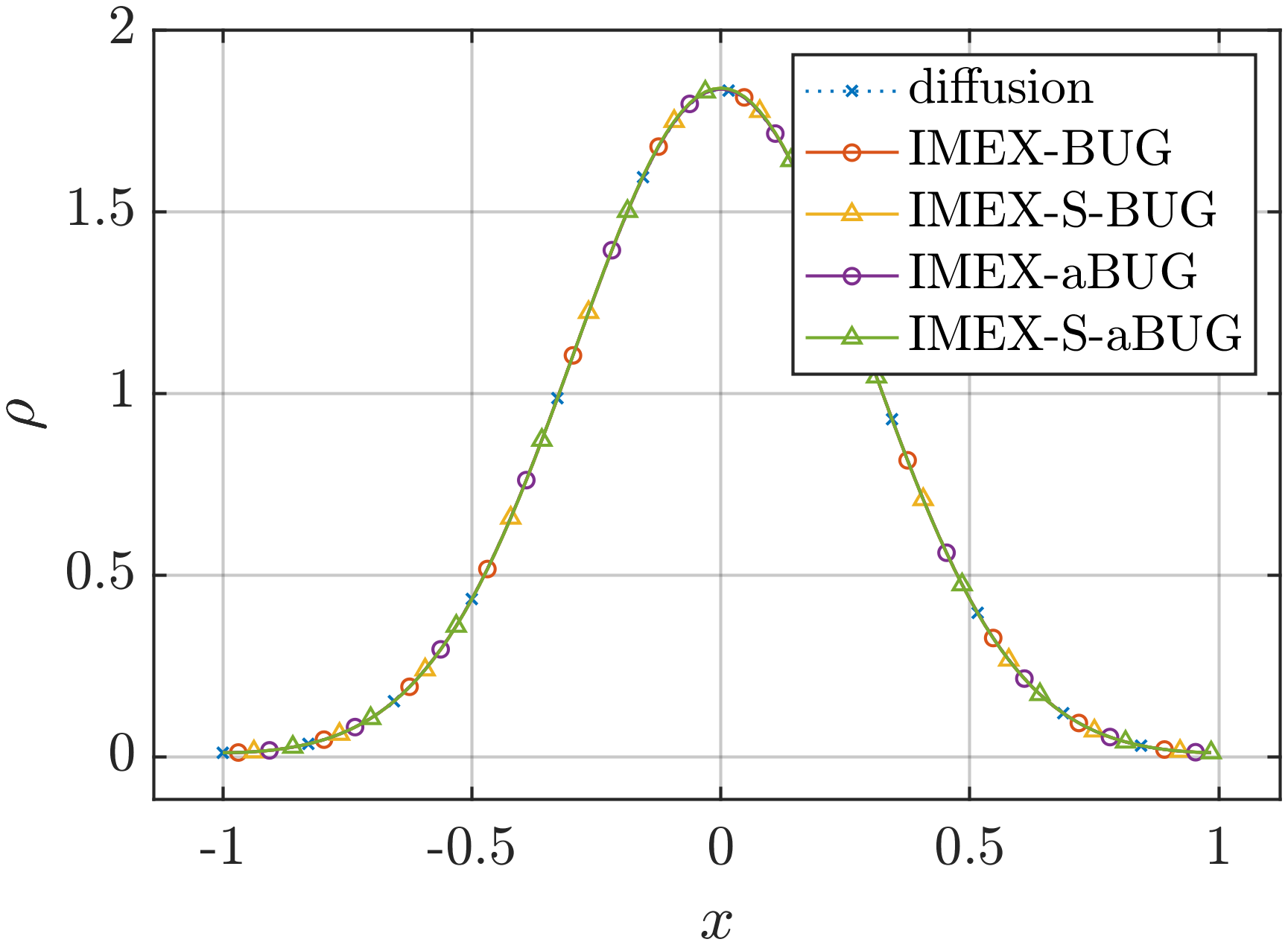}
    \end{subfigure}
    \begin{subfigure}[c]{0.3\textwidth}
        \centering
        \includegraphics[width=\textwidth]{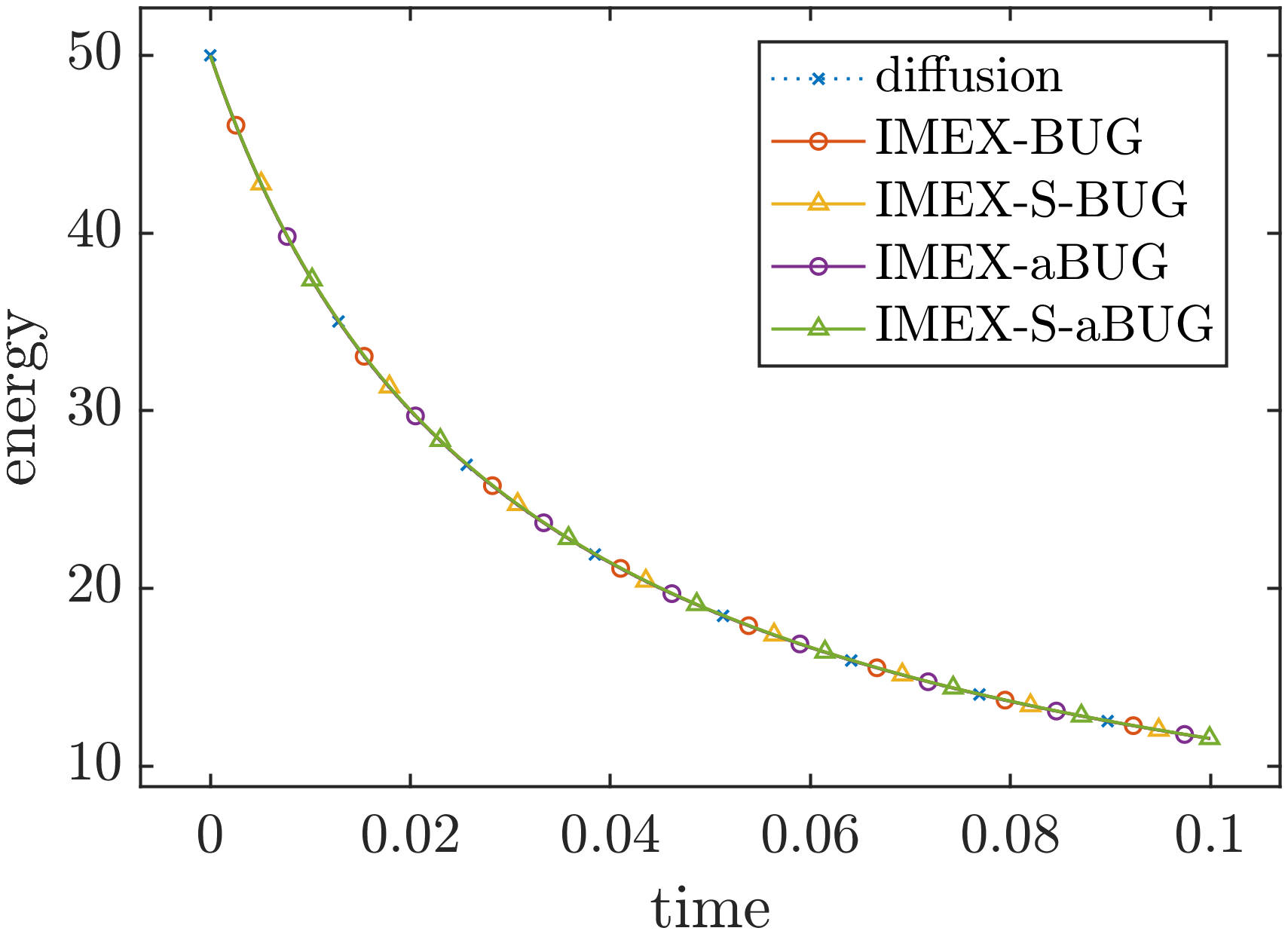}
    \end{subfigure}
    \begin{subfigure}[c]{0.3\textwidth}
        \centering
        \includegraphics[width=\textwidth]{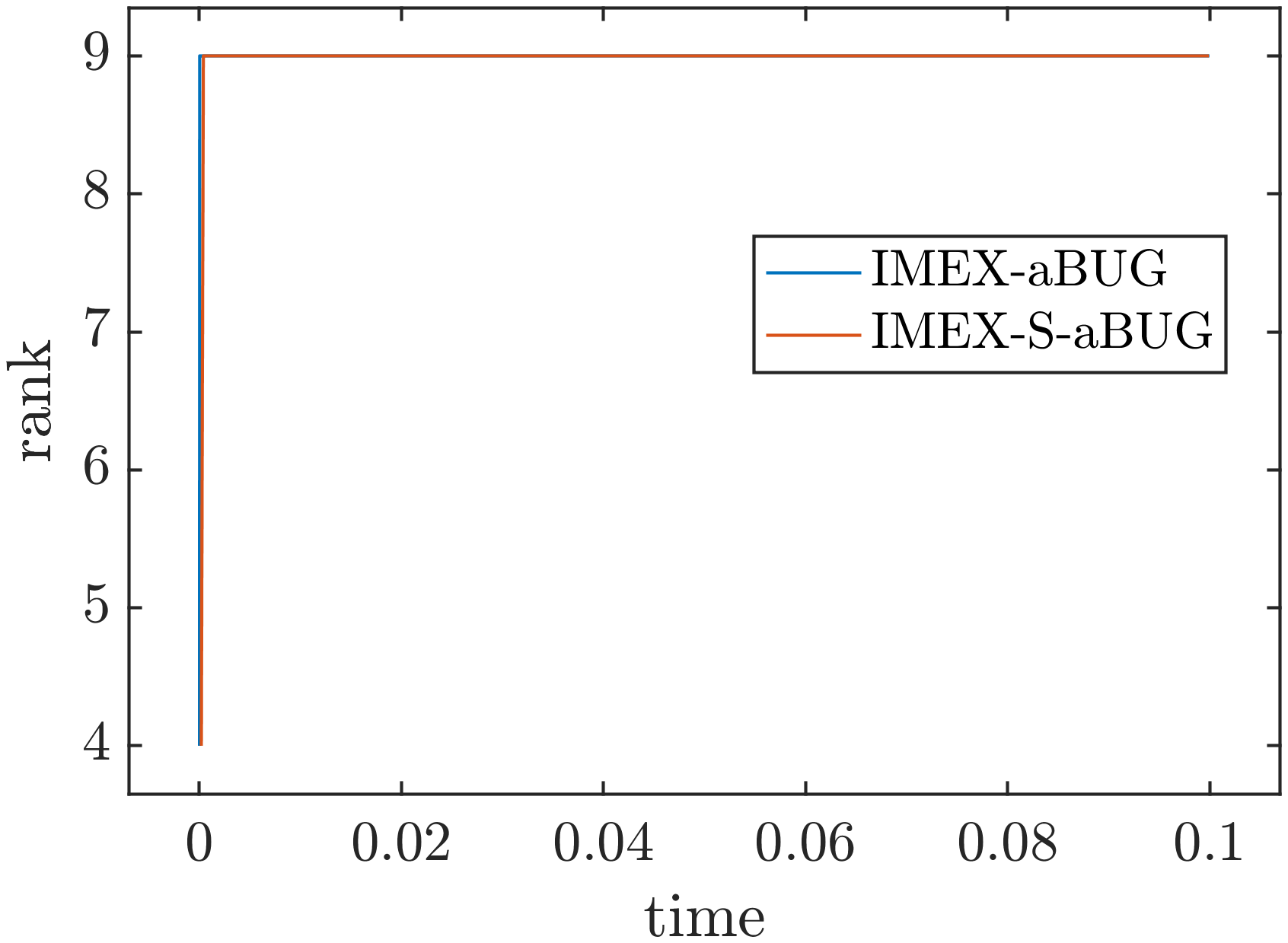}
    \end{subfigure}
    \caption{Gaussian initial data in two dimensions ($\varepsilon = 10^{-6}$).
        Density slice of solutions along $y = 0$ (left) at time $T = 0.1$,
        energy dissipation (middle),
        and rank of the microscopic component $\bm{g}$ of IMEX-S-BUG and IMEX-S-aBUG method (right).
    } \label{fig:Gaussian_2d_more}
\end{figure}

\subsection{Lattice problem}

We simulate a simplified fuel rod assembly on $[0,7]^2$ featuring checkerboard-patterned absorbing and scattering regions as defined in Figure~\ref{fig:lattice_2d_setup}.
The initial condition is given by
\[
    f(t,x,y,\theta,\mu) = \frac{1}{4\pi \Sigma^2} \exp\left(- \frac{(x-3.5)^2+(y-3.5)^2}{4\Sigma^2}\right), \quad \Sigma^2 = 10^{-2}.
\]
Using a $128 \times 128$ spatial mesh, $S_{32}$ quadrature, and $r=100$, the methods are tested in the kinetic regime ($\varepsilon=1, T=2.0$). As illustrated in Figure~\ref{fig:lattice_2d_base} and Figure~\ref{fig:lattice_2d_more}, both the IMEX and IMEX-S BUG schemes have comparable results, only the IMEX-S schemes exhibit slightly larger numerical dissipation; but this discrepancy can be effectively mitigated by employing a little smaller time step $\Delta t$.

\begin{figure}[htbp]
    \centering
    \begin{subfigure}[c]{0.3\textwidth}
        \centering
        \includegraphics[width=\textwidth]{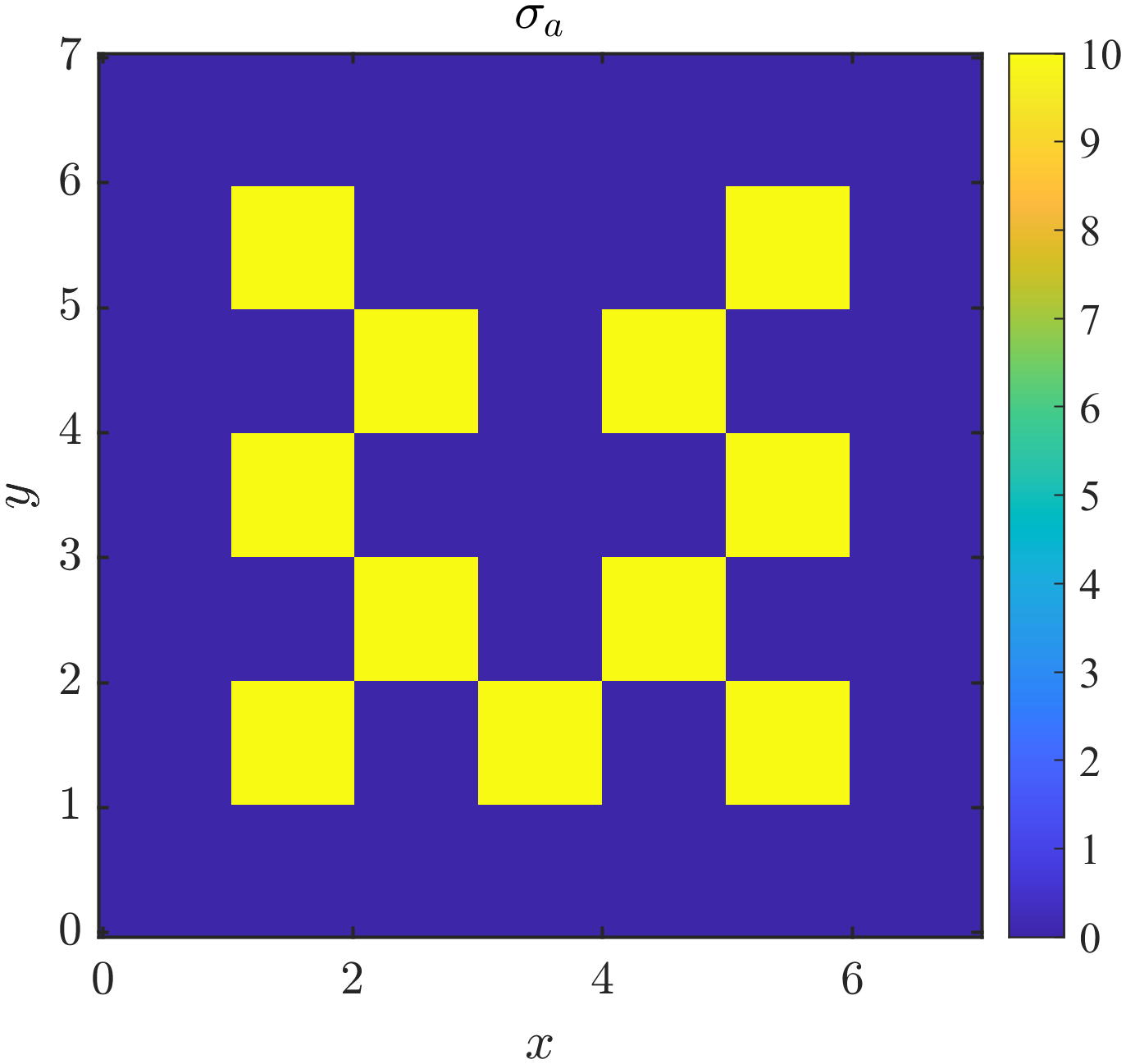}
    \end{subfigure}
    \begin{subfigure}[c]{0.3\textwidth}
        \centering
        \includegraphics[width=\textwidth]{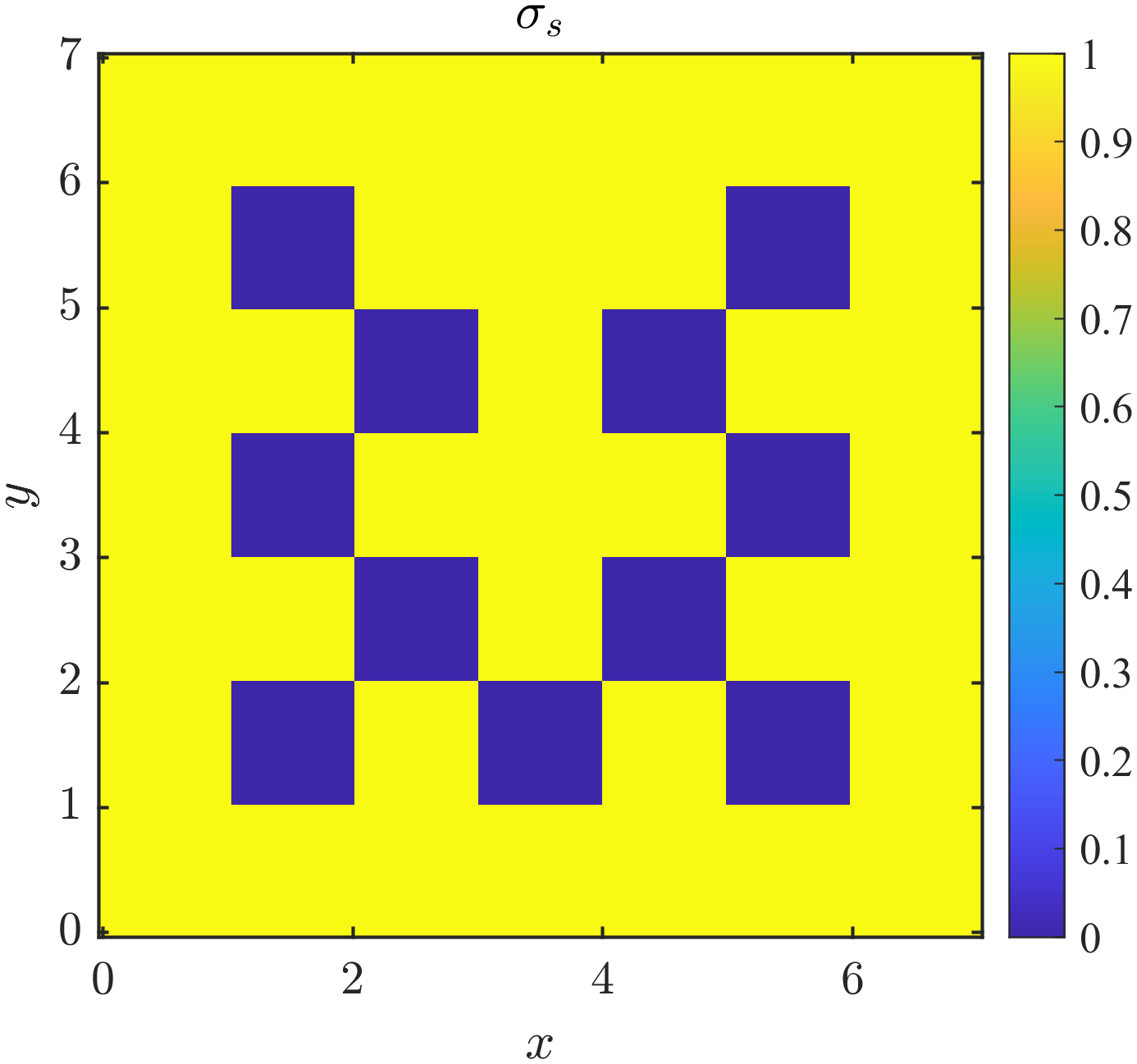}
    \end{subfigure}
    \begin{subfigure}[c]{0.3\textwidth}
        \centering
        \includegraphics[width=\textwidth]{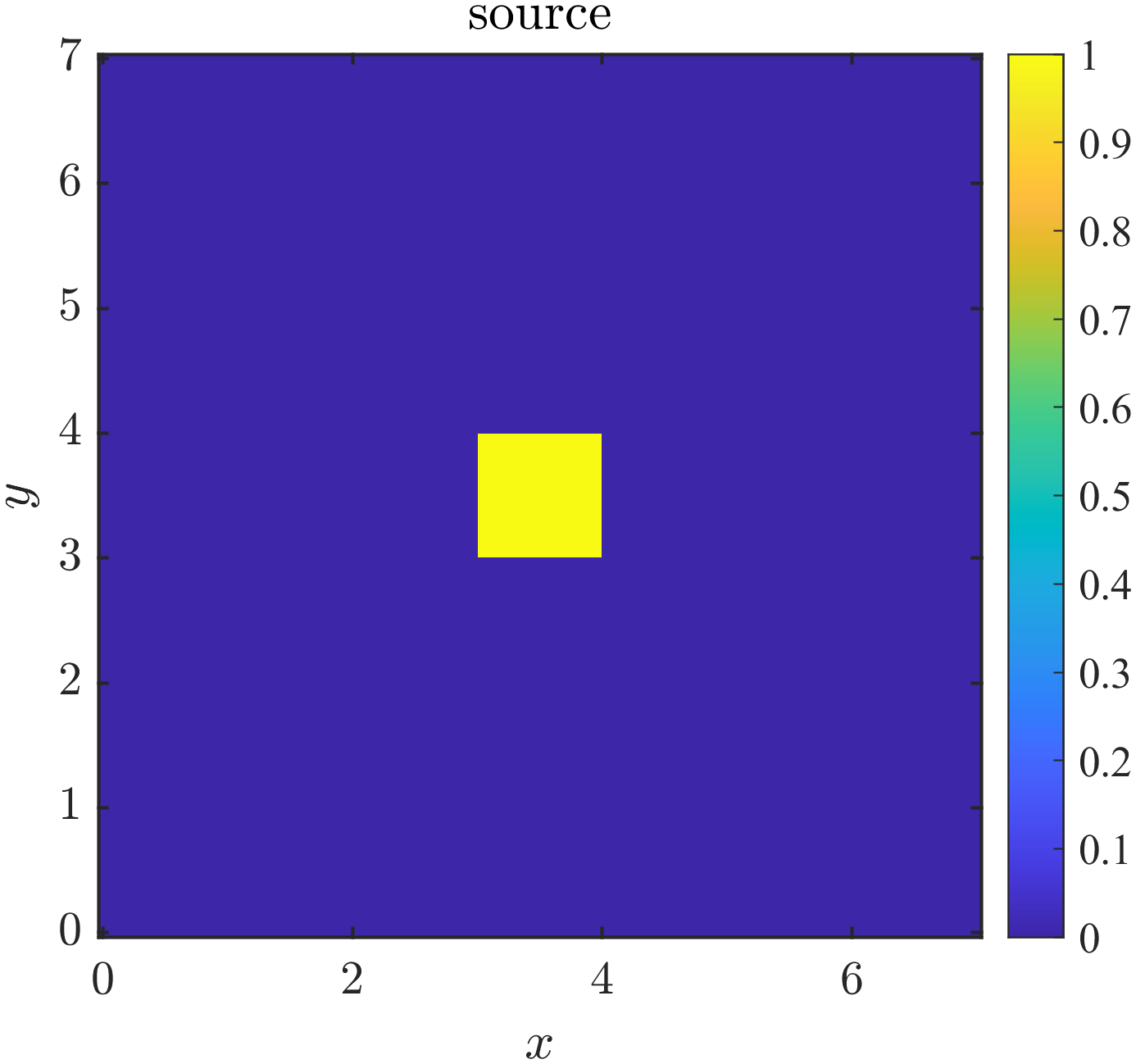}
    \end{subfigure}
    \caption{Material data used in the lattice example.
        Profiles of absorption coefficient $\sigma^a$ (left), scattering coefficient $\sigma^s$ (middle),
        and source function $\Phi$ (right).
    } \label{fig:lattice_2d_setup}
\end{figure}

\begin{figure}[htbp]
    \centering
    \begin{subfigure}[c]{0.3\textwidth}
        \centering
        \includegraphics[width=\textwidth]{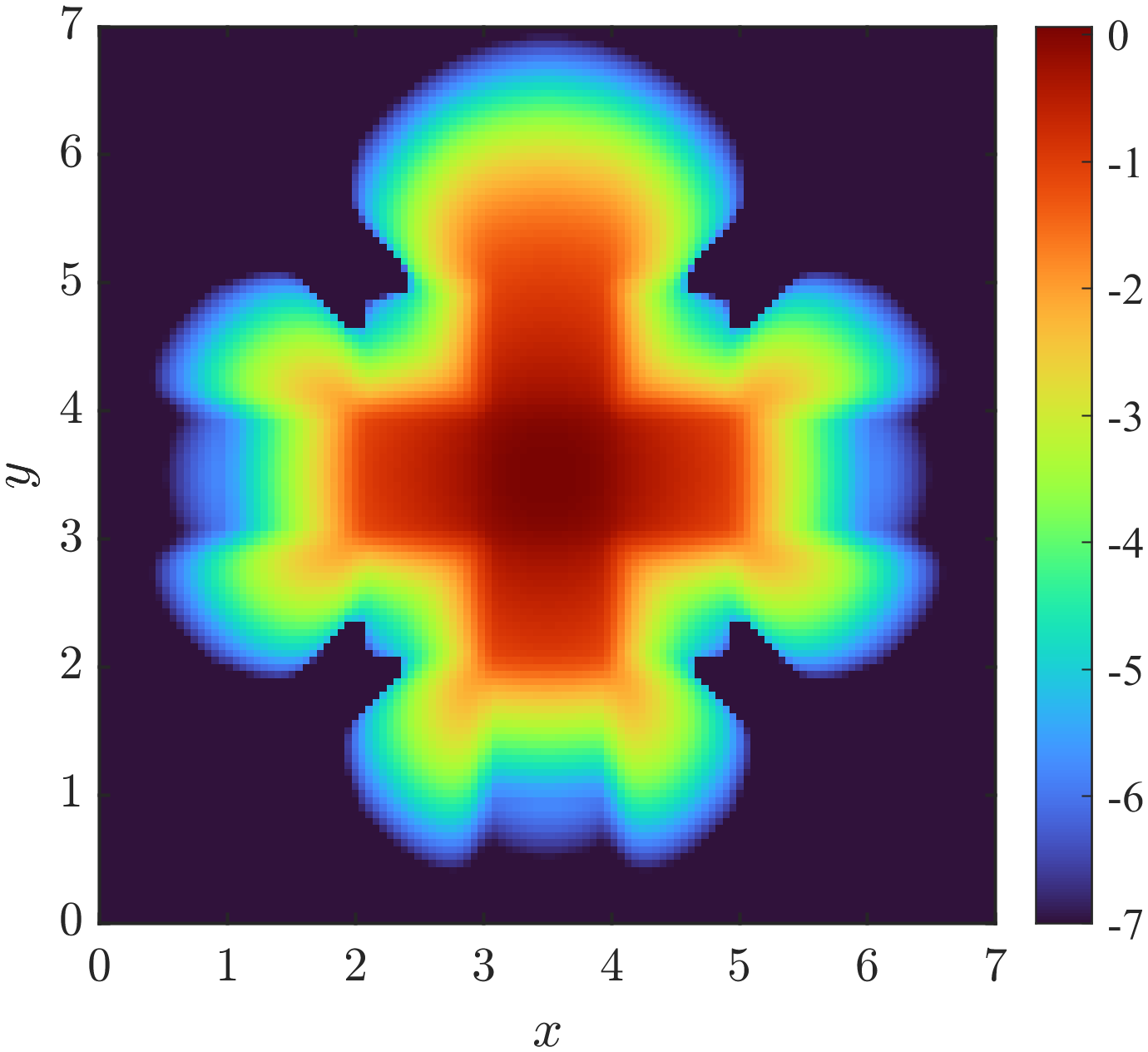}
    \end{subfigure}
    \begin{subfigure}[c]{0.3\textwidth}
        \centering
        \includegraphics[width=\textwidth]{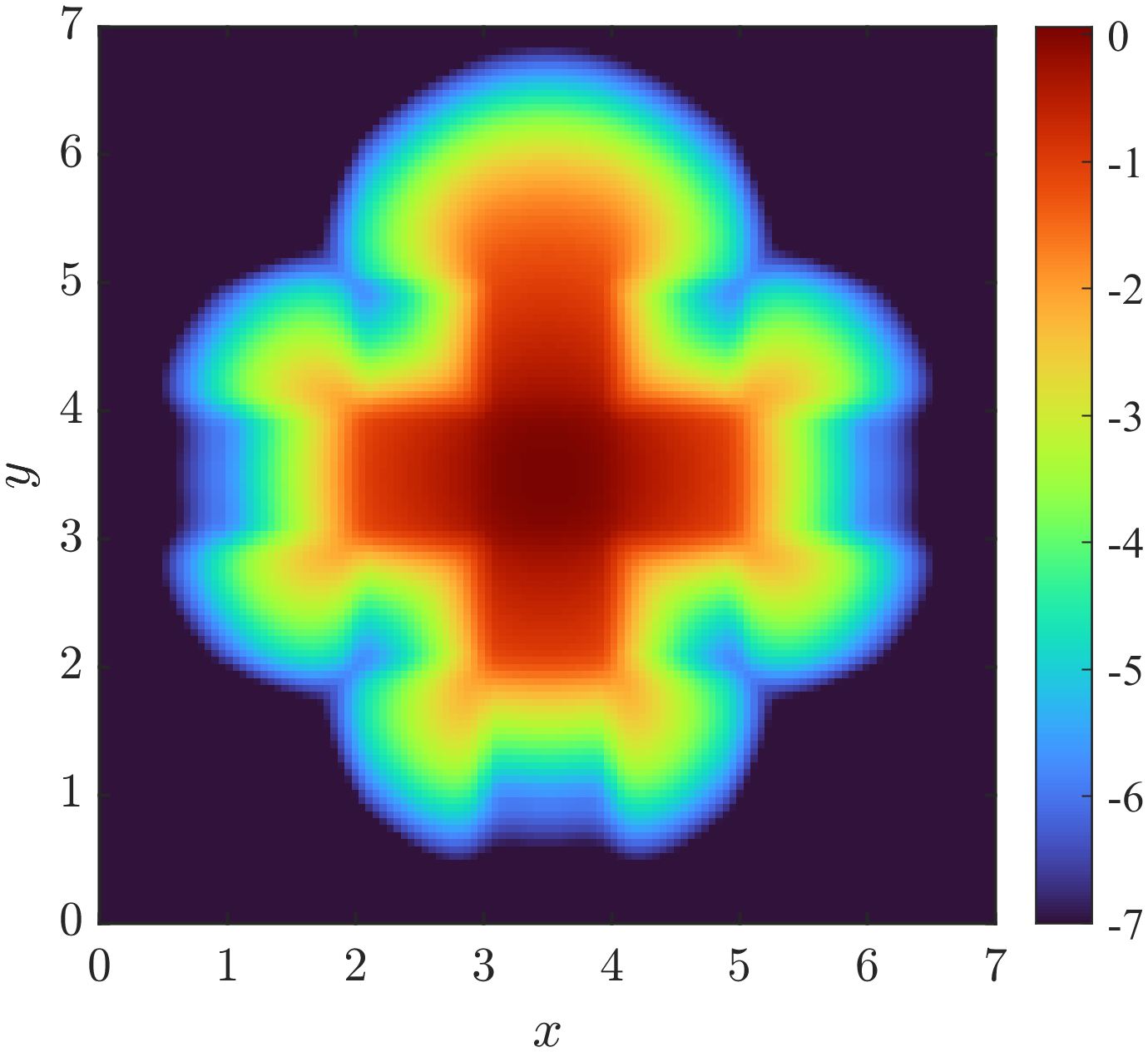}
    \end{subfigure}
    \begin{subfigure}[c]{0.3\textwidth}
        \centering
        \includegraphics[width=\textwidth]{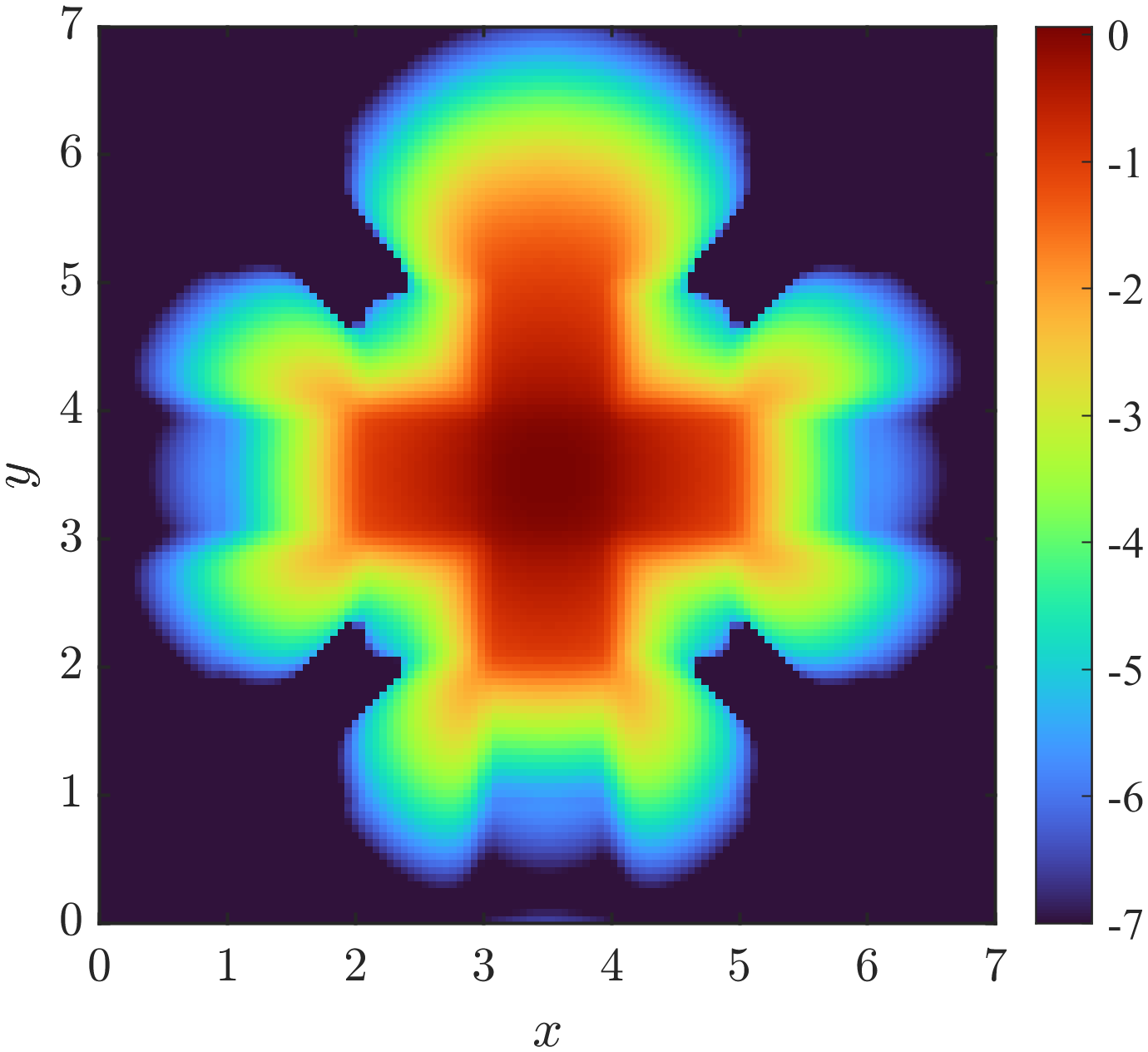}
    \end{subfigure}

    \caption{Lattice problem ($\varepsilon=1$).
        Contour plots of the log density on a $128 \times 128$ mesh at time $T = 2.0$.
        Left: IMEX-BUG $\Delta t = \Delta t_{(E)}$;
        Middle: IMEX-S-BUG $\Delta t = \Delta t_{(I)}$;
        Right: IMEX-S-BUG $\Delta t = 0.2 \Delta t_{(I)}$.
    } \label{fig:lattice_2d_base}
\end{figure}

\begin{figure}[htbp]
    \centering
    \begin{subfigure}[c]{0.3\textwidth}
        \centering
        \includegraphics[width=\textwidth]{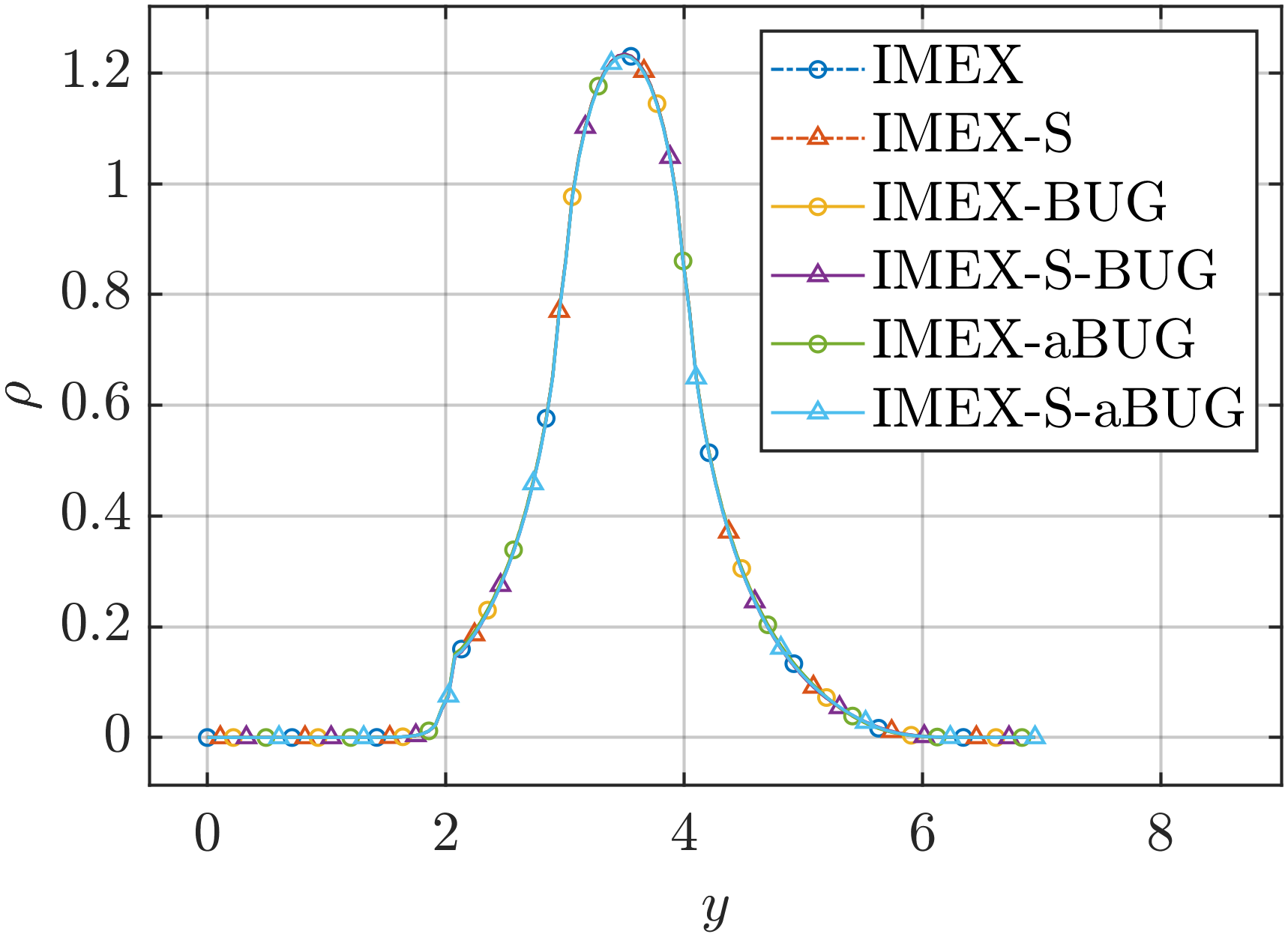}
    \end{subfigure}
    \begin{subfigure}[c]{0.3\textwidth}
        \centering
        \includegraphics[width=\textwidth]{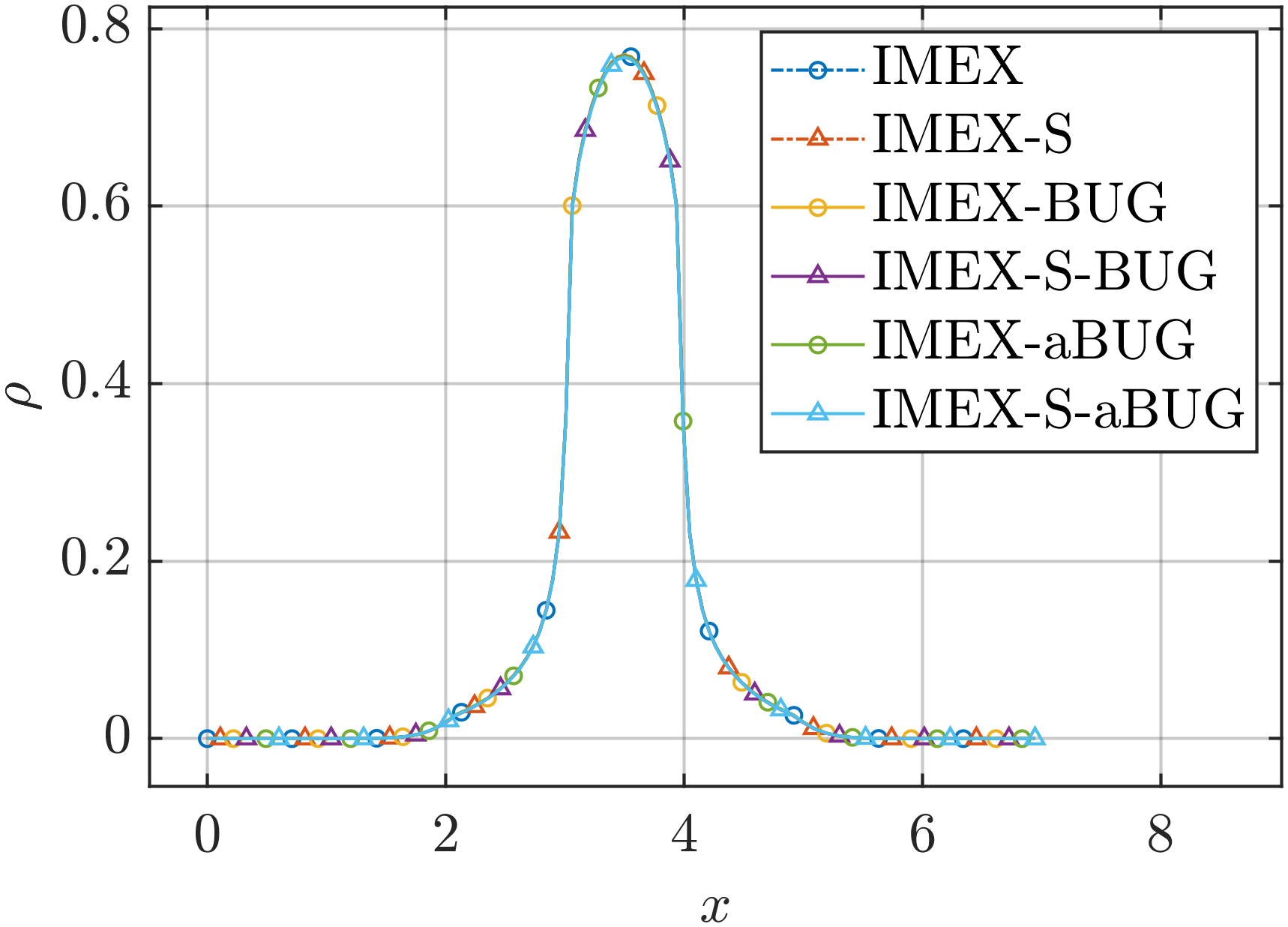}
    \end{subfigure}
    \begin{subfigure}[c]{0.3\textwidth}
        \centering
        \includegraphics[width=\textwidth]{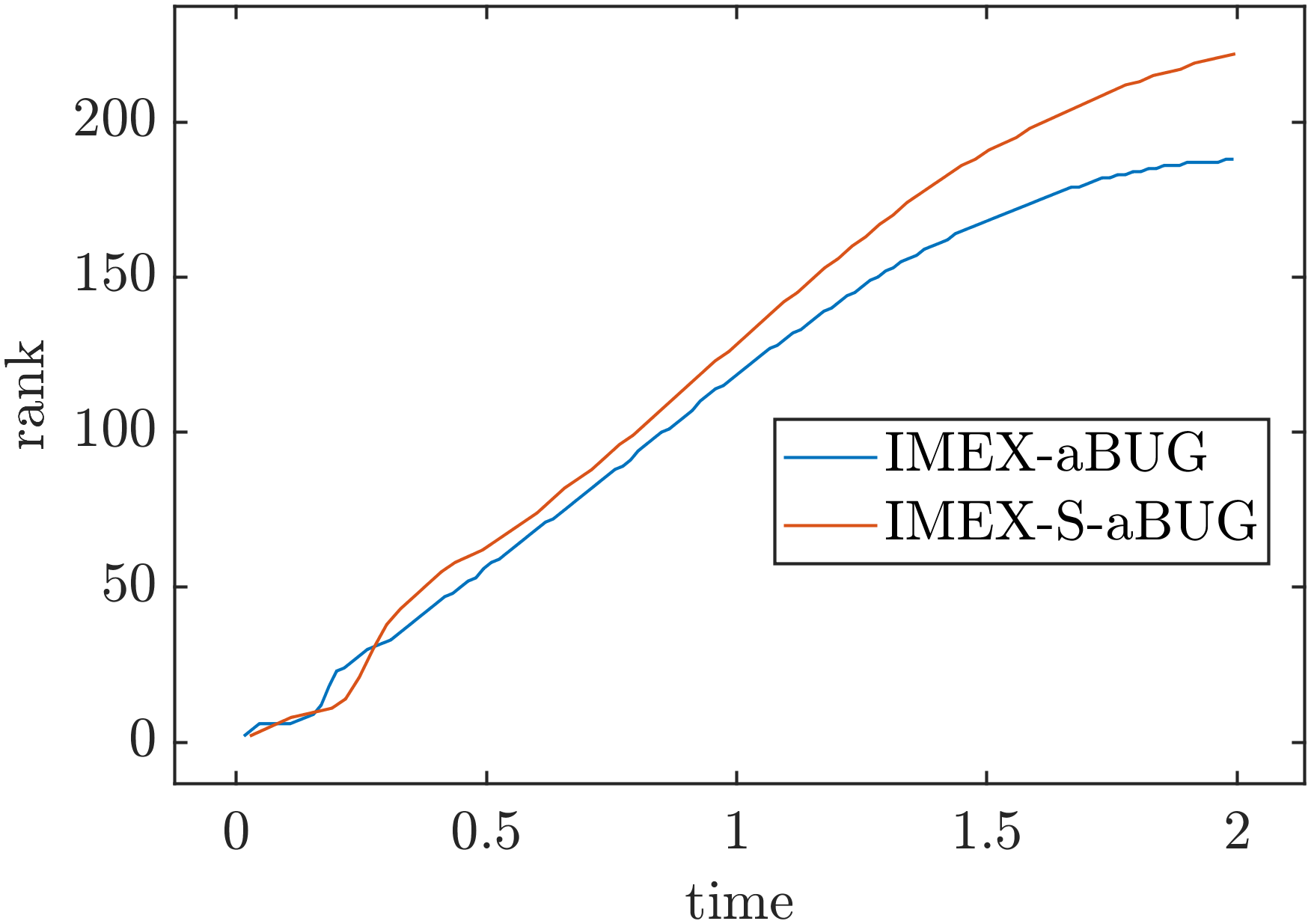}
    \end{subfigure}
    \caption{Lattice problem ($\varepsilon=1$).
        Density slice of solutions at time $T = 2.0$ along $x = 3.5$ (left) and $y = 4.047$ (middle), and the rank of the microscopic component $\bm{g}$ for the IMEX-S-BUG and IMEX-S-aBUG methods (right).
    } \label{fig:lattice_2d_more}
\end{figure}

\section{Conclusion} \label{sec:conclusion}

In this work, we have developed an asymptotic-preserving dynamical low-rank method for the multiscale linear transport equation.
By combining a macro-micro decomposition with a dynamical low-rank representation for the microscopic component and the IMEX-S temporal discretization, the scheme efficiently handles the high-dimensional phase space while remaining stable in both kinetic and diffusive regimes.
The application of the Schur complement allows the elimination of stiff microscopic terms, resulting in a reduced system for the macroscopic density and preserving the AP property without relying on weighting functions.
Additionally, the energy stability of the method is established by introducing a low-rank formulation consistent with the $S_N$ discrete energy.
Numerical experiments demonstrate that the method accurately captures the correct asymptotic limits, maintains energy stability, and significantly reduces both computational cost and memory requirements.
Future work will extend this framework to higher-order discretizations and other kinetic equations.

\bibliographystyle{plain}
\bibliography{references}

@article{appelo2025robust,
  author    = {Appel{\"o}, Daniel and Cheng, Yingda},
  title     = {Robust implicit adaptive low rank time-stepping methods for matrix differential equations},
  journal   = {Journal of Scientific Computing},
  year      = {2025},
  number    = {3},
  pages     = {81},
  publisher = {Springer},
  volume    = {102}
}

@article{bachmayr2023low,
  author    = {Bachmayr, Markus},
  title     = {{Low-rank tensor methods for partial differential equations}},
  journal   = {Acta Numerica},
  year      = {2023},
  pages     = {1--121},
  publisher = {Cambridge University Press},
  volume    = {32}
}

@article{ceruti2022rank,
  author    = {Ceruti, Gianluca and Kusch, Jonas and Lubich, Christian},
  title     = {{A rank-adaptive robust integrator for dynamical low-rank approximation}},
  journal   = {BIT Numerical Mathematics},
  year      = {2022},
  number    = {4},
  pages     = {1149--1174},
  publisher = {Springer},
  volume    = {62}
}

@article{ceruti2022unconventional,
  author    = {Ceruti, Gianluca and Lubich, Christian},
  title     = {{An unconventional robust integrator for dynamical low-rank approximation}},
  journal   = {BIT Numerical Mathematics},
  year      = {2022},
  number    = {1},
  pages     = {23--44},
  publisher = {Springer},
  volume    = {62}
}

@article{ceruti2024parallel,
  author    = {Ceruti, Gianluca and Kusch, Jonas and Lubich, Christian},
  title     = {A parallel rank-adaptive integrator for dynamical low-rank approximation},
  journal   = {SIAM Journal on Scientific Computing},
  year      = {2024},
  number    = {3},
  pages     = {B205--B228},
  publisher = {SIAM},
  volume    = {46}
}

@article{ceruti2024robust,
  author    = {Ceruti, Gianluca and Einkemmer, Lukas and Kusch, Jonas and Lubich, Christian},
  title     = {A robust second-order low-rank BUG integrator based on the midpoint rule},
  journal   = {BIT Numerical Mathematics},
  year      = {2024},
  number    = {3},
  pages     = {30},
  publisher = {Springer},
  volume    = {64}
}

@article{coughlin2024robust,
  author    = {Coughlin, Jack and Hu, Jingwei and Shumlak, Uri},
  title     = {Robust and conservative dynamical low-rank methods for the Vlasov equation via a novel macro-micro decomposition},
  journal   = {Journal of Computational Physics},
  year      = {2024},
  pages     = {113055},
  publisher = {Elsevier},
  volume    = {509}
}

@article{ding2021dynamical,
  author        = {Ding, Zhiyan and Einkemmer, Lukas and Li, Qin},
  title         = {Dynamical Low-Rank Integrator for the Linear {{Boltzmann}} Equation: Error Analysis in the Diffusion Limit},
  journal       = {SIAM Journal on Numerical Analysis},
  year          = 2021,
  archiveprefix = {arXiv},
  doi           = {10.1137/20M1380788},
  issn          = {0036-1429, 1095-7170},
  number        = {4},
  pages         = {2254--2285},
  shorttitle    = {Dynamical Low-Rank Integrator for the Linear {{Boltzmann}} Equation},
  volume        = {59}
}

@article{einkemmer2018lowrank,
  author  = {Einkemmer, Lukas and Lubich, Christian},
  title   = {A {{Low-Rank Projector-Splitting Integrator}} for the {{Vlasov--Poisson Equation}}},
  journal = {SIAM Journal on Scientific Computing},
  year    = 2018,
  doi     = {10.1137/18M116383X},
  issn    = {1064-8275, 1095-7197},
  number  = {5},
  pages   = {B1330-B1360},
  volume  = {40}
}

@article{einkemmer2021asymptoticpreserving,
  author  = {Einkemmer, Lukas and Hu, Jingwei and Wang, Yubo},
  title   = {An Asymptotic-Preserving Dynamical Low-Rank Method for the Multi-Scale Multi-Dimensional Linear Transport Equation},
  journal = {Journal of Computational Physics},
  year    = 2021,
  doi     = {10.1016/j.jcp.2021.110353},
  issn    = {00219991},
  pages   = {110353},
  volume  = {439}
}

@article{einkemmer2021efficient,
  author    = {Einkemmer, Lukas and Hu, Jingwei and Ying, Lexing},
  title     = {An {{Efficient Dynamical Low-Rank Algorithm}} for the {{Boltzmann-BGK Equation Close}} to the {{Compressible Viscous Flow Regime}}},
  journal   = {SIAM Journal on Scientific Computing},
  year      = 2021,
  doi       = {10.1137/21M1392772},
  issn      = {1064-8275},
  number    = {5},
  pages     = {B1057-B1080},
  publisher = {{Society for Industrial and Applied Mathematics}},
  volume    = {43}
}

@article{einkemmer2021mass,
  author  = {Einkemmer, Lukas and Joseph, Ilon},
  title   = {A Mass, Momentum, and Energy Conservative Dynamical Low-Rank Scheme for the {{Vlasov}} Equation},
  journal = {Journal of Computational Physics},
  year    = 2021,
  doi     = {10.1016/j.jcp.2021.110495},
  issn    = {00219991},
  pages   = {110495},
  volume  = {443}
}

@article{einkemmer2023conservation,
  author  = {Einkemmer, Lukas Daniel and Kusch, Jonas and Schotth{\"o}fer, Steffen},
  title   = {Conservation Properties of the Augmented Basis Update$\backslash$\& Galerkin Integrator for Kinetic Problems},
  journal = {Available at SSRN 4668132},
  year    = {2023}
}

@article{einkemmer2024asymptoticpreserving,
  author  = {Einkemmer, Lukas and Hu, Jingwei and Kusch, Jonas},
  title   = {Asymptotic-{{Preserving}} and {{Energy Stable Dynamical Low-Rank Approximation}}},
  journal = {SIAM Journal on Numerical Analysis},
  year    = 2024,
  doi     = {10.1137/23M1547603},
  issn    = {0036-1429, 1095-7170},
  number  = {1},
  pages   = {73--92},
  volume  = {62}
}

@article{einkemmer2025construction,
  author        = {Einkemmer, Lukas and Kusch, Jonas and Schotth{\"o}fer, Steffen},
  title         = {Construction of High-Order Conservative Basis-Update and {{Galerkin}} Dynamical Low-Rank Integrators},
  journal       = {arXiv preprint arXiv:2311.06399},
  year          = {2025},
  archiveprefix = {arXiv},
  eprint        = {2311.06399},
  primaryclass  = {math.NA},
  url           = {https://arxiv.org/abs/2311.06399}
}

@article{einkemmer2025review,
  author  = {Einkemmer, Lukas and Kormann, Katharina and Kusch, Jonas and McClarren, Ryan G. and Qiu, Jing-Mei},
  title   = {A Review of Low-Rank Methods for Time-Dependent Kinetic Simulations},
  journal = {Journal of Computational Physics},
  year    = 2025,
  doi     = {10.1016/j.jcp.2025.114191},
  issn    = {00219991},
  pages   = {114191},
  volume  = {538}
}

@book{ganapol2008analytical,
  author    = {Ganapol, Barry D.},
  title     = {Analytical Benchmarks for Nuclear Engineering Applications: Case Studies in Neutron Transport Theory},
  year      = {2008},
  address   = {Issy-les-Moulineaux, France},
  isbn      = {978-92-64-99056-2},
  note      = {NEA No. 6292},
  number    = {NEA/DB/DOC(2008)1},
  publisher = {Nuclear Energy Agency, Organisation for Economic Co-operation and Development},
  series    = {OECD/NEA Data Bank}
}

@article{grasedyck2013literature,
  author    = {Grasedyck, Lars and Kressner, Daniel and Tobler, Christine},
  title     = {{A literature survey of low-rank tensor approximation techniques}},
  journal   = {GAMM-Mitteilungen},
  year      = {2013},
  number    = {1},
  pages     = {53--78},
  publisher = {Wiley Online Library},
  volume    = {36}
}

@article{guo2024conservative,
  author    = {Guo, Wei and Qiu, Jing-Mei},
  title     = {A {{Conservative Low Rank Tensor Method}} for the {{Vlasov Dynamics}}},
  journal   = {SIAM Journal on Scientific Computing},
  year      = 2024,
  doi       = {10.1137/22M1473960},
  issn      = {1064-8275},
  number    = {1},
  pages     = {A232-A263},
  publisher = {{Society for Industrial and Applied Mathematics}},
  volume    = {46}
}

@article{guo2024local,
  author    = {Guo, Wei and Ema, Jannatul Ferdous and Qiu, Jing-Mei},
  title     = {A local macroscopic conservative ({LoMaC}) low rank tensor method with the discontinuous Galerkin method for the Vlasov dynamics},
  journal   = {Communications on Applied Mathematics and Computation},
  year      = {2024},
  number    = {1},
  pages     = {550--575},
  publisher = {Springer},
  volume    = {6}
}

@book{hackbusch2012tensor,
  author    = {Hackbusch, Wolfgang},
  title     = {{Tensor Spaces and Numerical Tensor Calculus}},
  year      = {2012},
  publisher = {Springer},
  volume    = {42}
}

@incollection{hu2017asymptotic,
  author    = {Hu, Jingwei and Jin, Shi and Li, Qin},
  title     = {Asymptotic-preserving schemes for multiscale hyperbolic and kinetic equations},
  year      = {2017},
  booktitle = {Handbook of Numerical Analysis},
  pages     = {103--129},
  publisher = {Elsevier},
  volume    = {18}
}

@article{jin1999efficient,
  author    = {Jin, Shi},
  title     = {{Efficient asymptotic-preserving (AP) schemes for some multiscale kinetic equations}},
  journal   = {SIAM Journal on Scientific Computing},
  year      = {1999},
  number    = {2},
  pages     = {441--454},
  publisher = {SIAM},
  volume    = {21}
}

@article{jin2010asymptotic,
  author    = {Jin, Shi},
  title     = {{Asymptotic preserving (AP) schemes for multiscale kinetic and hyperbolic equations: a review}},
  journal   = {Lecture notes for summer school on methods and models of kinetic theory (M\&MKT), Porto Ercole (Grosseto, Italy)},
  year      = {2010},
  pages     = {177--216},
  publisher = {Citeseer}
}

@article{jin2022asymptotic,
  author    = {Jin, Shi},
  title     = {Asymptotic-preserving schemes for multiscale physical problems},
  journal   = {Acta Numerica},
  year      = {2022},
  pages     = {415--489},
  publisher = {Cambridge University Press},
  volume    = {31}
}

@book{khoromskij2018tensor,
  author    = {Khoromskij, Boris N},
  title     = {{Tensor Numerical Methods in Scientific Computing}},
  year      = {2018},
  publisher = {Walter de Gruyter GmbH \& Co KG},
  volume    = {19}
}

@article{koch2007dynamical,
  author    = {Koch, Othmar and Lubich, Christian},
  title     = {{Dynamical low-rank approximation}},
  journal   = {SIAM Journal on Matrix Analysis and Applications},
  year      = {2007},
  number    = {2},
  pages     = {434--454},
  publisher = {SIAM},
  volume    = {29}
}

@article{kupper2016asymptotic,
  author    = {Kupper, Kerstin and Frank, Martin and Jin, Shi},
  title     = {An asymptotic preserving two-dimensional staggered grid method for multiscale transport equations},
  journal   = {SIAM Journal on Numerical Analysis},
  year      = {2016},
  number    = {1},
  pages     = {440--461},
  publisher = {SIAM},
  volume    = {54}
}

@article{lemou2008new,
  author    = {Lemou, Mohammed and Mieussens, Luc},
  title     = {A new asymptotic preserving scheme based on micro-macro formulation for linear kinetic equations in the diffusion limit},
  journal   = {SIAM Journal on Scientific Computing},
  year      = {2008},
  number    = {1},
  pages     = {334--368},
  publisher = {SIAM},
  volume    = {31}
}

@book{lewis1983computational,
  author    = {Lewis, Elmer Eugene and Miller, Warren F},
  title     = {{Computational Methods of Neutron Transport}},
  year      = {1983},
  publisher = {John Wiley and Sons, Inc., New York, NY}
}

@article{liu2004boltzmann,
  author    = {Liu, Tai-Ping and Yu, Shih-Hsien},
  title     = {Boltzmann equation: micro-macro decompositions and positivity of shock profiles},
  journal   = {Communications in mathematical physics},
  year      = {2004},
  number    = {1},
  pages     = {133--179},
  publisher = {Springer},
  volume    = {246}
}

@article{liu2010analysis,
  author  = {Liu, Jian-Guo and Mieussens, Luc},
  title   = {Analysis of an {{Asymptotic Preserving Scheme}} for {{Linear Kinetic Equations}} in the {{Diffusion Limit}}},
  journal = {SIAM Journal on Numerical Analysis},
  year    = 2010,
  doi     = {10.1137/090772770},
  issn    = {0036-1429, 1095-7170},
  number  = {4},
  pages   = {1474--1491},
  volume  = {48}
}

@article{patwardhan2025asymptoticpreserving,
  author        = {Patwardhan, Chinmay and Frank, Martin and Kusch, Jonas},
  title         = {Asymptotic-Preserving and Energy Stable Dynamical Low-Rank Approximation for Thermal Radiative Transfer Equations},
  journal       = {Multiscale Modeling \& Simulation},
  year          = 2025,
  archiveprefix = {arXiv},
  doi           = {10.1137/24M1646303},
  issn          = {1540-3459, 1540-3467},
  number        = {1},
  pages         = {278--312},
  volume        = {23}
}

@article{patwardhan2025parallel,
  author  = {Patwardhan, Chinmay and Kusch, Jonas},
  title   = {A {{Parallel}}, {{Energy-Stable Low-Rank Integrator}} for {{Nonlinear Multi-Scale Thermal Radiative Transfer}}},
  journal = {arXiv preprint arXiv:2502.20883},
  year    = {2025}
}

@article{peng2020low,
  author    = {Peng, Zhuogang and McClarren, Ryan G and Frank, Martin},
  title     = {A low-rank method for two-dimensional time-dependent radiation transport calculations},
  journal   = {Journal of Computational Physics},
  year      = {2020},
  pages     = {109735},
  publisher = {Elsevier},
  volume    = {421}
}

@article{peng2020stabilityenhanced,
  author  = {Peng, Zhichao and Cheng, Yingda and Qiu, Jing-Mei and Li, Fengyan},
  title   = {Stability-Enhanced {{AP IMEX-LDG}} Schemes for Linear Kinetic Transport Equations under a Diffusive Scaling},
  journal = {Journal of Computational Physics},
  year    = 2020,
  doi     = {10.1016/j.jcp.2020.109485},
  issn    = {00219991},
  pages   = {109485},
  volume  = {415}
}

@article{peng2021asymptotic,
  author  = {Peng, Zhichao and Li, Fengyan},
  title   = {Asymptotic {{Preserving IMEX-DG-S Schemes}} for {{Linear Kinetic Transport Equations Based}} on {{Schur Complement}}},
  journal = {SIAM Journal on Scientific Computing},
  year    = 2021,
  doi     = {10.1137/20M134486X},
  issn    = {1064-8275, 1095-7197},
  number  = {2},
  pages   = {A1194-A1220},
  volume  = {43}
}

@article{peng2021highorder,
  author  = {Peng, Zhuogang and McClarren, Ryan G.},
  title   = {A High-Order/Low-Order ({{HOLO}}) Algorithm for Preserving Conservation in Time-Dependent Low-Rank Transport Calculations},
  journal = {Journal of Computational Physics},
  year    = 2021,
  doi     = {10.1016/j.jcp.2021.110672},
  issn    = {00219991},
  pages   = {110672},
  volume  = {447}
}

@article{sands2025high,
  author    = {Sands, William A and Guo, Wei and Qiu, Jing-Mei and Xiong, Tao},
  title     = {High-order adaptive rank integrators for multiscale linear kinetic transport equations in the hierarchical tucker format},
  journal   = {SIAM Journal on Scientific Computing},
  year      = {2025},
  number    = {6},
  pages     = {A3383--A3412},
  publisher = {SIAM},
  volume    = {47}
}

@book{zhang2006schur,
  author    = {Zhang, Fuzhen},
  title     = {The Schur complement and its applications},
  year      = {2006},
  publisher = {Springer Science \& Business Media},
  volume    = {4}
}


\end{document}